\documentclass[11pt,reqno]{amsart}
\usepackage[margin=25.4mm]{geometry}
\usepackage{amsmath,amssymb,amsthm,mathtools}
\usepackage{bm,microtype,booktabs,array,enumitem,xcolor}
\definecolor{crossrefcolor}{RGB}{0,105,135}
\usepackage[colorlinks=true,linkcolor=crossrefcolor,citecolor=black,urlcolor=crossrefcolor]{hyperref}
\newcommand{\refnum}[1]{\textcolor{crossrefcolor}{#1}}
\usepackage[nameinlink,capitalize,noabbrev]{cleveref}
\allowdisplaybreaks
\numberwithin{equation}{section}

\newtheorem{theorem}{Theorem}[section]
\newtheorem{proposition}[theorem]{Proposition}
\newtheorem{lemma}[theorem]{Lemma}
\newtheorem{corollary}[theorem]{Corollary}

\theoremstyle{definition}
\theoremstyle{remark}\newtheorem{remark}[theorem]{Remark}

\newcommand{\R}{\mathbb R}

\newcommand{\Per}{\operatorname{Per}}

\newcommand{\weakstar}{\mathrel{\stackrel{*}{\rightharpoonup}}}

\newcommand{\Id}{\mathrm{Id}}

\title[Isolated singularities of the capillary equation]
{Isolated singularities of the capillary equation\\
with negative gravity}
\author{Bin Deng \and Jiahuan Li \and Yilu Liu \and Xi-Nan Ma}
\date{September 2026}

\begin{document}
\begin{abstract}
We classify isolated singularities of classical solutions to the capillary
equation with negative gravity,
\[
 \operatorname{div}\frac{Du}{\sqrt{1+|Du|^2}}=-u
 \qquad\text{in }B_R\setminus\{0\}\subset\mathbb R^n,
 \qquad n\ge2,
\]
without a priori assumptions on symmetry, sign, one-sided boundedness,
or blow-up rate.
Every such solution is either bounded near the puncture or tends uniformly to
$+\infty$ or $-\infty$.  In the bounded case, the solution extends across
the puncture as a $W^{1,1}_{\mathrm{loc}}$ distributional solution; this
extension has a unique continuous representative when $2\le n\le7$.
In the unbounded case, for some $\varepsilon\in\{-1,1\}$,
\[
 u(x)=\varepsilon\left(\frac{n-1}{|x|}
       -\frac{n+3}{2(n-1)^2}|x|^3\right)+O(|x|^5),
\]
uniformly in the angular variable.  In every fixed smaller ball, all
sufficiently high level sets of $\varepsilon u$ are connected, smooth,
strictly convex hypersurfaces enclosing the puncture.  The corresponding
pressure-rescaled graphs converge smoothly with multiplicity one to the
round cylinder $\partial B_{n-1}\times\mathbb R$.
Nonradial examples form infinite-dimensional families that agree
with the radial pole to every algebraic order, so the complete asymptotic
expansion does not determine the singular solution germ.
The proof of the classification combines
critical tail estimates and logarithmic $BV$ compactness with level-set
rigidity and the translation identities of the capillary equation.
The critical two-dimensional case requires an additional finite-height
analysis to exclude a translation defect.
\end{abstract}
\maketitle

\section{Introduction}
\label{sec:intro}

\subsection{Capillary surfaces and the gravity sign}

A graphical liquid interface in three-dimensional space is described
by a height function of two horizontal variables.  At
equilibrium, the Young--Laplace law balances surface tension against the
pressure difference across the interface; hydrostatic pressure makes this
difference affine in the height.  The resulting graphical equation is
\begin{equation}\label{eq:general-capillary-equation}
 \operatorname{div}\mathcal A(Du)=\kappa u+\lambda,
 \qquad \mathcal A(p):=\frac{p}{\sqrt{1+|p|^2}},
\end{equation}
where $\kappa$ is the signed capillary constant and $\lambda$ accounts for
the reference pressure.  The physical theory originates in the work of
Young, Laplace, and Gauss \cite{Young1805,Laplace1806,Gauss1830}; for its
mathematical development, see Finn \cite{Finn1986,Finn1999}.
When $\kappa\ne0$, a vertical translation removes $\lambda$, and scaling
both horizontal and vertical lengths by the capillary length reduces
$\kappa$ to $+1$ or $-1$.

In the convention of \eqref{eq:general-capillary-equation}, positive gravity describes, for
example, the upper surface of a sessile drop, whereas negative gravity
is associated with the lower interface of a pendent drop.  Both arise
under ordinary gravity \cite{Finn1986,Nickolov2002}.  The symmetry
$u\mapsto-u$ preserves each normalized equation; it does not interchange
the two signs.  We study the negative-gravity equation
\begin{equation}\label{eq:pde}
 u\in C^2(B_R\setminus\{0\}),\qquad
 \operatorname{div}\frac{Du}{\sqrt{1+|Du|^2}}=-u,
 \qquad B_R\subset\mathbb R^n,\quad n\ge2.
\end{equation}
Here $n$ is the dimension of the base, so $n=2$ describes a liquid
interface in $\mathbb R^3$.  We ask whether every unbounded isolated end
has the asymptotic shape of the classical singular pendent drop, without
assuming axial symmetry.

We prove that every solution is either bounded near the puncture or
tends uniformly to $+\infty$ or $-\infty$.  In the physical case $n=2$,
the bounded branch has a continuous weak extension, while every
unbounded branch has the classical two-term pole expansion.
After changing the sign if necessary, its sufficiently high level curves
are single, smooth, strictly convex curves that become circular after
rescaling.  No sign, one-sided bound, blow-up rate,
convexity, or level-set topology is assumed.  The asymptotic expansion does not determine the solution germ:
nonradial examples agree with the radial pole to every algebraic order;
see Remark~\ref{rem:nonradial-germs}.

Our interior problem imposes no supporting surface, contact angle, or
stability condition.  Boundary-value problems prescribe the height
(Dirichlet data), its normal derivative (Neumann data), or the normalized
flux $\mathcal A(Du)\cdot\nu$ (contact-angle data), where $\nu$ is the
outer unit normal to the base domain.  Serrin \cite{Serrin1969} related
Dirichlet solvability to domain geometry and prescribed mean curvature.  Spruck
\cite{Spruck1975}, Simon--Spruck \cite{SimonSpruck1976}, and Gerhardt
\cite{Gerhardt1976} developed existence and boundary regularity theory
for prescribed contact angle.  Giusti \cite{Giusti1976} treated
Dirichlet, capillarity, and mixed boundary conditions through a
variational formulation in $BV$; see also Finn \cite{Finn1986}.

Within this framework, Ma \cite{Ma2000} obtained size estimates for
zero-gravity capillary surfaces over planar convex domains with constant
contact angle, and Ma--Xu
\cite{MaXu2016} proved boundary gradient estimates for prescribed mean
curvature equations with Neumann data.  Related developments include
mean curvature flow with prescribed contact angle
\cite{GaoMaWangWeng2021} and gradient estimates for higher-order curvature
equations with contact-angle conditions \cite{DengMa2023}.  Loss of uniform ellipticity at large gradients makes boundary gradient
estimates central to this theory.  At our interior puncture, neither
height nor gradient is assumed bounded.

\subsection{Removability and the singular pendent drop}

Classical results on isolated singularities differ according to the
gravity sign.  In the minimal surface case, $\kappa=\lambda=0$, Bers
\cite{Bers1951} and Finn \cite{Finn1953} proved that isolated
singularities of planar graphs are removable.  The contrast with the
poles and multipoles of harmonic functions shows that the nonlinear
geometric equation can exclude singular behavior admitted by its
linearization.  Serrin developed removability and local singularity
theory for wider classes of elliptic equations
\cite{Serrin1964,Serrin1965,Serrin1965Removable}.

For \emph{positive gravity}, $\kappa=1$, Finn \cite{Finn1961} proved
that every isolated interior singularity is removable, in every
dimension $n\ge2$, without a symmetry, boundedness, or growth assumption
at the puncture.  This result belongs to a broader removability theory
for elliptic equations with bounded flux and nondecreasing dependence
on the height; see Finn's account in
\cite[pp.~156--157]{Finn1976}.  The capillary flux satisfies
$|\mathcal A(Du)|<1$, and the height term $\kappa u$ is increasing
when $\kappa>0$.  The negative-gravity equation loses this monotonicity.
Later, Finn--Lu \cite{FinnLu1998} obtained interior gradient bounds
and a Harnack inequality for planar graphs with nonconstant,
nondecreasing prescribed mean curvature $H(u)$, without assuming
positivity of the solution.  Their results include the positive-gravity
case $H(u)=u/2$.

For \emph{negative gravity}, $\kappa=-1$, isolated poles exist.
Concus and Finn \cite{ConcusFinnGravity} proposed a singular radial limit
of regular axisymmetric solutions.  They then proved its existence and
derived its asymptotic expansion \cite{ConcusFinn1975I}, and established
uniqueness within the radial class under an additional asymptotic
condition \cite{ConcusFinn1975II}.  In dimension two, the first terms of
the positive branch are $r^{-1}-\tfrac52r^3$ in our normalization
\cite[p.~128]{Nickolov2002}.  Their work on pendent drops
\cite{ConcusFinn1979} also related this singular profile to limits of
finite-length drops.  The pole thus arises within the classical
equilibrium theory of pendent drops as an end of infinite length.

Bidaut--V\'eron proved global continuation and strengthened the radial
uniqueness criteria \cite{BidautVeron1986}, and later studied more
general height-dependent prescribed mean curvature
\cite{BidautVeron1996}.  Nickolov \cite{Nickolov2002} completed the
radial uniqueness proof by showing that every radial nonremovable
singularity satisfies the required criterion.  Riera and Risler
\cite{RieraRisler2002} developed a dynamical-system description of
axisymmetric capillary surfaces, and Risler \cite{Risler2015} used it
to give a shorter uniqueness proof.  These results determine the
radial pole; they do not classify arbitrary solutions of the PDE.

Using constant-mean-curvature comparison surfaces, Finn
\cite[Theorems~\refnum{1 and~2}]{Finn1976} obtained a boundedness
criterion from the values on a single surrounding sphere, without
assuming symmetry.  For $n=2$, it states that
\[
 \sup_{|x|=r_0}|u(x)|\le \frac1{r_0}-\frac{\pi}{2}r_0
 \quad\text{for some }0<r_0<R
\]
implies boundedness near the origin.  In every dimension, a consequence
is that $|u(x)|=o(|x|^{-1})$ implies boundedness; see
\cite[p.~159, Remark~\refnum{iii}]{Finn1976}.  These comparison results
restrict nonradial growth but do not determine the asymptotic profile.

Finn also asked whether bounded isolated singularities of the exact
negative-gravity equation are classically removable
\cite[pp.~157--158 and Sections~\refnum{6--7}]{Finn1976}.  His Section~\refnum{6} example
has bounded height, bounded mean curvature, and finite area, but is
discontinuous.  The H\"older-continuous example in Section~\refnum{7} has unbounded
gradient and satisfies
$\operatorname{div}\mathcal A(Dz)=-z+f(x,y)$ with H\"older-continuous
forcing.  Neither is a counterexample for the unforced equation
\eqref{eq:pde}.  Our bounded-branch conclusion concerns continuous weak
extension in the stated dimensions; classical smooth removability is
a separate question.

The remaining classification problem is already present in dimension
two.  Without axial symmetry, a punctured solution could have different
limits in different directions, both positive and negative tails, or
several high-level components.  Global symmetry theorems for bounded
drops, such as Wente's \cite{Wente1980Symmetry}, use supporting-plane and
boundary hypotheses that are absent here.  We exclude these alternatives directly from \eqref{eq:pde} and prove
that every unbounded punctured solution has the classical two-term
expansion and a single round asymptotic end.
Remark~\ref{rem:nonradial-germs} shows that even the complete algebraic
expansion does not determine the solution germ.

\subsection{Main theorem}

Here $B_\rho=B_\rho(0)\subset\mathbb R^n$, $\partial^*$ denotes reduced
boundary, and $\partial_U$ denotes topological boundary relative to $U$.

\begin{theorem}\label{thm:main}
Let $n\ge2$ and let $u$ satisfy \eqref{eq:pde}.  Exactly one of the
following alternatives occurs.
\begin{enumerate}[label=\textnormal{(\roman*)},leftmargin=3.2em]
\item\label{item:removable-main}
The function $u$ is bounded in a punctured neighborhood of the origin and
has a unique $W^{1,1}_{\mathrm{loc}}(B_R)$ weak extension
$\widetilde u$ across the puncture, unique up to equality almost everywhere,
which satisfies \eqref{eq:pde} distributionally on $B_R$.  If
$2\le n\le7$, then $u$ admits a unique continuous extension to $B_R$;
this continuous extension represents the preceding weak
$W^{1,1}_{\mathrm{loc}}(B_R)$ solution class.

\item\label{item:positive-main}
One has $u(x)\to+\infty$ uniformly as $x\to0$.  For every fixed
$0<r<R$ there exists $T_r<\infty$ such that, for every $t\ge T_r$, there
is a smooth strictly convex body $K_{t,r}\Subset B_r$, with
$0\in\operatorname{int}K_{t,r}$, for which
\begin{equation}\label{eq:localized-positive-level-main}
 E_{t,r}:=\{x\in B_r\setminus\{0\}:u(x)>t\}
 =\operatorname{int}K_{t,r}\setminus\{0\}.
\end{equation}
Moreover,
\begin{equation}\label{eq:localized-positive-boundary-main}
 \partial_{B_r\setminus\{0\}}E_{t,r}
 =\partial^*E_{t,r}=\partial K_{t,r},
\end{equation}
and
\begin{equation}\label{eq:positive-expansion-main}
 u(x)=\frac{n-1}{|x|}
 -\frac{n+3}{2(n-1)^2}|x|^3+O(|x|^5).
\end{equation}
For every $L>0$, the pressure-rescaled graph pieces
\[
 \left\{\bigl(tx,t(u(x)-t)\bigr):
  x\in B_r\setminus\{0\},\ |u(x)-t|<L/t\right\}
\]
converge smoothly with multiplicity one on every smaller vertical slab to
the corresponding portion of $\partial B_{n-1}\times\mathbb R$.

\item\label{item:negative-main}
One has $u(x)\to-\infty$ uniformly as $x\to0$.  For every fixed
$0<r<R$ there exists $T_r<\infty$ such that, for every $t\ge T_r$, there
is a smooth strictly convex body $K^-_{t,r}\Subset B_r$, with
$0\in\operatorname{int}K^-_{t,r}$, for which
\begin{equation}\label{eq:localized-negative-level-main}
 F_{t,r}:=\{x\in B_r\setminus\{0\}:u(x)<-t\}
 =\operatorname{int}K^-_{t,r}\setminus\{0\}.
\end{equation}
Moreover,
\begin{equation}\label{eq:localized-negative-boundary-main}
 \partial_{B_r\setminus\{0\}}F_{t,r}
 =\partial^*F_{t,r}=\partial K^-_{t,r},
\end{equation}
and
\begin{equation}\label{eq:negative-expansion-main}
 u(x)=-\frac{n-1}{|x|}
 +\frac{n+3}{2(n-1)^2}|x|^3+O(|x|^5).
\end{equation}
The corresponding negative pressure graphs have the sign-reversed smooth
cylindrical description.
\end{enumerate}
The remainders in \eqref{eq:positive-expansion-main} and
\eqref{eq:negative-expansion-main} are uniform in the angular variable.
\end{theorem}

For $n=2$, the theorem gives a finite limit and a continuous weak
extension, or the expansion
\[
 u(x)=\varepsilon\left(\frac1{|x|}-\frac52|x|^3\right)
       +O(|x|^5),\qquad \varepsilon\in\{-1,1\}.
\]
In the latter case, every sufficiently high level of $\varepsilon u$
in a fixed smaller disk is one strictly convex closed curve enclosing
the puncture.  The pressure-rescaled surface converges smoothly with
multiplicity one to $\mathbb S^1\times\mathbb R$.

\begin{remark}\label{rem:bounded-regularity}
The $W^{1,1}_{\mathrm{loc}}$ distributional extension exists in every
dimension $n\ge2$; Proposition~\ref{prop:bounded-continuous-lowdim}
provides its continuous representative when $2\le n\le7$.
For $n\ge8$ no continuity assertion is made here.  Classical smooth removability is not asserted.
\end{remark}

\begin{remark}\label{rem:local-levels}
All level sets are localized in a fixed $B_r\Subset B_R$, and the
threshold $T_r$ may depend on this ball.  The conclusion concerns sufficiently high levels within $B_r$.
\end{remark}

\begin{remark}\label{rem:planar-endpoint-intro}
The physical dimension $n=2$ is also critical for the proof.
The obstruction is a possible horizontal translation defect at the
puncture; its exclusion is described below.
\end{remark}

\begin{remark}[Nonradial singular germs]\label{rem:nonradial-germs}
Theorem~\ref{thm:main} gives an asymptotic classification.  Even the
complete algebraic expansion of the positive radial pole does not
determine the singular solution germ.

More precisely, fix $n\ge2$ and
$0<c<(n-1)\sqrt{n+1}/2$.  There exist $R>0$, $a_*>0$, and classical
solutions $u_a$, $|a|<a_*$, on $B_R\setminus\{0\}$ satisfying
\[
 \operatorname{div}\frac{Du_a}{\sqrt{1+|Du_a|^2}}=-u_a,
 \qquad u_a(x)\longrightarrow+\infty\quad\text{as }x\to0,
\]
uniformly in direction.  If $U_n(|x|)$ denotes the classical positive
radial pole, then $u_0(x)=U_n(|x|)$ and
\[
 u_a(x)-U_n(|x|)=O_a\!\left(e^{-c/|x|^2}\right)
 \qquad\text{as }x\to0.
\]
For $a\ne0$, the solution $u_a$ is nonradial in every punctured
neighborhood of the origin, and distinct parameter values give distinct
solution germs.  In particular,
\[
 u_a(x)-U_n(|x|)=O_a(|x|^N)\qquad\text{for every }N>0,
\]
so all members of the family have the same expansion to every algebraic
order.

The solutions may be chosen centrally symmetric, with all sufficiently
high level sets smooth and strictly convex.  More generally, small even,
mean-zero angular data give an infinite-dimensional family of such
nonradial singular germs.  Thus the universal two-term expansion and
round asymptotic shape in Theorem~\ref{thm:main} coexist with
nonuniqueness of the singular solution germ.
\end{remark}

\subsection{Geometric singularities and asymptotic models}

Simon \cite{SimonAsymptotics1983} developed an asymptotic theory for
isolated singularities, and
Hardt--Simon \cite{HardtSimon1985} studied minimizing isolated
singularities and the minimizing foliations associated with their
cone models.  Edelen--Spolaor \cite{EdelenSpolaor2023} analyzed the
local structure near minimizing quadratic cones.  Constructions of
nonconical minimal hypersurfaces \cite{CaffarelliHardtSimon1984} and
of hypersurfaces with multiple isolated singularities \cite{Smale1989}
show the range of possible geometries.  For suitable smooth ambient
metrics, Simon \cite{SimonSingularSets2023} constructed strictly stable
minimal hypersurfaces with prescribed closed singular sets.
These results concern singularities at finite ambient
points.  In the present problem, the bounded branch concerns regularity
of a completed subgraph boundary, whereas a pole is an end at infinite
height.  Cylindrical limits also arise in the mean curvature flow work
of Colding--Minicozzi \cite{ColdingMinicozzi2015,ColdingMinicozzi2018};
here the equation is static, and the level sets need not evolve by
mean curvature.

The theory of integral currents \cite{FedererFleming1960} and the
relation between generalized graphs and sets of locally finite perimeter
\cite{Miranda1964} allow us to treat subgraph boundaries without
controlling the height gradient.  Geometric measure theory also enters capillary boundary
regularity directly; see Simon \cite{SimonCapillary1980}.  We combine
finite-perimeter compactness, the weak Alexandrov theorem of
Delgadino--Maggi \cite{DelgadinoMaggi2019}, and Allard regularity.
The key task is to derive compactness before the sign or growth rate
is known and to retain the exact translation identities in the limit.
Those identities select a single end centered at the puncture.

\subsection{The critical two-dimensional problem}

At large gradients the equation loses uniform ellipticity.  The
rescaling suggested by the pole is
\[
 w_t(y)=t^{-1}u(y/t),
 \qquad
 \operatorname{div}\frac{Dw_t}{\sqrt{t^{-4}+|Dw_t|^2}}=-w_t.
\]
Its formal limit is $\operatorname{div}(Dw/|Dw|)=-w$ where $Dw\ne0$.
This limit alone does not select a centered radial pole.  In dimension
two, the competing dipole and half-dipole profiles are
\begin{equation}\label{eq:dipoles}
 D_e(x):=\frac{2e\cdot x}{|x|^2},
 \qquad
 H_e(x):=\frac{2(e\cdot x)_+}{|x|^2},
 \qquad e\in\mathbb S^1.
\end{equation}
The full dipole $D_e$ satisfies the formal equation away from the
puncture, and $H_e=(D_e)_+$.  Their positive superlevel sets are disks
tangent to the puncture.  They are candidates for the rescaled limit,
not solutions of the original capillary equation.

To describe the competing behaviors, define the extended-real cluster
set
\[
 \mathcal C_0(u):=
 \left\{a\in[-\infty,+\infty]:
   \text{there are }x_j\to0\text{ with }u(x_j)\to a\right\}.
\]
This set is an interval because punctured balls are connected.
A half-line with a finite endpoint leads, up to sign, to the
half-dipole candidate; the full extended line leads to the dipole.
To exclude them we use horizontal translation balance for the exact
equation.  On a fixed finite-height band, a puncture cutoff at radius
$\rho$ produces an error of order $O(\rho^{n-2})$.  The error vanishes
for $n\ge3$, while in dimension two it may leave a nonzero force.
This is the critical obstruction.

To describe the residual force on the vertical axis, write
\[
 \Sigma:=\{(x,u(x)):0<|x|<R\},
 \qquad
 E:=\{(x,z):0<|x|<R,\ z<u(x)\}.
\]
The value of $\mathbf1_E$ on the vertical axis is immaterial, since
the axis has zero ambient Lebesgue measure.  In the exceptional planar
cases, we complete the subgraph boundary as an integral current and
track the axis-supported part of its first variation.  Shrinking-height
area estimates and ambient blow-ups give multiplicity-one vertical
planes at the relevant finite heights, forcing the axis-supported part
of the first variation to vanish.  The horizontal component of this
defect is the distributional height derivative of the horizontal level
flux.  The resulting zero flux fixes the centers of the limiting level
circles, excluding the dipole and half-dipole, whose centers at level
$s>0$ are $e/s$.  Thus finite-height translation balance for the exact
equation excludes profiles allowed by the formal limit.

\subsection{Outline of the proof}

The proof proceeds from weak compactness to a centered geometric
limit, then to smoothness and the two-term expansion.

\smallskip
\noindent\emph{Tail estimates and compactness.}
Finite-height stress estimates control the graph area across the
puncture.  Level-set flux and dilation identities then give, for each
unbounded sign $\varepsilon\in\{-1,1\}$ and a fixed ball
$B_r\Subset B_R$,
\[
 \lim_{t\to\infty}t^n
 \bigl|\{x\in B_r\setminus\{0\}:\varepsilon u(x)>t\}\bigr|
 =C_\varepsilon>0.
\]
The logarithmic truncations $\log(\varepsilon w_t/a)$ on
$\{\varepsilon w_t>a\}$, with fixed $a>0$, extended by zero elsewhere,
have bounded variation and controlled integrability.
A profile decomposition and an exact balance between total variation
and the nonlinear source exclude residual variation and yield strict
$BV$ convergence.  The weak formulations use finite-perimeter theory
\cite{Giusti1984,AFP2000,Maggi2012}, Anzellotti's pairing
\cite{Anzellotti1983}, and Reshetnyak continuity
\cite{Reshetnyak1968}.

\smallskip
\noindent\emph{Rigidity of the limiting level sets.}
Passing to the limit in the stress identities on finite value bands
shows that almost
every limiting level has constant mean curvature in the distributional
sense.  The finite-perimeter Alexandrov theorem of Delgadino and Maggi
\cite{DelgadinoMaggi2019}, extending \cite{Alexandrov1962}, identifies
such levels as finite unions of equal-radius balls.  It does not yet
give a single ball or a common center.  Nesting and quantitative matching
reduce the number of balls to at most two, and the exact limiting
divergence equation excludes the two-ball configuration.  The translation argument above then selects the center and the sign.
The full family of rescalings converges to a radial pole centered at
the puncture.

\smallskip
\noindent\emph{Regularity of the high-level end.}
Near a large height $t$, we use the pressure rescaling
\[
 (x,u(x))\longmapsto\bigl(tx,t(u(x)-t)\bigr).
\]
On each fixed vertical slab the subgraphs converge strictly in $BV$ to
the solid round cylinder.  Completion of the boundary currents before
slicing preserves orientation, while the exact mass limit fixes the
multiplicity.  First-variation and support control allow Allard's
regularity theorem \cite{Allard1972,Simon1983} to produce a single
$C^{1,\alpha}$ sheet over the cylinder; elliptic regularity then gives
smoothness.  Since the argument controls the full support,
it excludes small additional components as well as excess sheets.
The high levels form a smooth strictly convex foliation of a full
punctured neighborhood, and the height diverges uniformly with the
selected sign.

\smallskip
\noindent\emph{The two-term expansion.}
We parametrize this foliation by height and Gauss normal.  For its
support function $h(t,\omega)$, the variables $q=th$ and $Z=t^2/2$
satisfy an exact coupled system near the round state.  Translation
balance controls the degree-one modes; a uniform elliptic estimate and
the spherical spectral gap give exponential decay of the higher angular
modes in $Z$.  The remaining mean modes satisfy the radial system up to
an exponentially small error.  Its underlying integrable structure is
classical \cite{RieraRisler2002,Risler2015}.  Here a coercive Hamiltonian
barrier controls the perturbed mean system, and inversion from Gauss to
Euclidean polar coordinates gives
\eqref{eq:positive-expansion-main} and its sign reverse.

\smallskip
\noindent\emph{Bounded singularities.}
The finite-height estimates also give the weak
$W^{1,1}_{\mathrm{loc}}$ extension.  To prove continuity for
$2\le n\le7$, we use the calibrated subgraph as a local perimeter
quasiminimizer.  Codimension-one regularity, together with a directed
tangent-cone argument when $n=7$, yields regularity of the completed
interface.  Analytic elliptic regularity for the exact equation then
implies that the interface is analytic.  A nontrivial finite vertical
cluster segment cannot terminate on such an analytic hypersurface, so
the cluster interval reduces to one point.  This proves that the height
has a finite limit without assuming a gradient bound.

\subsection{Organization of the paper}

Section~\ref{sec:finite-band-complete} proves the finite-height estimates
and the critical tail law.  Sections~\ref{sec:log-profiles}
and~\ref{sec:one-ball} establish logarithmic compactness and level-set
rigidity.  Section~\ref{sec:planar-defects} excludes the planar dipole
and half-dipole candidates.  Section~\ref{sec:pressure-complete} gives
the smooth high-level foliation, and Section~\ref{sec:round-complete}
derives the asymptotic expansion.  Section~\ref{sec:bounded} treats
bounded singularities and completes the classification.

\section{Finite-height estimates and the critical tail}
\label{sec:finite-band-complete}

Throughout this section, $n\geq2$, $B_r=B_r(0)\subset\mathbb R^n$,
$k=n-1$, and $\omega_n=|B_1|$.  For a solution of \eqref{eq:pde} we write
\[
 W:=\sqrt{1+|Du|^2}.
\]
Thus $\mathcal A(Du)=Du/W$, $|\mathcal A(Du)|\leq1$, and \eqref{eq:pde} reads
$\operatorname{div}(\mathcal A(Du))=-u$.

For a measurable set $E$, its Lebesgue measure is $|E|$,
$P(E;U)$ denotes its perimeter in an open set $U$, $\partial^*E$ is its
reduced boundary, and $\nu_E$ is the measure-theoretic outer unit normal.
All topological boundaries carrying a subscript, such as
$\partial_U E$, are relative to that open set; an unsubscripted
$\partial E$ is the boundary in $\mathbb R^n$.  We write
$P(E)=P(E;\mathbb R^n)$ when no localization is needed.  The symbol
$C$ denotes a finite positive constant whose dependence is recorded when it
matters and which may change from line to line.

Whenever a working ball $B_{r_0}\Subset B_R$ has been fixed, the localized
tails are
\[
 E_t^+=\{x\in B_{r_0}\setminus\{0\}:u(x)>t\},
 \qquad
 E_t^-=\{x\in B_{r_0}\setminus\{0\}:u(x)<-t\}.
\]
All unqualified tail volumes, perimeters, and rescalings refer to the
working ball specified at the beginning of the relevant proof phase.
 
The Noether stress
\begin{equation}\label{eq:stress-complete}
 \mathsf S=\frac{Du\otimes Du}{W}-\left(W-\frac{u^2}{2}\right)\Id
\end{equation}
satisfies
\begin{equation}\label{eq:stress-identities-complete}
 \operatorname{div}\mathsf S=0,
 \qquad \mathsf S Du=\left(\frac{u^2}{2}-W^{-1}\right)Du.
\end{equation}
Both formulas follow by differentiating $W$ and using $\operatorname{div}(\mathcal A(Du))=-u$.

The first estimate localizes the area to an arbitrary bounded height band.
It requires no information about either tail of $u$.

\begin{lemma}\label{lem:finite-band-complete}
Fix $0<R_0<R$.  For every finite interval $[a,b]$ there is
$C=C(n,a,b,R_0,u)$ such that
\begin{equation}\label{eq:finite-band-complete}
 \int_{B_r\cap\{a<u<b\}}W\,dx\leq Cr^{n-1}
 \qquad(0<r\leq R_0).
\end{equation}
\end{lemma}

\begin{proof}
Choose $\chi\in C_c^\infty(\R)$ with $\chi\geq0$ and $\chi\geq1$ on $[a,b]$, and put
\[
 \mathcal H_\chi(s)=\frac12\int_0^s\tau^2\chi'(\tau)\,d\tau,
 \qquad \mathcal B_\chi(s)=\frac12\chi(s)s^2-\mathcal H_\chi(s).
\]
Both $\mathcal H_\chi$ and $\mathcal B_\chi$ are bounded.  Define
\[
 \mathbb Q_\chi=-\{\chi(u)\mathsf S-\mathcal H_\chi(u)\Id\}=\mathbb P_\chi-\mathcal B_\chi(u)\Id,
 \qquad \mathbb P_\chi=\chi(u)W(\Id-\mathcal A(Du)\otimes\mathcal A(Du)).
\]
Using \eqref{eq:stress-identities-complete} and
$\mathcal H_\chi'(s)=s^2\chi'(s)/2$, we obtain
\[
 \operatorname{div}\mathbb Q_\chi=\chi'(u)\mathcal A(Du),
 \qquad
 \operatorname{tr}\mathbb Q_\chi=\chi(u)W(n-|\mathcal A(Du)|^2)-n\mathcal B_\chi(u)
 \geq(n-1)\chi(u)W-C.
\]
Since $|\mathcal A(Du)|\leq1$, $\mathbb P_\chi$ is positive semidefinite and, for every unit vector $\nu$,
\begin{equation}\label{eq:Q-normal-complete}
 0\leq \mathbb P_\chi\nu\cdot\nu\leq \chi(u)W,
 \qquad \mathbb Q_\chi\nu\cdot\nu\leq \chi(u)W+C.
\end{equation}

 Put $\Omega_{\varepsilon,r}=B_r\setminus\overline{B_\varepsilon}$.
 Since
 \[
  \operatorname{div}(\mathbb Q_\chi x)=\operatorname{tr}\mathbb Q_\chi+x\cdot\operatorname{div}\mathbb Q_\chi,
 \]
 the divergence theorem gives
 \[
 \begin{aligned}
  \int_{\Omega_{\varepsilon,r}}\operatorname{tr}\mathbb Q_\chi\,dx
  &=r\int_{\partial B_r}\mathbb Q_\chi\nu\cdot\nu\,d\mathcal H^{n-1}
    -\varepsilon\int_{\partial B_\varepsilon}
       \mathbb Q_\chi\nu\cdot\nu\,d\mathcal H^{n-1}\\
  &\quad-\int_{\Omega_{\varepsilon,r}}
       x\cdot\operatorname{div}\mathbb Q_\chi\,dx.
 \end{aligned}
 \]
 Here the outward normal of $\Omega_{\varepsilon,r}$ on
 $\partial B_\varepsilon$ is $-\nu$.  The four terms are estimated as follows:
 \[
 \begin{aligned}
  (n-1)\int_{\Omega_{\varepsilon,r}}\chi(u)W\,dx
  &\leq \int_{\Omega_{\varepsilon,r}}\operatorname{tr}\mathbb Q_\chi\,dx
      +C|\Omega_{\varepsilon,r}|,\\
  r\int_{\partial B_r}\mathbb Q_\chi\nu\cdot\nu\,d\mathcal H^{n-1}
  &\leq r\int_{\partial B_r}\chi(u)W\,d\mathcal H^{n-1}+Cr^n,\\
  -\varepsilon\int_{\partial B_\varepsilon}
       \mathbb Q_\chi\nu\cdot\nu\,d\mathcal H^{n-1}
  &=-\varepsilon\int_{\partial B_\varepsilon}\mathbb P_\chi\nu\cdot\nu
       \,d\mathcal H^{n-1}
    +\varepsilon\int_{\partial B_\varepsilon}\mathcal B_\chi(u)\,d\mathcal H^{n-1}\\
  &\leq C\varepsilon^n,\\
  \left|\int_{\Omega_{\varepsilon,r}}
       x\cdot\operatorname{div}\mathbb Q_\chi\,dx\right|
  &\leq \|\chi'\|_{L^\infty}
       \int_{\Omega_{\varepsilon,r}}|x|\,|\mathcal A(Du)|\,dx
   \leq Cr^{n+1}\leq CR_0r^n.
 \end{aligned}
 \]
 Combining these inequalities yields
\begin{equation}\label{eq:annular-M-complete}
 (n-1)\int_{B_r\setminus B_\varepsilon}\chi(u)W\,dx
 \leq r\int_{\partial B_r}\chi(u)W\,d\mathcal H^{n-1}
 +C(r^n+\varepsilon^n).
\end{equation}

 Finally, fix $r=R_0$.  The outer integral is finite because the solution is smooth near
 $\partial B_{R_0}$.  Letting $\varepsilon\downarrow0$ in
 \eqref{eq:annular-M-complete} and using monotone convergence proves that
\[
 M(R_0):=\int_{B_{R_0}}\chi(u)W\,dx<\infty.
\]
 Put $g_\chi(x)=\chi(u(x))W(x)$.  The preceding conclusion says that
 $g_\chi\in L^1(B_{R_0})$.  Polar coordinates give
 \[
  M(r)=\int_{B_r}g_\chi(x)\,dx
      =\int_0^r m(s)\,ds,
  \qquad
  m(s)=\int_{\partial B_s}g_\chi\,d\mathcal H^{n-1}\in L^1(0,R_0).
 \]
 Hence $M\in AC([0,R_0])$ and, for almost every $r\in(0,R_0)$,
 \[
  M'(r)=m(r)=\int_{\partial B_r}\chi(u)W\,d\mathcal H^{n-1}.
 \]
 Fix such an $r$.  Letting $\varepsilon\downarrow0$ in
 \eqref{eq:annular-M-complete} gives
 \[
  \int_{B_r\setminus B_\varepsilon}\chi(u)W\,dx\longrightarrow M(r),
  \qquad C\varepsilon^n\longrightarrow0,
 \]
 and therefore
\begin{equation}\label{eq:M-differential-complete}
 (n-1)M(r)\leq rM'(r)+Cr^n,
 \qquad
 \left(\frac{M(r)}{r^{n-1}}\right)'
 =\frac{rM'(r)-(n-1)M(r)}{r^n}\geq-C.
\end{equation}
 The quotient is absolutely continuous on each interval
 $[r,R_0]\subset(0,R_0]$, so integrating from $r$ to $R_0$ gives
\[
 M(r)\leq
 \left(\frac{M(R_0)}{R_0^{n-1}}+CR_0\right)r^{n-1}.
\]
Since $\chi\geq1$ on $[a,b]$, this proves \eqref{eq:finite-band-complete} for every $r$.
\end{proof}

\begin{corollary}\label{cor:finite-band-stress-complete}
For every finite $a<b$ and $0<r\leq R_0$,
\begin{equation}\label{eq:finite-band-stress-complete}
 \int_{B_r\cap\{a<u<b\}}|\mathsf S|\,dx\leq Cr^{n-1}.
\end{equation}
\end{corollary}

\begin{proof}
On a finite height band, \eqref{eq:stress-complete} gives $|\mathsf S|\leq C_{n,a,b}W$.
 Thus \eqref{eq:finite-band-stress-complete} follows from
 Lemma~\ref{lem:finite-band-complete}.
\end{proof}

For later use, let $\chi_r$ satisfy $\chi_r=0$ on $B_r$,
$\chi_r=1$ outside $B_{2r}$, and $|D\chi_r|\leq C/r$.  If
$2r\leq R_0$, then
\begin{equation}\label{eq:finite-band-stress-cutoff}
\begin{aligned}
 \int_{\{a<u<b\}}|\mathsf S|\,|D\chi_r|\,dx
 &\leq\frac Cr\int_{B_{2r}\cap\{a<u<b\}}|\mathsf S|\,dx\\
 &\leq Cr^{n-2}.
\end{aligned}
\end{equation}
For $n\geq3$ the right-hand side tends to zero as $r\downarrow0$.  In dimension two it is only bounded; this is precisely the critical horizontal translation defect isolated later.
 
 \subsection{One-tail identities and the critical volume law}
\label{sec:level-volume-complete}

Assume in this section only that
\begin{equation}\label{eq:one-tail-limsup}
 \limsup_{x\to0}u(x)=+\infty.
\end{equation}
Fix $r_0\in(0,R)$.  Unless another radius is specified, all tail sets,
perimeters, level surfaces, and rescalings below are taken relative to
$B_{r_0}$.  In the notation introduced above, write
\[
 E_t:=E_t^+,\qquad V(t):=|E_t|,\qquad
 \rho(t):=\sup_{x\in E_t}|x|.
\]
Every sufficiently high finite $E_t$ is nonempty.  Moreover
$\rho(t)\to0$: for every $0<\delta<r_0$, the function $u$ is bounded on
$\overline B_{r_0}\setminus B_\delta$, so $E_t\subset B_\delta$ once $t$
exceeds that bound.  Choose $t_0>0$ so large that
$E_s\Subset B_{r_0}$ for every $s\geq t_0$.  Thus every $E_s$ is a bounded
measurable subset of $\mathbb R^n$; whenever it has finite perimeter, its
relative and full-space perimeters agree.

We next justify coarea across the puncture.  For $t_0<a<b<\infty$, let
 \[
  T_{a,b}(s)=\max\{a,\min\{s,b\}\}.
 \]
 The Sobolev chain rule on $B_{r_0}\setminus\{0\}$ gives
 \[
  D(T_{a,b}(u))=\mathbf1_{\{a<u<b\}}Du\quad\text{a.e.}
 \]
 Moreover, $|Du|\leq W$ and Lemma~\ref{lem:finite-band-complete} imply
 \[
  \int_{B_{r_0}\cap\{a<u<b\}}|Du|\,dx
  \leq\int_{B_{r_0}\cap\{a<u<b\}}W\,dx<\infty.
 \]
 Let $\varphi\in C_c^1(B_{r_0};\mathbb R^n)$, put
 $\Omega_\varepsilon=B_{r_0}\setminus\overline B_\varepsilon$, and let
 $\nu_\varepsilon$ be the outer normal of $B_\varepsilon$.  Since the outer
 normal of $\Omega_\varepsilon$ on the inner sphere is $-\nu_\varepsilon$,
 integration by parts gives
 \[
 \begin{aligned}
  \int_{\Omega_\varepsilon}T_{a,b}(u)\operatorname{div}\varphi\,dx
  &=-\int_{\Omega_\varepsilon}\mathbf1_{\{a<u<b\}}
       Du\cdot\varphi\,dx\\
  &\quad-\int_{\partial B_\varepsilon}T_{a,b}(u)
       \varphi\cdot\nu_\varepsilon\,d\mathcal H^{n-1}.
 \end{aligned}
 \]
 The boundary term satisfies
 \[
 \begin{aligned}
  \left|\int_{\partial B_\varepsilon}T_{a,b}(u)
       \varphi\cdot\nu_\varepsilon\,d\mathcal H^{n-1}\right|
  &\leq \max\{|a|,|b|\}\|\varphi\|_{L^\infty}
       \mathcal H^{n-1}(\partial B_\varepsilon)\\
  &=O(\varepsilon^{n-1})\longrightarrow0.
 \end{aligned}
 \]
 Both volume integrands are integrable on $B_{r_0}$, the second by the
 finite-band estimate above.  Hence dominated convergence, together with the
 vanishing boundary term, allows us to let $\varepsilon\downarrow0$ and gives

 \[
  \int_{B_{r_0}}T_{a,b}(u)\operatorname{div}\varphi\,dx
  =-\int_{B_{r_0}}\mathbf1_{\{a<u<b\}}Du\cdot\varphi\,dx.
 \]
 Hence
 \[
  T_{a,b}(u)\in W^{1,1}(B_{r_0}),
  \qquad
  D(T_{a,b}(u))=\mathbf1_{\{a<u<b\}}Du
  \quad\text{in }\mathcal D'(B_{r_0}).
 \]
 The Sobolev coarea formula now gives
 \begin{equation}\label{eq:finite-level-coarea}
  \int_a^b P(E_s;B_{r_0})\,ds
  =\int_{B_{r_0}\cap\{a<u<b\}}|Du|\,dx<\infty.
 \end{equation}
 Because $E_s\Subset B_{r_0}$ on this range,
 $P(E_s;B_{r_0})=P(E_s;\mathbb R^n)$.

 To justify the weighted coarea formula, note first that the critical set of
 $u$ has measure zero.  On every compact set
 $K\Subset B_{r_0}\setminus\{0\}$ the gradient is bounded, so the
 capillary equation is uniformly elliptic there.  Standard elliptic
 bootstrapping and analytic regularity give
 \[
  u\in C^\omega(B_{r_0}\setminus\{0\})
 \]
 \cite{GilbargTrudinger2001,Morrey1958}.  Since
 \eqref{eq:one-tail-limsup} excludes constant solutions, some derivative
 $\partial_{i_0}u$ is a nontrivial analytic function; hence
 \begin{equation}\label{eq:critical-set-null}
  |\{Du=0\}|=0.
 \end{equation}
 Likewise, for every $s\in\mathbb R$, the nontrivial analytic function
 $u-s$ has a zero set of Lebesgue measure zero.
 Applying coarea with the weight $|Du|^{-1}$ on $\{|Du|>0\}$ and using
 \eqref{eq:critical-set-null}, we obtain, for $t_0<a<b<\infty$,
 \begin{equation}\label{eq:volume-absolute-continuity}
  V(a)-V(b)
  =\int_a^b\int_{\partial^*E_s}\frac1{|Du|}
       \,d\mathcal H^{n-1}\,ds.
 \end{equation}
 In particular, $V\in AC_{\mathrm{loc}}((t_0,\infty))$.

 By \eqref{eq:finite-level-coarea}, $E_t$ has finite perimeter for almost
 every $t>t_0$. Since $u\in C^\infty(B_{r_0}\setminus\{0\})$, Sard's
 theorem shows that almost every $t$ is a regular value. Throughout this
 subsection, identities involving regular-level boundary integrals or level
 derivatives are understood to hold for almost every sufficiently high
 $t$; exceptional null sets are immaterial when integrating in the level
 variable. Whenever a specific level is selected, it is chosen from the
 common full-measure set on which the required identities hold.

 Set
 \[
  \alpha:=W^{-1},\qquad
  \sigma_t:=\frac{|Du|}{W},\qquad
  \nu_t:=-\frac{Du}{|Du|},\qquad
  P(t):=\mathcal H^{n-1}(\partial^*E_t).
 \]
 At a regular level, $\nu_t$ is the measure-theoretic outer normal and
 $\sigma_t^2+\alpha^2=1$.

 Applying Gauss--Green to $E_t\setminus\overline B_\varepsilon$, using
 $\operatorname{div}(\mathcal A(Du))=-u$, $\mathcal A(Du)\cdot\nu_t=-\sigma_t$, and $|\mathcal A(Du)|\le1$,
 gives
 \[
  \int_{E_t\setminus B_\varepsilon}u\,dx
  =\int_{\partial^*E_t\setminus B_\varepsilon}\sigma_t\,d\mathcal H^{n-1}
   +O(\varepsilon^{n-1}).
 \]
 Since $u>t>0$ on $E_t$, monotone convergence yields
 \begin{equation}\label{eq:positive-source-integrable}
  u\mathbf1_{E_t}\in L^1(B_{r_0}),
 \end{equation}
 and hence, after letting $\varepsilon\downarrow0$,
 \[
  \int_{E_t}u\,dx
  =\int_{\partial^*E_t}\sigma_t\,d\mathcal H^{n-1}.
 \]
 Increasing $t_0$ if necessary, assume that $t_0$ is such a regular level.
 For every $t\ge t_0$ define
 \[
  J(t):=\int_{E_t}u\,dx,\qquad
  G_1(t):=\int_t^\infty V(s)\,ds.
 \]
 The first quantity is finite because
 $0\le u\mathbf1_{E_t}\le u\mathbf1_{E_{t_0}}\in L^1$, and at regular
 levels it agrees with the boundary flux above.
Tonelli's theorem gives the layer-cake identity
\begin{equation}\label{eq:tail-J}
\begin{aligned}
 J(t)=\int_{E_t}u\,dx
 &=tV(t)+\int_{E_t}(u-t)\,dx\\
 &=tV(t)+\int_{E_t}\int_t^{u(x)}ds\,dx\\
 &=tV(t)+\int_t^\infty V(s)\,ds
  =tV(t)+G_1(t).
\end{aligned}
\end{equation}
This calculation also proves $G_1(t)<\infty$ and identifies \(J\) with a
locally absolutely continuous representative.  By
\eqref{eq:volume-absolute-continuity}, for almost every $t$,
\[
 V'(t)=-\int_{\partial^*E_t}\frac1{|Du|}\,d\mathcal H^{n-1}
       =-\int_{\partial^*E_t}\frac\alpha{\sigma_t}\,d\mathcal H^{n-1}.
\]
Since $G_1'(t)=-V(t)$ almost everywhere, differentiating
\eqref{eq:tail-J} yields
\begin{equation}\label{eq:tail-derivative-ledger}
 G_1'(t)=-V(t),\qquad
 V'(t)=-\int_{\partial^*E_t}\frac\alpha{\sigma_t}
       \,d\mathcal H^{n-1},\qquad
 J'(t)=V(t)+tV'(t)+G_1'(t)=tV'(t)
\end{equation}
for almost every \(t>t_0\).
Finally,
\[
 P(t)-J(t)=\int_{\partial^*E_t}(1-\sigma_t)
       \,d\mathcal H^{n-1}\geq0.
\]
The pointwise identities $\sigma_t^2+\alpha^2=1$ and
$0<\sigma_t,\alpha\leq1$ imply
\[
 1-\sigma_t=\frac{\alpha^2}{1+\sigma_t}
 \leq\frac{\alpha^2}{\sigma_t}\leq\frac\alpha{\sigma_t}.
\]
Therefore
\begin{equation}\label{eq:tail-PJ}
 0\leq P(t)-J(t)
 \leq\int_{\partial^*E_t}\frac\alpha{\sigma_t}\,d\mathcal H^{n-1}
 =-V'(t)
\end{equation}
for almost every sufficiently high $t$.

\paragraph{The translation balance.} The next conclusion is needed directly in dimensions $n\ge3$; in dimension two its possible failure is retained as a defect measure and analyzed in \cref{sec:planar-defects}.

\begin{lemma}\label{lem:tail-translation}
Assume $n\geq3$.  For almost every sufficiently high $t$,
\begin{equation}\label{eq:tail-translation}
 \int_{\partial^*E_t}\alpha\nu_t\,d\mathcal H^{n-1}=0.
\end{equation}
\end{lemma}

\begin{proof}
On \(u=s\), the stress satisfies
\begin{equation}\label{eq:tail-stress-flux}
 \mathsf S\nu_s=\left(\frac{s^2}{2}-\alpha\right)\nu_s.
\end{equation}
 Fix regular levels $t<T$, a constant vector
 \(e\in\mathbb R^n\), and a
 puncture cutoff $\chi_\delta$.  Since $\operatorname{div}\mathsf S=0$, integration
 by parts on $\{t<u<T\}$ gives
 \[
 \begin{aligned}
  0
  &=-\int_{\{t<u<T\}}\mathsf S e\cdot D\chi_\delta\,dx
    +\int_{\partial^*E_t}\chi_\delta(\mathsf S e)\cdot\nu_t
       \,d\mathcal H^{n-1}\\
  &\quad-\int_{\partial^*E_T}\chi_\delta(\mathsf S e)\cdot\nu_T
       \,d\mathcal H^{n-1}.
 \end{aligned}
 \]
 The minus sign on the upper level occurs because the outward normal of the
 slab there is $-\nu_T$.  There is no contribution from
 $\partial B_{r_0}$ because $E_t\Subset B_{r_0}$.  For fixed $t,T$,
 \[
 \begin{aligned}
  \left|\int_{\{t<u<T\}}\mathsf S e\cdot D\chi_\delta\,dx\right|
  &\leq\frac{C|e|}{\delta}
       \int_{B_{2\delta}\cap\{t<u<T\}}|\mathsf S|\,dx\\
  &\leq C_{t,T}|e|\delta^{n-2}\longrightarrow0.
 \end{aligned}
 \]
 On either fixed level surface, \eqref{eq:tail-stress-flux} makes the
 boundary integrand bounded and integrable.  Dominated convergence therefore
 permits $\chi_\delta\to1$ and yields, since the equality holds for every
 $e$,
 \begin{equation}\label{eq:tail-translation-flux-independence}
  \int_{\partial^*E_t}\left(\frac{t^2}{2}-\alpha\right)\nu_t
       \,d\mathcal H^{n-1}
  =\int_{\partial^*E_T}\left(\frac{T^2}{2}-\alpha\right)\nu_T
       \,d\mathcal H^{n-1}.
 \end{equation}
 For every bounded finite-perimeter set $E_s$ and every constant vector
 $e$, Gauss--Green gives
 \[
  e\cdot\int_{\partial^*E_s}\nu_s\,d\mathcal H^{n-1}
  =\int_{E_s}\operatorname{div}e\,dx=0.
 \]
 Hence $\int_{\partial^*E_s}\nu_s\,d\mathcal H^{n-1}=0$, and
 \eqref{eq:tail-translation-flux-independence} reduces to
 \[
  \int_{\partial^*E_t}\alpha\nu_t\,d\mathcal H^{n-1}
  =\int_{\partial^*E_T}\alpha\nu_T\,d\mathcal H^{n-1}.
 \]
 Moreover,
 \[
 \begin{aligned}
  \left|\int_{\partial^*E_s}\alpha\nu_s\,d\mathcal H^{n-1}\right|
  &\leq\int_{\partial^*E_s}\alpha\,d\mathcal H^{n-1}\\
  &\leq\int_{\partial^*E_s}\frac\alpha{\sigma_s}\,d\mathcal H^{n-1}
   =-V'(s).
 \end{aligned}
 \]
 Since
 \[
  \int_{t_0}^\infty[-V'(s)]\,ds
  =V(t_0)-\lim_{T\to\infty}V(T)<\infty,
 \]
 the integrals of $-V'$ over $[j,j+1]$ tend to zero. Hence one may choose
 $s_j\in[j,j+1]$ from the common full-measure set so that
 \[
  0\le -V'(s_j)
  \le 2\int_j^{j+1}[-V'(s)]\,ds\longrightarrow0.
 \]
 The flux in \eqref{eq:tail-translation-flux-independence} is independent of
 the regular level after using $\int_{\partial^*E_s}\nu_s=0$; evaluating it
 along this sequence therefore forces it to vanish.
 For any fixed regular level $t$, 
 \eqref{eq:tail-translation-flux-independence} then gives
 \[
  \left|\int_{\partial^*E_t}\alpha\nu_t\,d\mathcal H^{n-1}\right|
  =\left|\int_{\partial^*E_{s_j}}\alpha\nu_{s_j}
       \,d\mathcal H^{n-1}\right|
  \leq -V'(s_j)\longrightarrow0,
 \]
 which proves \eqref{eq:tail-translation}.
\end{proof}

\begin{remark}\label{rem:translation-dimension-two}
The restriction $n\ge3$ in Lemma~\ref{lem:tail-translation} is essential
for this direct stress-cutoff proof.  When $n=2$ the cutoff cost is only
$O(1)$ and is retained as a height-dependent horizontal translation
defect.  It will be identified with the horizontal component of the
axis-supported first variation in Section~\ref{sec:planar-defects}.
\end{remark}

The dilation identity is most conveniently written in terms of the boundary
flux and its nonnegative defect
\[
 \mathcal K_{\rm dil}(t):=\int_{\partial^*E_t}\alpha\,x\cdot\nu_t\,d\mathcal H^{n-1},
 \qquad
 \mathcal R_{\rm def}(t):=n\int_{\partial^*E_t}\frac{\alpha^2}{\sigma_t}\,d\mathcal H^{n-1}.
\]
For almost every $t$ at which
\eqref{eq:tail-derivative-ledger} holds,
\[
\begin{aligned}
 |\mathcal K_{\rm dil}(t)|
 &\leq\rho(t)\int_{\partial^*E_t}\alpha\,d\mathcal H^{n-1}\\
 &\leq\rho(t)\int_{\partial^*E_t}\frac\alpha{\sigma_t}
      \,d\mathcal H^{n-1}
  =\rho(t)[-V'(t)]<\infty.
\end{aligned}
\]
The other integrand that appears below is also integrable, because
\[
 \sigma_t^{-1}-1
 =\frac{\alpha^2}{\sigma_t(1+\sigma_t)}
 \leq\frac\alpha{\sigma_t},
 \qquad
 \frac{\alpha^2}{\sigma_t}\leq\frac\alpha{\sigma_t}.
\]
In particular,
\[
 0\leq\int_{\partial^*E_t}
 \left((n-1)(\sigma_t^{-1}-1)+\frac{\alpha^2}{\sigma_t}\right)
 \,d\mathcal H^{n-1}
 \leq n[-V'(t)].
\]

Fix regular levels $t<T$.  Since
$\operatorname{div}(\mathsf S x)=\operatorname{tr}\mathsf S$, integration by parts with the
puncture cutoff $\chi_\delta$ gives
\[
\begin{aligned}
 \int_{\{t<u<T\}}\chi_\delta\operatorname{tr}\mathsf S\,dx
 &=\int_{\partial^*E_t}\chi_\delta(\mathsf S x)\cdot\nu_t
      \,d\mathcal H^{n-1}\\
 &\quad-\int_{\partial^*E_T}\chi_\delta(\mathsf S x)\cdot\nu_T
      \,d\mathcal H^{n-1}\\
 &\quad-\int_{\{t<u<T\}}Sx\cdot D\chi_\delta\,dx.
\end{aligned}
\]
The cutoff term vanishes, since $|x||D\chi_\delta|\leq C$ and
\[
 \left|\int_{\{t<u<T\}}Sx\cdot D\chi_\delta\,dx\right|
 \leq C\int_{B_{2\delta}\cap\{t<u<T\}}|\mathsf S|\,dx
 \leq C_{t,T}\delta^{n-1}\longrightarrow0.
\]
The remaining volume and boundary integrands are integrable by the
finite-band estimate and the finite perimeters of the two levels, so
dominated convergence permits $\delta\downarrow0$.

On $u=s$,
\[
 \mathsf S\nu_s=\left(\frac{s^2}{2}-\alpha\right)\nu_s,
 \qquad
 \int_{\partial^*E_s}x\cdot\nu_s\,d\mathcal H^{n-1}
 =\int_{E_s}\operatorname{div}x\,dx=nV(s).
\]
Consequently
\[
 \int_{\partial^*E_s}(\mathsf S x)\cdot\nu_s\,d\mathcal H^{n-1}
 =\frac n2s^2V(s)-\mathcal K_{\rm dil}(s).
\]
Remembering that the upper slab normal is $-\nu_T$, we obtain
\[
 \int_{\{t<u<T\}}\operatorname{tr}\mathsf S\,dx
 =\left(\frac n2t^2V(t)-\mathcal K_{\rm dil}(t)\right)
  -\left(\frac n2T^2V(T)-\mathcal K_{\rm dil}(T)\right).
\]
The signed coarea formula rewrites the left-hand side as
\[
 \int_t^T\int_{\partial^*E_s}
 \frac{\operatorname{tr}\mathsf S}{|Du|}\,d\mathcal H^{n-1}\,ds.
\]
Hence $nt^2V(t)/2-\mathcal K_{\rm dil}(t)$ is locally absolutely continuous and, for almost
every $t$,
\[
 \left(\frac n2t^2V(t)-\mathcal K_{\rm dil}(t)\right)'
 =-\int_{\partial^*E_t}\frac{\operatorname{tr}\mathsf S}{|Du|}
      \,d\mathcal H^{n-1}.
\]
Since
\[
\begin{aligned}
 -\frac{\operatorname{tr}\mathsf S}{|Du|}
 &=(n-1)\frac W{|Du|}+\frac\alpha{|Du|}
   -\frac n2t^2\frac1{|Du|}\\
 &=\frac{n-1}{\sigma_t}+\frac{\alpha^2}{\sigma_t}
   -\frac n2t^2\frac1{|Du|},
\end{aligned}
\]
we have
\[
 \left(\frac n2t^2V-K\right)'
 =(n-1)\int_{\partial^*E_t}\frac1{\sigma_t}\,d\mathcal H^{n-1}
  +\int_{\partial^*E_t}\frac{\alpha^2}{\sigma_t}\,d\mathcal H^{n-1}
  +\frac n2t^2V'(t).
\]
Expanding the derivative on the left and cancelling the two copies of
$nt^2V'(t)/2$ gives
\begin{equation}\label{eq:tail-dilation}
\begin{aligned}
 \mathcal K_{\rm dil}'(t)
 &=ntV(t)-(n-1)\int_{\partial^*E_t}\frac1{\sigma_t}
      \,d\mathcal H^{n-1}
   -\int_{\partial^*E_t}\frac{\alpha^2}{\sigma_t}
      \,d\mathcal H^{n-1}\\
 &=ntV(t)-(n-1)P(t)\\
 &\quad-\int_{\partial^*E_t}
  \left((n-1)(\sigma_t^{-1}-1)+\frac{\alpha^2}{\sigma_t}\right)
  \,d\mathcal H^{n-1}.
\end{aligned}
\end{equation}

Because \(J=tV+G_1\),
\[
 tV-(n-1)G_1=ntV-(n-1)J.
\]
Subtracting \eqref{eq:tail-dilation} and using
$P-J=\int_{\partial^*E_t}(1-\sigma_t)$ gives
\[
\begin{aligned}
 tV-(n-1)G_1-K'
 &=\int_{\partial^*E_t}
 \left[(n-1)(1-\sigma_t)+(n-1)(\sigma_t^{-1}-1)
       +\frac{\alpha^2}{\sigma_t}\right]d\mathcal H^{n-1}\\
 &=\int_{\partial^*E_t}
 \left[(n-1)(\sigma_t^{-1}-\sigma_t)+\frac{\alpha^2}{\sigma_t}\right]
 d\mathcal H^{n-1}\\
 &=n\int_{\partial^*E_t}\frac{\alpha^2}{\sigma_t}
      \,d\mathcal H^{n-1},
\end{aligned}
\]
where \(\sigma_t^{-1}-\sigma_t=\alpha^2/\sigma_t\).  Therefore
\begin{align}
 tV(t)-(n-1)G_1(t)&=\mathcal K_{\rm dil}'(t)+\mathcal R_{\rm def}(t),\label{eq:tail-master}\\
 0\leq \mathcal R_{\rm def}(t)
 &\leq n\int_{\partial^*E_t}\frac\alpha{\sigma_t}
      \,d\mathcal H^{n-1}
  =n[-V'(t)],\label{eq:tail-R-bound}\\
 |\mathcal K_{\rm dil}(t)|
 &\leq\rho(t)\int_{\partial^*E_t}\frac\alpha{\sigma_t}
      \,d\mathcal H^{n-1}
  =\rho(t)[-V'(t)].\label{eq:tail-K-bound}
\end{align}

All identities above are henceforth used in their almost-everywhere sense.
Whenever a particular regular level is needed, it is chosen from the common
full-measure set on which the relevant identities hold.

\begin{proposition}\label{prop:critical-volume-complete}
There is a constant $C_*>0$ such that
\begin{equation}\label{eq:critical-volume-complete}
 \lim_{t\to\infty}t^nV(t)=C_*,\qquad
 \lim_{t\to\infty}t^{n-1}J(t)=\frac{nC_*}{n-1},\qquad
 C_*\geq\omega_n(n-1)^n.
\end{equation}
\end{proposition}

\begin{proof}
Write
\[
 \mathfrak If(t):=\int_t^\infty f(s)\,ds,
 \qquad G_0:=V,\qquad G_m:=\mathfrak I^mV\quad(m\ge1).
\]
Whenever finite,
\[
 G_m(t)=\frac1{(m-1)!}\int_t^\infty(s-t)^{m-1}V(s)\,ds.
\]
The first moment is finite by \eqref{eq:tail-J}.  From
\eqref{eq:tail-K-bound}, monotonicity of $\rho$, and integration by parts,
\begin{equation}\label{eq:K-integrability-hierarchy}
 \int_T^\infty |\mathcal K_{\rm dil}(t)|\,dt\le \rho(T)V(T),
 \qquad
 |\mathfrak I^m\mathcal K_{\rm dil}(t)|\le \rho(t)G_{m-1}(t)\quad(m\ge1),
\end{equation}
and, whenever $G_m$ is finite,
\begin{equation}\label{eq:IK-integrability-hierarchy}
 \int_T^\infty|\mathfrak I^m\mathcal K_{\rm dil}(t)|\,dt\le \rho(T)G_m(T).
\end{equation}
In particular $\mathcal K_{\rm dil}\in L^1(T,\infty)$ and there are good levels tending to
infinity on which $\mathcal K_{\rm dil}\to0$.  This conclusion uses only the scalar bound
\eqref{eq:tail-K-bound}; it does not use the horizontal translation
identity.

\smallskip
\noindent\textbf{The terminal step when $n=2$.}
In this case \eqref{eq:tail-master} is already terminal:
\[
 M(t):=tG_1(t)+\mathcal K_{\rm dil}(t),\qquad M'(t)=-\mathcal R_{\rm def}(t)\le0.
\]
The monotone function $M$ has a finite limit $\ell\ge0$: along good
levels with $\mathcal K_{\rm dil}\to0$ one has $M=tG_1+o(1)\ge o(1)$, while $M$ is bounded
above by its value at one fixed level.  Moreover
\[
 \int_T^\infty|\mathcal K_{\rm dil}(t)|\,dt\le\rho(T)V(T)\longrightarrow0.
\]
For every large $t$, average $|\mathcal K_{\rm dil}|$ on
$[t-\sqrt t,t]$ and $[t,t+\sqrt t]$ and choose good points
$t_-\le t\le t_+$ there for which $\mathcal K_{\rm dil}(t_\pm)\to0$.  Then
$t_\pm/t\to1$, $M(t_\pm)\to\ell$, and monotonicity of $G_1$ gives
\[
 \frac{t}{t_+}\,t_+G_1(t_+)
 \le tG_1(t)
 \le \frac{t}{t_-}\,t_-G_1(t_-).
\]
Hence
\begin{equation}\label{eq:n2-terminal-moment}
 tG_1(t)\longrightarrow\ell.
\end{equation}
Set $C_*=\ell$ for the moment; positivity will be proved below.

\smallskip
\noindent\textbf{The moment hierarchy when $n\ge3$.}
For $m=1$, \eqref{eq:tail-master} gives
\[
 \frac d{dt}\{tG_1+\mathcal K_{\rm dil}\}=-(n-2)G_1-\mathcal R_{\rm def}\le0.
\]
Let $L$ be its limit.  Since $\mathcal K_{\rm dil}\in L^1$ there are levels on which
$\mathcal K_{\rm dil}\to0$, so $L\ge0$.  If $L>0$, then outside the finite-measure set
$\{|\mathcal K_{\rm dil}|>L/2\}$ one has $tG_1(t)\ge L/2$.  Thus
$\int^\infty G_1(t)\,dt=\infty$, while integrating the preceding
differential inequality would force $tG_1+\mathcal K_{\rm dil}\to-\infty$, a contradiction.
Hence the limit is zero, and integration from $t$ to infinity gives
\[
 tG_1+\mathcal K_{\rm dil}=(n-2)G_2+\mathfrak I\mathcal R_{\rm def}.
\]

The same argument iterates.  Suppose $1\le m\le n-2$ and $G_m$ is
finite.  Differentiating gives
\begin{equation}\label{eq:moment-differential}
 \frac d{dt}\bigl(tG_m+\mathfrak I^{m-1}\mathcal K_{\rm dil}\bigr)
 =-(n-m-1)G_m-\mathfrak I^{m-1}\mathcal R_{\rm def}.
\end{equation}
When $m<n-1$ the coefficient $n-m-1$ is positive.  By
\eqref{eq:IK-integrability-hierarchy}, $\mathfrak I^{m-1}\mathcal K_{\rm dil}$ is integrable
on every tail.  The same finite-measure bad-set argument therefore shows
that the decreasing quantity on the left of
\eqref{eq:moment-differential} has limit zero.  Consequently
\begin{equation}\label{eq:moment-hierarchy-complete}
 tG_m+\mathfrak I^{m-1}\mathcal K_{\rm dil}
 =(n-m-1)G_{m+1}+\mathfrak I^m\mathcal R_{\rm def},
 \qquad 1\le m\le n-2.
\end{equation}
This identity inductively proves the finiteness of the next moment.

At the terminal index $m=n-1$,
\[
 \frac d{dt}\bigl(tG_{n-1}+\mathfrak I^{n-2}\mathcal K_{\rm dil}\bigr)
 =-\mathfrak I^{n-2}\mathcal R_{\rm def}\le0.
\]
The pointwise estimate in \eqref{eq:K-integrability-hierarchy} gives
$\mathfrak I^{n-2}\mathcal K_{\rm dil}(t)\to0$.  Hence a finite limit
\begin{equation}\label{eq:terminal-moment-limit}
 \frac{C_*}{(n-1)!}:=\lim_{t\to\infty}tG_{n-1}(t)\ge0
\end{equation}
exists.  For $n=2$, \eqref{eq:terminal-moment-limit} is exactly
\eqref{eq:n2-terminal-moment}.

\smallskip
\noindent\textbf{Positivity of the terminal constant.}
Assume $C_*=0$.  The moment representation at $t/3$ gives
\[
 G_{n-1}(t/3)
 \ge \frac{t^{n-1}}{3^{n-1}(n-2)!}V(t),
\]
with the evident interpretation $(n-2)!=1$ when $n=2$.  Therefore
\begin{equation}\label{eq:small-critical-volume}
 t^nV(t)\longrightarrow0.
\end{equation}
Put $\delta(t)=\sup_{s\ge t}s^nV(s)\to0$.  Splitting $G_1(t)$ at
$t+\varepsilon V(t)^{-1/n}$ gives
\[
 G_1(t)
 \le \varepsilon V(t)^{(n-1)/n}
 +\frac{\delta(t)}{n-1}\varepsilon^{1-n}V(t)^{(n-1)/n}.
\]
First $t\to\infty$ and then $\varepsilon\downarrow0$ yield
\[
 G_1=o\bigl(V^{(n-1)/n}\bigr).
\]
Equation \eqref{eq:small-critical-volume} also gives
$tV=o(V^{(n-1)/n})$, hence
\[
 J(t)=o\bigl(V(t)^{(n-1)/n}\bigr).
\]
For almost every high $t$, the Euclidean isoperimetric inequality and
\eqref{eq:tail-PJ} imply
\[
 n\omega_n^{1/n}V(t)^{(n-1)/n}
 \le P(t)\le J(t)-V'(t).
\]
Thus
\[
 -\frac d{dt}V(t)^{1/n}\ge c_n>0
\]
on a sufficiently high tail, forcing $V$ to vanish at a finite height,
contrary to the assumed unbounded tail.  Therefore $C_*>0$.

\smallskip
\noindent\textbf{Recovery of the critical density and the sharp lower bound.}
We use the elementary monotone-density fact
\[
 t^\beta\mathfrak If(t)\to A
 \quad\Longrightarrow\quad
 t^{\beta+1}f(t)\to\beta A
\]
for nonnegative nonincreasing $f$.  Applying it successively to
$G_m=\mathfrak IG_{m-1}$ and starting from
\eqref{eq:terminal-moment-limit} gives
\begin{equation}\label{eq:critical-volume-density-recovery}
 t^nV(t)\longrightarrow C_*,\qquad
 t^{n-1}G_1(t)\longrightarrow\frac{C_*}{n-1}.
\end{equation}
Since $J=tV+G_1$,
\[
 t^{n-1}J(t)\longrightarrow\frac{nC_*}{n-1}.
\]
Finally
\[
 \int_T^{2T}[-V'(s)]\,ds=V(T)-V(2T)=O(T^{-n}).
\]
Choose a good $t_T\in[T,2T]$ so that
$-V'(t_T)=o(t_T^{1-n})$.  At this level
\[
 n\omega_n^{1/n}V(t_T)^{(n-1)/n}
 \le P(t_T)\le J(t_T)-V'(t_T).
\]
Divide by $t_T^{1-n}$ and use
\eqref{eq:critical-volume-density-recovery}.  The result is
\[
 C_*\ge\omega_n(n-1)^n.
\]
This proves \eqref{eq:critical-volume-complete} in every dimension
$n\ge2$.
\end{proof}

 One final identity will be used in the pressure slabs.  Graph coarea and
 \eqref{eq:tail-R-bound} give, for $t_0<a<b<\infty$,
 \begin{equation}\label{eq:finite-graph-band-identity}
 \int_{B_{r_0}\cap\{a<u<b\}}W\,dx
 =\int_a^b\left(J(s)+\frac1n\mathcal R_{\rm def}(s)\right)ds.
\end{equation}
Moreover, for every fixed $L>0$, the volume asymptotic and
$R\leq n(-V')$ imply
\begin{equation}\label{eq:narrow-R-vanishing}
 t^n\int_{t-L/t}^{t+L/t}\mathcal R_{\rm def}(s)\,ds
 \leq nt^n\{V(t-L/t)-V(t+L/t)\}\longrightarrow0.
\end{equation}
All arguments in this section apply verbatim to $-u$ when the negative tail is unbounded.
 
 \subsection{Morrey and shrinking-height consequences}
\begin{corollary}\label{cor:L1-morrey}
For every solution and every sufficiently small $r$,
\[
 \int_{B_r}|u|\,dx\le Cr^{n-1}.
\]
Consequently, for every finite $a<b$,
\begin{equation}\label{common:shrinking-band}
 \int_{B_r\cap\{a<u<b\}}\sqrt{1+|Du|^2}\,dx
 \le C\bigl(r^{n-1}|b-a|+r^n\bigr).
\end{equation}
In graph variables, on every compact height interval $I$ and every
finite-height axis point,
\begin{equation}\label{common:ambient-area}
 \mathcal H^n\bigl(\Sigma\cap B^{n+1}_r((0,z_0))\bigr)\le C_{z_0}r^n.
\end{equation}
\end{corollary}
\begin{proof}
For the first estimate use the distribution formula separately on the
positive and negative parts.  Above a fixed tail threshold the critical
law gives $V_\pm(t)\le Ct^{-n}$, whereas for bounded $t$ the trivial bound
$|B_r|\le Cr^n$ is better.  Hence
\[
 \int_{B_r}|u|\,dx
 =\int_0^\infty|B_r\cap\{|u|>t\}|\,dt
 \le C\int_0^\infty\min\{r^n,t^{-n}\}\,dt
 \le Cr^{n-1},
\]
where $n>1$ is exactly what makes the final integral finite.

We next record the dependence of the finite-band estimate on the height
width.  Choose a nondecreasing Lipschitz function $h$ with
$h'=1$ on $[a,b]$ and $0\le h\le|b-a|$.  Let $\eta$ be supported in
$B_{2r}$, equal one on $B_r$, and satisfy $|D\eta|\le C/r$; let
$\chi_\delta$ be a puncture cutoff.  Test
$\operatorname{div}(Du/W)=-u$ by $\eta\chi_\delta h(u)$.  Since
\[
 \frac{Du}{W}\cdot Du=\frac{|Du|^2}{W}=W-W^{-1}\ge W-1,
\]
the term containing $h'(u)$ controls the graph area on the desired band.
The inner-sphere error is $O(|b-a|\delta^{n-1})$ and disappears for every
$n\ge2$.  Letting $\delta\downarrow0$ and estimating the remaining cutoff
and source terms gives
\[
 \int_{B_r\cap\{a<u<b\}}W\,dx
 \le Cr^n+C|b-a|r^{n-1}
      +C|b-a|\int_{B_{2r}}|u|\,dx.
\]
The first part of the corollary absorbs the last term, and therefore
\[
 \int_{B_r\cap\{a<u<b\}}W\,dx
 \le C\bigl(r^{n-1}|b-a|+r^n\bigr).
\]

Finally, an ambient ball
$B_r^{n+1}((0,z_0))$ projects into $B_r$ and restricts the graph height to
$[z_0-r,z_0+r]$.  Applying the preceding estimate with a height interval
of length $2r$ gives
\[
 \mathcal H^n\bigl(\Sigma\cap B_r^{n+1}((0,z_0))\bigr)\le C_{z_0}r^n.
\]
\end{proof}

\section{Logarithmic profiles and finite-perimeter quantization}
\label{sec:log-profiles}
Fix an unbounded sign $\varepsilon\in\{+1,-1\}$, retain the
working ball $B_{r_0}$ from the tail analysis, and set
\[
 \Omega_t=tB_{r_0},\qquad
 w_t(y)=t^{-1}u(y/t),\qquad
 Z_t=\frac{Dw_t}{\sqrt{t^{-4}+|Dw_t|^2}},\qquad
 Z_t^\varepsilon=\varepsilon Z_t.
\]
The fields $w_t,Z_t$ are used on $\Omega_t\setminus\{0\}$; the
logarithmic truncation below has compact support in $\Omega_t$ and is
extended by zero to $\mathbb R^n$.
For $a>0$ define the globally Lipschitz truncation
\[
 \mathfrak l_a(s):=
 \begin{cases}
  0,&s\le a,\\[1mm]
  \log(s/a),&s>a,
 \end{cases}
\]
and put
\begin{equation}\label{eq:log-profile-def}
 \Phi_{a,t}^\varepsilon:=\mathfrak l_a(\varepsilon w_t).
\end{equation}
Thus the logarithm is taken only on the region $\{\varepsilon w_t>a\}$;
in particular \eqref{eq:log-profile-def} is defined on all of $\Omega_t$.
The parameter $a$ is retained because the proof later uses arbitrary
positive finite value bands.

\subsection{Exact logarithmic ledgers}
\begin{proposition}\label{prop:log-ledgers}
For every $a>0$, every $\lambda\ge0$, and every finite $p>0$,
\begin{align}
 |\{\Phi_{a,t}^\varepsilon>\lambda\}|
 &\longrightarrow C_\varepsilon a^{-n}e^{-n\lambda},\label{eq:log-distribution}\\
 \|\Phi_{a,t}^\varepsilon\|_{L^p}^p
 &\longrightarrow \frac{\Gamma(p+1)}{n^p}C_\varepsilon a^{-n}.\label{eq:log-Lp}
\end{align}
Moreover $\Phi_{a,t}^\varepsilon\in BV(\mathbb R^n)$ after zero extension,
\begin{align}
 |D\Phi_{a,t}^\varepsilon|(\mathbb R^n)
 &=t^{n-1}\int_{at}^{\infty}\frac{P_\varepsilon(s)}s\,ds,\label{eq:log-BV}\\
 \int(Z_t^\varepsilon,D\Phi_{a,t}^\varepsilon)
 &=a\int \Phi_{a,t}^\varepsilon e^{\Phi_{a,t}^\varepsilon}\,dy
 =t^{n-1}\int_{at}^{\infty}\frac{J_\varepsilon(s)}s\,ds.\label{eq:log-pairing}
\end{align}
Finally
\begin{equation}\label{eq:log-defect}
0\le |D\Phi_{a,t}^\varepsilon|-
\int(Z_t^\varepsilon,D\Phi_{a,t}^\varepsilon)
\le Ct^{-2},
\end{equation}
and
\begin{equation}\label{eq:log-B-limit}
 |D\Phi_{a,t}^\varepsilon|(\mathbb R^n)
 \longrightarrow B_{\varepsilon,a}:=
 \frac{nC_\varepsilon}{(n-1)^2}a^{1-n}.
\end{equation}
\end{proposition}
\begin{proof}
The superlevel identity is exact:
\[
 \{\Phi_{a,t}^\varepsilon>\lambda\}
 =tE^\varepsilon_{ae^\lambda t}.
\]
The critical tail law and a decreasing-tail envelope give a uniform
majorant $Ce^{-n\lambda}$, hence layer cake and dominated convergence
prove \eqref{eq:log-distribution}--\eqref{eq:log-Lp}.

For $M<\infty$ put $\Phi_M=\min\{\Phi_{a,t}^\varepsilon,M\}$.
Coarea on the punctured rescaled domain gives
\[
 |D\Phi_M|
 =t^{n-1}\int_{at}^{ae^Mt}\frac{P_\varepsilon(s)}s\,ds.
\]
Multiplying by a cutoff that vanishes on $B_\eta$ creates at most
$CM\eta^{n-1}$ variation; this tends to zero for every $n\ge2$.  Thus no
variation atom is created at the puncture.  Letting $M\to\infty$ proves
\eqref{eq:log-BV}.

The same bounded truncation makes the Anzellotti pairing legitimate.  The
puncture integration-by-parts error is again $O(M\eta^{n-1})$.  On the
support of the logarithmic profile, $\varepsilon w_t=ae^{\Phi}$, so
\[
 \int(Z_t^\varepsilon,D\Phi_M)
 =-\int\Phi_M\,\operatorname{div}Z_t^\varepsilon
 \longrightarrow a\int\Phi e^\Phi.
\]
Coarea on level sets gives the last expression in
\eqref{eq:log-pairing}.  Since $0\le P-J\le -V'$, the difference of
\eqref{eq:log-BV} and \eqref{eq:log-pairing} is bounded by
\[
 \frac{t^{n-1}}{at}V_\varepsilon(at)
 =a^{-1}t^{n-2}V_\varepsilon(at)=O_{a}(t^{-2}).
\]
Finally, with $s=t\sigma$ and
$J_\varepsilon(t\sigma)\sim\frac{nC_\varepsilon}{n-1}
(t\sigma)^{1-n}$,
\[
 t^{n-1}\int_{at}^\infty\frac{J_\varepsilon(s)}s\,ds
 \longrightarrow
 \frac{nC_\varepsilon}{(n-1)^2}a^{1-n},
\]
which proves \eqref{eq:log-B-limit}.
\end{proof}

\subsection{Quantitative profile decomposition}
The following concentration--compactness lemma is purely $BV$ and will be
used with $f_j=\Phi_{a,t_j}^\varepsilon$.  Its only integrability
requirement is $p>2$; there is no restriction $p<n$.
\begin{lemma}
\label{lem:quantitative-bv-profiles}
Let $p>2$ and let $f_j\geq0$ satisfy
\[
\sup_j\bigl(\|f_j\|_{BV(\mathbb R^n)}+\|f_j\|_{L^p(\mathbb R^n)}
                   +|\{f_j>0\}|\bigr)<\infty.
\]
Here $\|f\|_{BV(\mathbb R^n)}:=\|f\|_{L^1(\mathbb R^n)}
+|Df|(\mathbb R^n)$.
After passage to a subsequence there are at most countably many nonzero
$f^a\in BV(\mathbb R^n)\cap L^p(\mathbb R^n)$, translation sequences
$z_j^a$, integers
$A_j\to A_\infty\in\{0,1,2,\ldots\}\cup\{\infty\}$, and mutually disjoint
balls $B_{\rho_j^a}(z_j^a)$, $a\leq A_j$, with $\rho_j^a\to\infty$ for
each fixed $a$.  Define
\[
 \pi_j^a=f_j\mathbf1_{B_{\rho_j^a}(z_j^a)},
 \qquad
 \mathfrak r_j=f_j\mathbf1_{\mathbb R^n\setminus
                  \bigcup_{a\leq A_j}B_{\rho_j^a}(z_j^a)}.
\]
Then
\begin{equation}\label{eq:summed-global-profile-errors}
\begin{aligned}
 &\sum_{a\leq A_j}
 \bigl(\|\pi_j^a(\,\cdot+z_j^a)-f^a\|_{L^1}
       +\|\pi_j^a(\,\cdot+z_j^a)-f^a\|_{L^2}\bigr)\\
 &\hspace{28mm}+\|\mathfrak r_j\|_{L^1}+\|\mathfrak r_j\|_{L^2}\longrightarrow0,
\end{aligned}
\end{equation}
and the balls can be chosen simultaneously so that
\begin{equation}\label{eq:cutting-upper-ledger}
 \sum_{a\leq A_j}|D\pi_j^a|(\mathbb R^n)+|D\mathfrak r_j|(\mathbb R^n)
 \leq |Df_j|(\mathbb R^n)+o(1).
\end{equation}
\end{lemma}

\begin{proof}
The proof has four steps.

\emph{Vanishing criterion.}
Tile $\mathbb R^n$ by unit cubes $Q_z=z+[0,1)^n$ and let $Q_z^*$ be fixed
concentric enlargements with bounded overlap. Put
\[
 \mathfrak q_{\rm cc}(g):=\sup_{z\in\mathbb Z^n}\int_{Q_z}|g|.
\]
The local $BV$ Sobolev inequality and interpolation on each cube give
\begin{equation}\label{eq:cubevanishing}
 \|g\|_{L^{(n+1)/n}}^{(n+1)/n}
 \le C \mathfrak q_{\rm cc}(g)^{1/n}\bigl(|Dg|(\mathbb R^n)+\|g\|_{L^1}\bigr).
\end{equation}
Hence, under the bounds of the lemma,
$\mathfrak q_{\rm cc}(g_j)\to0$ implies $g_j\to0$ in $L^{(n+1)/n}$. Interpolation with the
uniform $L^p$ bound gives $L^2$ convergence, and the support-measure bound
gives $L^1$ convergence.

\emph{Extraction of mutually escaping profiles.}
Assume that the preceding vanishing alternative does not hold. Choose a unit
cube carrying a positive limiting amount of $f_j$-mass, translate it to the
origin, and use local $BV$ compactness and interpolation to obtain a nonzero
profile $f^1$.  Having selected $z_j^1,\dots,z_j^A$, define
\[
 \Lambda_A:=\sup_{\{\xi_j\}}
 \left\{\limsup_{j\to\infty}\int_{Q_{\xi_j}}f_j:
 |\xi_j-z_j^a|\to\infty\ \text{for every }1\le a\le A\right\}.
\]
If $\Lambda_A=0$ the extraction stops.  Otherwise choose cube centers
$\xi_j$ in the defining class for which the displayed $\limsup$ is at
least $\Lambda_A/2$, pass to a subsequence on which the corresponding
cube masses have a positive limit, and set $z_j^{A+1}:=\xi_j$.
Local $BV$ compactness then extracts the next nonzero profile.
Thus
\[
 |z_j^a-z_j^b|\to\infty\qquad(a\ne b),
\]
and for every fixed $a$,
\[
 f_j(\,\cdot+z_j^a)\to f^a
 \quad\text{in }L^1_{\rm loc}\cap L^2_{\rm loc}.
\]
The disjointness of fixed-radius balls around finitely many centers and the
uniform $L^2$ bound imply $\sum_A\Lambda_A^2<\infty$; hence
$\Lambda_A\to0$ if infinitely many profiles are extracted.

\emph{Simultaneous low-cost cutting.}
We make the diagonal choice explicit.  Suppose first that
$A_\infty=\infty$.  For each integer $m\ge1$, choose $L_m\ge m^2$, with
$L_m\uparrow\infty$, so large that
\[
 \sum_{a=1}^m\Bigl(
 \|f^a\|_{L^1(\mathbb R^n\setminus B_{L_m})}
 +\|f^a\|_{L^2(\mathbb R^n\setminus B_{L_m})}\Bigr)<\frac1m.
\]
For this fixed $m$, mutual escape of the first $m$ centers and their local
$L^1\cap L^2$ convergence allow us to choose $J_m$, strictly increasing in
$m$, so that for every $j\ge J_m$ and every $m\ge2$,
\[
 4L_m<\min_{1\le a<b\le m}|z_j^a-z_j^b|,
\]
with the separation condition understood as vacuous for $m=1$, and
\[
 \sum_{a=1}^m\Bigl(
 \|f_j(\cdot+z_j^a)-f^a\|_{L^1(B_{2L_m})}
 +\|f_j(\cdot+z_j^a)-f^a\|_{L^2(B_{2L_m})}\Bigr)<\frac1m.
\]
For $J_m\le j<J_{m+1}$ set $A_j=m$ and $L_j=L_m$.  Then
$A_j\to\infty$, $L_j\to\infty$, $A_j/L_j\to0$, and the displayed
separation, local-error, and profile-tail bounds all tend to zero.  If
$A_\infty<\infty$, keep $A_j=A_\infty$ for all large $j$ and choose
$L_j\uparrow\infty$ sufficiently slowly so that the same three bounds hold;
then again $A_j/L_j\to0$.

For each $a\le A_j$, radial Fubini and the uniform $L^1$ bound give a radius
$\rho_j^a\in[L_j,2L_j]$ for which
\[
 \int_{\partial B_{\rho_j^a}(z_j^a)}f_j^*\,d\mathcal H^{n-1}
 \le \frac{C}{L_j},
 \qquad
 |Df_j|(\partial B_{\rho_j^a}(z_j^a))=0.
\]
The selected balls are mutually disjoint. Defining $\pi_j^a$ and $\mathfrak r_j$ as in
the statement, the $BV$ product and trace formulas yield
\[
 \sum_{a\le A_j}|D\pi_j^a|+|D\mathfrak r_j|
 \le |Df_j|+C\frac{A_j}{L_j}
 =|Df_j|+o(1),
\]
which is \eqref{eq:cutting-upper-ledger}. Since
$L_j\le\rho_j^a\le2L_j$, the local-error bound on $B_{2L_j}$ controls the
part inside the cutting ball, while the profile-tail bound outside
$B_{L_j}$ controls the omitted part. Consequently
\[
 \sum_{a\le A_j}
 \bigl(\|\pi_j^a(\cdot+z_j^a)-f^a\|_{L^1}
      +\|\pi_j^a(\cdot+z_j^a)-f^a\|_{L^2}\bigr)\to0.
\]

\emph{The remainder.}
We claim $q(\mathfrak r_j)\to0$. Otherwise there are $\delta>0$ and unit cubes
$Q_{\xi_j}$ with $\int_{Q_{\xi_j}}\mathfrak r_j\ge\delta$. Since $\mathfrak r_j$ vanishes on
every selected ball and $\rho_j^a\to\infty$, the centers $\xi_j$ escape from
each fixed family $z_j^1,\dots,z_j^A$. By maximality of the extraction,
$\delta\le\Lambda_A$ for every fixed $A$. This contradicts either the
stopping condition in the finite-profile case or $\Lambda_A\to0$ in the
infinite-profile case. The vanishing criterion therefore gives
$\|\mathfrak r_j\|_{L^1}+\|\mathfrak r_j\|_{L^2}\to0$, proving
\eqref{eq:summed-global-profile-errors}.
\end{proof}

We record the consequences of the decomposition that will be used below.
By definition and disjointness of the balls,
\begin{gather}
 f_j=\sum_{a\leq A_j}\pi_j^a+\mathfrak r_j\quad\text{a.e.},
 \qquad \pi_j^a\pi_j^b=0\ (a\ne b),\quad \pi_j^a\mathfrak r_j=0,
 \label{eq:pointwise-profile-split}\\
 |z_j^a-z_j^b|\longrightarrow\infty\quad(a\ne b),
 \qquad \rho_j^a\longrightarrow\infty.
 \label{eq:profile-separation}
\end{gather}
Here $f^a$ denotes the $a$-th translated profile limit, characterized by
$\pi_j^a(\cdot+z_j^a)\to f^a$ in $L^1\cap L^2$.
Pointwise disjointness and \eqref{eq:summed-global-profile-errors} give
the exact scalar ledgers
\begin{align}
 \lim_j\|f_j\|_{L^1}&=\sum_{a=1}^{A_\infty}\|f^a\|_{L^1},
 \label{eq:profile-L1-ledger}\\
 \lim_j\|f_j\|_{L^2}^2&=\sum_{a=1}^{A_\infty}\|f^a\|_{L^2}^2.
 \label{eq:profile-L2-ledger}
\end{align}
Finally, lower semicontinuity and \eqref{eq:cutting-upper-ledger} imply
\begin{equation}\label{eq:profile-variation-lower-ledger}
 \sum_{a=1}^{A_\infty}|Df^a|(\mathbb R^n)
 +\liminf_j|D\mathfrak r_j|(\mathbb R^n)
 \leq\liminf_j|Df_j|(\mathbb R^n).
\end{equation}

Apply the lemma to $f_j=\Phi_{a,t_j}^{\varepsilon}$, where
$t_j\to\infty$ is arbitrary.  Put
\[
 \rho_{\varepsilon}(s):=\sup\{|x|:x\in E_s^{\varepsilon}\}.
\]
By construction, the unit extraction cube based at each $z_j^\alpha$ has
positive $\Phi_{a,t_j}^{\varepsilon}$-mass; it therefore meets
\[
 \{\Phi_{a,t_j}^{\varepsilon}>0\}=t_jE_{at_j}^{\varepsilon}.
\]
Since the diameter of the cube is fixed,
\begin{equation}\label{eq:centerlocation}
 |z_j^\alpha|\leq t_j\rho_{\varepsilon}(at_j)+O(1)=o(t_j).
\end{equation}
Consequently the translated natural domains exhaust $\mathbb R^n$.
The translated puncture is the point $-z_j^\alpha$; it is not assumed to
escape to infinity.

\subsection{Signed source compactness and exact profile calibration}
\begin{lemma}\label{lem:log-signed-source}
Let $t_j\to\infty$ and apply Lemma~\ref{lem:quantitative-bv-profiles} to
$f_j=\Phi_{a,t_j}^\varepsilon$.  After a common subsequence, every
nonzero profile $\phi^\alpha$ has a field
$\mathcal Z^\alpha\in L^\infty_{\rm loc}(\mathbb R^n;\mathbb R^n)$ and a source
\(w^\alpha\in L^1_{\rm loc}(\mathbb R^n)\) with
$|\mathcal Z^\alpha|\le1$ and
\[
 \operatorname{div}\mathcal Z^\alpha=-w^\alpha,
 \qquad
 w^\alpha=ae^{\phi^\alpha}
 \quad\text{a.e.\ on }\{\phi^\alpha>0\}.
\]
Moreover
\begin{equation}\label{eq:log-profile-lower}
 |D\phi^\alpha|(\mathbb R^n)
 \ge a\int\phi^\alpha e^{\phi^\alpha}\,dy.
\end{equation}
In fact equality holds for every profile, the residual variation tends to
zero, and each localized profile converges strictly in $BV$.
\end{lemma}
\begin{proof}
The critical tail bounds for both signs give, whenever the opposite sign is
unbounded,
\[
 \int_{\{(w_t)^\pm>M\}}(w_t)^\pm\,dy\le CM^{1-n}.
\]
If the opposite sign is bounded near the puncture its rescaled negative
part is $O(t^{-1})$ on compact sets.  Hence the translated sources are uniformly integrable.  By
Banach--Alaoglu the translated flux fields converge weak-star locally in
$L^\infty$, while Dunford--Pettis gives weak $L^1_{\rm loc}$ convergence
of the translated sources.  To remove the translated puncture, test with
a radial cutoff that vanishes on $B_\eta$ and equals one outside
$B_{2\eta}$.  The flux-cutoff term is $O(\eta^{n-1})$ because
$|Z_t|\le1$, and the source contribution on $B_{2\eta}$ tends to zero
uniformly as $\eta\downarrow0$ by uniform integrability.  First letting
$j\to\infty$ and then $\eta\downarrow0$ therefore gives
$\operatorname{div}\mathcal Z^\alpha=-w^\alpha$ on all of $\mathbb R^n$.  On
$\{\phi^\alpha>0\}$, almost-everywhere profile convergence and uniform
integrability identify $w^\alpha=ae^{\phi^\alpha}$.

It remains to justify the global calibration inequality.  Let
$T_M(s)=\min\{s,M\}$ for $s\ge0$, and choose
$\chi_R\in C_c^1(B_{2R})$ with $\chi_R=1$ on $B_R$ and
$|D\chi_R|\le C/R$.  Since $|\mathcal Z^\alpha|\le1$, the Anzellotti pairing gives
\[
 |DT_M(\phi^\alpha)|(\mathbb R^n)
 \ge \int \chi_R\,(\mathcal Z^\alpha,DT_M(\phi^\alpha)).
\]
Using $\operatorname{div}\mathcal Z^\alpha=-w^\alpha$ to integrate the right-hand
side by parts, the outer-cutoff error is bounded by
\[
 \frac{C}{R}\int_{B_{2R}\setminus B_R}\phi^\alpha\,dy\longrightarrow0,
\]
because $\phi^\alpha\in L^1(\mathbb R^n)$.  Letting $R\to\infty$ and
then $M\to\infty$, and using $w^\alpha=ae^{\phi^\alpha}$ on
$\{\phi^\alpha>0\}$, gives \eqref{eq:log-profile-lower}.

It remains to close the nonlinear source budget.  This step is where the
logarithmic variable replaces the old $L^1+L^2$ algebraic ledger.  By the
pointwise disjoint decomposition,
\[
 |\{f_j>\lambda\}|
 =\sum_{\alpha\le A_j}|\{\pi_j^\alpha>\lambda\}|
  +|\{\mathfrak r_j>\lambda\}|.
\]
Since
\[
 \int_0^\infty|\{\mathfrak r_j>\lambda\}|\,d\lambda=\|\mathfrak r_j\|_{L^1}\to0,
\]
after one scalar subsequence the residual superlevel measure tends to zero
for almost every $\lambda>0$.  For every fixed profile $\alpha$, the
$L^1$ convergence
$\pi_j^\alpha(\cdot+z_j^\alpha)\to\phi^\alpha$ similarly gives
\[
 |\{\pi_j^\alpha>\lambda\}|\to
 |\{\phi^\alpha>\lambda\}|
\]
for almost every $\lambda$.

To pass from finitely many profiles to the whole family, fix such a
$\lambda>0$ and an integer $A$.  Chebyshev gives
\[
 \sum_{\alpha=A+1}^{A_j}|\{\pi_j^\alpha>\lambda\}|
 \le \lambda^{-1}
 \sum_{\alpha=A+1}^{A_j}\|\pi_j^\alpha\|_{L^1},
\]
while
\[
 \sum_{\alpha>A}|\{\phi^\alpha>\lambda\}|
 \le \lambda^{-1}
 \sum_{\alpha>A}\|\phi^\alpha\|_{L^1}.
\]
The exact $L^1$ ledger shows that the first right-hand side converges to
the second profile tail.  First let $j\to\infty$ with $A$ fixed and then
let $A\to A_\infty$.  Using the exact distribution
\eqref{eq:log-distribution} of $f_j=\Phi_{a,t_j}^\varepsilon$ gives, for
almost every $\lambda>0$,
\begin{equation}\label{eq:log-level-exhaustion}
 \sum_\alpha |\{\phi^\alpha>\lambda\}|
 =C_\varepsilon a^{-n}e^{-n\lambda}.
\end{equation}
Thus neither a diffuse residual nor an infinite profile tail can carry any
part of the limiting logarithmic distribution.  Since
\[
 \int\phi e^\phi
 =\int_0^\infty(1+\lambda)e^\lambda
   |\{\phi>\lambda\}|\,d\lambda,
\]
Tonelli and \eqref{eq:log-level-exhaustion} give
\[
 a\sum_\alpha\int\phi^\alpha e^{\phi^\alpha}
 =\frac{nC_\varepsilon}{(n-1)^2}a^{1-n}=B_{\varepsilon,a}.
\]
The variation upper ledger from
Lemma~\ref{lem:quantitative-bv-profiles} therefore gives
\[
 B_{\varepsilon,a}\ge
 \sum_\alpha|D\phi^\alpha|+\liminf|D\mathfrak r_j|
 \ge a\sum_\alpha\int\phi^\alpha e^{\phi^\alpha}
 =B_{\varepsilon,a}.
\]
Every inequality is an equality.  Hence $|D\mathfrak r_j|\to0$ and every profile is
exactly calibrated.  If one localized profile retained a positive
variation overshoot, lower semicontinuity for finitely many other profiles
and the cutting upper ledger would yield
$B_{\varepsilon,a}\ge B_{\varepsilon,a}+\delta$, a contradiction.  Thus
the localized profiles converge strictly in $BV$.
\end{proof}

\subsection{Finite-value stress passage}
\begin{lemma}\label{lem:point-CMC-removable}
Let $E\subset\mathbb R^n$, $n\ge2$, have finite measure and finite
perimeter.  Suppose that for one point $p$ and one constant $s>0$,
\[
 \int_{\partial^*E}\operatorname{div}_{\partial^*E}X\,d\mathcal H^{n-1}
 =s\int_{\partial^*E}X\cdot\nu_E\,d\mathcal H^{n-1}
\]
for every $X\in C_c^1(\mathbb R^n\setminus\{p\};\mathbb R^n)$.  Then the
same identity holds for every $X\in C_c^1(\mathbb R^n;\mathbb R^n)$.
\end{lemma}
\begin{proof}
Let
\[
 \mu=\mathcal H^{n-1}\llcorner\partial^*E,
 \qquad
 A=(\Id-\nu_E\otimes\nu_E)\mu,
\]
and define the order-one distribution
\[
 \Lambda(X)=\int A:DX-s\int X\cdot\nu_E\,d\mu.
\]
By hypothesis $\operatorname{spt}\Lambda\subset\{p\}$.

We first show that no derivative of a Dirac mass can occur.  Suppose
$X(p)=0$, and choose $\chi_r=1$ on $B_r(p)$, supported in $B_{2r}(p)$,
with $|D\chi_r|\le C/r$.  Since $\Lambda$ is supported at $p$,
\[
 \Lambda(X)=\Lambda(\chi_rX).
\]
The mean-value estimate gives
$\sup_{B_{2r}(p)}|X|\le2r\|DX\|_\infty$.  Consequently
\[
 |\Lambda(X)|
 \le C(1+sr)\|DX\|_\infty\,\mu(B_{2r}(p)).
\]
The perimeter measure has no atom at a point, so
$\mu(B_{2r}(p))\to0$.  Hence $\Lambda(X)=0$ whenever $X(p)=0$.
Therefore there is a vector $b\in\mathbb R^n$ such that
\[
 \Lambda(X)=b\cdot X(p).
\]

It remains to kill this possible Dirac force.  Fix $e\in\mathbb R^n$ and
let $\eta_R=1$ on $B_R(p)$, vanish outside $B_{2R}(p)$, and satisfy
$|D\eta_R|\le C/R$.  Testing with $X=\eta_Re$ gives
\[
 \left|\int A:D(\eta_Re)\right|\le \frac{C|e|}{R}P(E)\longrightarrow0.
\]
By finite-perimeter Gauss--Green,
\[
 \int_{\partial^*E}\eta_Re\cdot\nu_E\,d\mu
 =\int_E e\cdot D\eta_R\,dx,
\]
whose absolute value is at most $C|e||E|/R\to0$.  Thus
$b\cdot e=0$ for every $e$, and hence $b=0$.  Therefore
$\Lambda\equiv0$, which is the global CMC identity.
\end{proof}

\begin{lemma}\label{lem:log-CMC-passage}
For every nonzero logarithmic profile and for almost every $s>a$, the
finite-perimeter set
\[
 F_s=\{ae^{\phi}>s\}
\]
satisfies
\begin{equation}\label{eq:log-CMC}
 \int_{\partial^*F_s}\operatorname{div}_{\partial^*F_s}X\,d\mathcal H^{n-1}
 =s\int_{\partial^*F_s}X\cdot\nu_s\,d\mathcal H^{n-1}
\end{equation}
for every $X\in C_c^1(\mathbb R^n;\mathbb R^n)$.
\end{lemma}
\begin{proof}
Let $t_j\to\infty$ be the sequence which produces the profile, let
$z_j=o(t_j)$ be its translation, and set
\[
 \Psi_j(y)=\Phi_{a,t_j}^{\varepsilon}(y+z_j),\qquad
 w_j(y)=\varepsilon w_{t_j}(y+z_j).
\]
On the positive profile region the limiting source is
$w=ae^{\phi}$.  Put
\[
 \mathfrak a_j=(t_j^{-4}+|Dw_j|^2)^{1/2},
\]
and introduce the rescaled stress
\begin{equation}\label{eq:finite-value-rescaled-stress}
 \mathbb T_j
 =\frac{Dw_j\otimes Dw_j}{\mathfrak a_j}
  -\left(\mathfrak a_j-\frac{w_j^2}{2}\right)\Id.
\end{equation}
Away from the translated puncture $\xi_j=-z_j$ one has
$\operatorname{div}\mathbb T_j=0$.  At a regular level $\{w_j=r\}$,
with $\nu_{j,r}$ the outer normal of $\{w_j>r\}$,
\begin{equation}\label{eq:finite-value-traction}
 \mathbb T_j\nu_{j,r}
 =\left(\frac{r^2}{2}-\frac{t_j^{-4}}{\mathfrak a_j}\right)
   \nu_{j,r}.
\end{equation}

\smallskip
\noindent\textbf{Step 1: simultaneous good physical levels.}
Choose a smooth bounded exhaustion $\mathcal O_m\Subset \mathcal O_{m+1}$ of $\mathbb R^n$
with $|D\phi|(\partial \mathcal O_m)=0$.  Local strict $BV$ convergence of
$\Psi_j$ gives
\[
 \Psi_j\to\phi\quad\hbox{in }L^1(\mathcal O_m),\qquad
 |D\Psi_j|(\mathcal O_m)\to|D\phi|(\mathcal O_m).
\]
For $\lambda>0$ put
\[
 \mathcal P_{j,m}(\lambda)=P(\{\Psi_j>\lambda\};\mathcal O_m),\qquad
 \mathcal P_m(\lambda)=P(\{\phi>\lambda\};\mathcal O_m).
\]
Coarea and lower semicontinuity imply
$\mathcal P_m\le\liminf_j\mathcal P_{j,m}$ a.e., while the strict convergence gives
$\int\mathcal P_{j,m}\to\int\mathcal P_m$.  Since
$\mathcal P_m\le\liminf_j\mathcal P_{j,m}$ a.e.,
$(\mathcal P_m-\mathcal P_{j,m})_+\to0$ a.e.; moreover
$0\le(\mathcal P_m-\mathcal P_{j,m})_+\le\mathcal P_m\in L^1(d\lambda)$.
Dominated convergence therefore gives
\[
 \int (\mathcal P_m-\mathcal P_{j,m})_+\,d\lambda\longrightarrow0.
\]
Using the convergence of the total integrals,
\[
 \int (\mathcal P_{j,m}-\mathcal P_m)_+\,d\lambda
 =\int(\mathcal P_{j,m}-\mathcal P_m)\,d\lambda
  +\int(\mathcal P_m-\mathcal P_{j,m})_+\,d\lambda
 \longrightarrow0.
\]
Therefore
\[
 \mathcal P_{j,m}\longrightarrow\mathcal P_m\qquad\hbox{in }L^1(d\lambda).
\]
After one scalar diagonal extraction, for every $m$ and for almost every
$\lambda>0$,
\begin{equation}\label{eq:finite-value-level-strict}
 \mathbf1_{\{\Psi_j>\lambda\}}\to\mathbf1_{\{\phi>\lambda\}}
 \quad\hbox{in }L^1(\mathcal O_m),\qquad
 P(\{\Psi_j>\lambda\};\mathcal O_m)\to P(\{\phi>\lambda\};\mathcal O_m).
\end{equation}
Since $r=ae^{\lambda}$, this gives one common full-measure set of physical
values $r>a$ on which the vector measures
$D\mathbf1_{\{w_j>r\}}$ converge strictly on every $\mathcal O_m$ to
$D\mathbf1_{F_r}$.  We also remove the countable union of the null sets of
critical values of the smooth functions $w_j$.

\smallskip
\noindent\textbf{Step 2: the prelimit finite-band stress identity.}
Fix two such values $a<s<R<\infty$ and a compactly supported
$C^1$ vector field $X$.  If $n\ge3$, remove the translated puncture by a
radial cutoff.  In physical variables the band corresponds, for each fixed
$j$, to the finite signed height interval $st_j<\varepsilon u<Rt_j$; no
bound uniform in $j$ is asserted.  The finite-band stress estimate gives,
with $j$ fixed before the puncture radius is sent to zero, a cutoff cost
bounded by
\[
 C_jt_j^{-1}\eta^{n-2}\longrightarrow0
 \qquad(\eta\downarrow0).
\]
If $n=2$, pass to a further subsequence so that either $|\xi_j|\to\infty$ or
$\xi_j\to\xi_\infty$.  In the first case a fixed compact support eventually misses
the puncture.  In the second case we first take
$X\in C_c^1(\mathbb R^2\setminus\{\xi_\infty\};\mathbb R^2)$, so again no cutoff
is needed.  Thus in all cases presently allowed the divergence theorem
applied to $\mathbb T_jX$ on $\{s<w_j<R\}$ yields
\begin{align}
 &\int_{\{s<w_j<R\}}
 \left[\left(\frac{Dw_j\otimes Dw_j}{\mathfrak a_j}-\mathfrak a_jI\right):DX
       +\frac{w_j^2}{2}\operatorname{div}X\right]dy \notag\\
 &\quad=\int_{\{w_j=s\}}
 \left(\frac{s^2}{2}-\frac{t_j^{-4}}{\mathfrak a_j}\right)
 X\cdot\nu_{j,s}\,d\mathcal H^{n-1}
 -\int_{\{w_j=R\}}
 \left(\frac{R^2}{2}-\frac{t_j^{-4}}{\mathfrak a_j}\right)
 X\cdot\nu_{j,R}\,d\mathcal H^{n-1}.
\label{eq:finite-value-prelimit-stress}
\end{align}
The elementary estimates
\begin{equation}\label{eq:finite-value-a-defect}
 0\le\mathfrak a_j-|Dw_j|\le t_j^{-2},\qquad
 0\le\frac{t_j^{-4}}{\mathfrak a_j}\le t_j^{-2}
\end{equation}
show that replacing the gradient part in the bulk by
\[
 \left(\frac{Dw_j\otimes Dw_j}{|Dw_j|}-|Dw_j|I\right):DX
\]
costs $o(1)$, and the two boundary corrections in
\eqref{eq:finite-value-prelimit-stress} also tend to zero because the
selected-level perimeters stay bounded.  For the source term we use only
the logarithmic-profile convergence, not an unproved strong convergence of
the full sources.  On $\{w_j>a\}$ one has $w_j=ae^{\Psi_j}$ and
$\Psi_j\to\phi$ in measure locally.  Since the selected values $s,R$
belong to the common good set,
\[
 \mathbf1_{\{s<w_j<R\}}\longrightarrow
 \mathbf1_{\{s<ae^\phi<R\}}
 \quad\text{in }L^1_{\rm loc}.
\]
Thus
\[
 \mathbf1_{\{s<w_j<R\}}w_j^2
 =\mathbf1_{\{s<w_j<R\}}a^2e^{2\Psi_j}
 \longrightarrow
 \mathbf1_{\{s<w<R\}}w^2
 \quad\text{in }L^1_{\rm loc},
\]
because these functions are bounded by $R^2$ on the fixed value band.
This is exactly the source convergence required in the bulk term.

\smallskip
\noindent\textbf{Step 3: passage of the tangential stress.}
For a selected physical level $r$ put
\[
 \mathcal A_j^{\rm tan}(r)=\int_{\partial^*\{w_j>r\}}
       \operatorname{div}_{\partial^*\{w_j>r\}}X\,d\mathcal H^{n-1},
 \qquad
 \mathcal B_j^{\rm nor}(r)=\int_{\partial^*\{w_j>r\}}X\cdot\nu_{j,r}\,d\mathcal H^{n-1}.
\]
Strict convergence in \eqref{eq:finite-value-level-strict} and
Reshetnyak continuity give $\mathcal A_j^{\rm tan}(r)\to\mathcal A^{\rm tan}(r)$ and $\mathcal B_j^{\rm nor}(r)\to\mathcal B^{\rm nor}(r)$ for
a.e.\ $r\in[s,R]$, where the same formulas with $F_r$ define $\mathcal A^{\rm tan},\mathcal B^{\rm nor}$.
Moreover
\[
 |\mathcal A_j^{\rm tan}(r)|\le C_XP(\{w_j>r\};\mathcal O_m),
\]
and the perimeter densities converge in $L^1([s,R])$.  They are therefore
uniformly integrable, and Vitali's theorem gives
\[
 \int_s^R\mathcal A_j^{\rm tan}(r)\,dr\longrightarrow\int_s^R\mathcal A^{\rm tan}(r)\,dr.
\]
Coarea applied to the gradient part of
\eqref{eq:finite-value-prelimit-stress} consequently yields the limit
identity
\begin{equation}\label{eq:finite-value-integrated-limit}
 -\int_s^R\mathcal A^{\rm tan}(r)\,dr
 +\frac12\int_{F_s\setminus F_R}w^2\operatorname{div}X\,dy
 =\frac{s^2}{2}\mathcal B^{\rm nor}(s)-\frac{R^2}{2}\mathcal B^{\rm nor}(R).
\end{equation}

\smallskip
\noindent\textbf{Step 4: differentiation in the value parameter.}
Finite-perimeter Gauss--Green and layer cake give the exact identity
\begin{equation}\label{eq:finite-value-layercake-identity}
 \frac12\int_{F_s\setminus F_R}w^2\operatorname{div}X\,dy
 -\frac{s^2}{2}\mathcal B^{\rm nor}(s)+\frac{R^2}{2}\mathcal B^{\rm nor}(R)
 =\int_s^R r\mathcal B^{\rm nor}(r)\,dr.
\end{equation}
Comparing \eqref{eq:finite-value-integrated-limit} and
\eqref{eq:finite-value-layercake-identity} gives
\[
 \int_s^R\bigl(\mathcal A^{\rm tan}(r)-r\mathcal B^{\rm nor}(r)\bigr)\,dr=0
\]
for every pair $s<R$ in the common full-measure set.  Hence
$\mathcal A^{\rm tan}(r)=r\mathcal B^{\rm nor}(r)$ for almost every $r>a$.  Taking a countable dense family of
compactly supported test fields and then using continuity in the $C^1$
norm gives \eqref{eq:log-CMC} simultaneously for every test field.

In the planar case with $\xi_j\to\xi_\infty$ the preceding argument has so far
proved \eqref{eq:log-CMC} for test fields supported away from $\xi_\infty$.
Lemma~\ref{lem:point-CMC-removable} removes this one-point defect and
completes the proof.
\end{proof}

\subsection{Equal-ball quantization and exhaustion}
\begin{proposition}\label{prop:profile-quantization}
For every sequence $t_j\to\infty$, after a subsequence there are finitely
many nonzero logarithmic profiles.  For almost every $s>a$, each
$F_s^\alpha=\{ae^{\phi^\alpha}>s\}$ is a finite union of essentially
disjoint balls of radius $k/s$, and
\begin{equation}\label{eq:quantization}
 C_\varepsilon=m_\varepsilon\omega_nk^n,
 \qquad m_\varepsilon\in\mathbb N.
\end{equation}
Moreover the union of all profile balls exhausts the prelimit superlevel in
symmetric difference.
\end{proposition}
\begin{proof}
The extracted family is initially at most countable.  For each profile
$\phi^\alpha$, Lemma~\ref{lem:log-CMC-passage} supplies a full-measure
set $G_\alpha\subset(a,\infty)$ on which the CMC identity holds.  Intersect
these sets with the coarea full-measure sets on which
$|F_s^\alpha|<\infty$ and $P(F_s^\alpha)<\infty$, and then take the
countable intersection
\[
 G:=\bigcap_\alpha G_\alpha.
\]
For every $s\in G$ and every profile index, the finite-perimeter
Alexandrov theorem of Delgadino--Maggi \cite[Theorem~\refnum{1}]{DelgadinoMaggi2019} applies to
$F_s^\alpha$.
Since the boundary dimension is $k=n-1$, the scalar mean-curvature
identity with coefficient $s$ forces every component to be a ball of
radius $k/s$.  Thus
\begin{equation}\label{eq:profile-ball-decomposition}
 F_s^\alpha
 =\bigcup_{h=1}^{m_\alpha(s)}B_{k/s}(c_h^\alpha(s))
 \qquad\text{modulo null sets},
\end{equation}
with pairwise essentially disjoint interiors.

The logarithmic distribution is exhausted exactly by the profiles.  With
$\lambda=\log(s/a)$, equation \eqref{eq:log-level-exhaustion} becomes
\begin{equation}\label{eq:profile-volume-exhaustion-physical-level}
 \sum_\alpha|F_s^\alpha|=C_\varepsilon s^{-n}
 \qquad\text{for a.e.\ }s>a.
\end{equation}
Since each radius-$k/s$ ball has volume
$\omega_nk^ns^{-n}$, we obtain
\begin{equation}\label{eq:total-ball-count}
 \sum_\alpha m_\alpha(s)
 =\frac{C_\varepsilon}{\omega_nk^n}
 \qquad\text{for a.e.\ }s>a.
\end{equation}
The left side is integer-valued.  Hence the quotient on the right is an
integer, denoted $m_\varepsilon$, and
\[
 C_\varepsilon=m_\varepsilon\omega_nk^n.
\]
In particular the total number of ball components is finite for a.e.\ $s$.

We next show that components cannot be transferred between different
profiles as the level changes.  Fix $s<r$ in the common good set with
$r/s<2^{1/n}$.  Choose the open representatives in
\eqref{eq:profile-ball-decomposition}.  Nesting holds modulo null sets:
$F_r^\alpha\subset F_s^\alpha$ a.e.\  A connected inner ball is in fact
contained in one outer ball.  Indeed, if it had a point outside the union
of the closed outer balls, openness would give a positive-volume violation
of the a.e.\ inclusion.  If it met two distinct outer balls, those outer
balls have disjoint interiors and can meet only tangentially; near a
tangency point their union misses a full transverse wedge.  Since the
inner ball is open and connected, crossing from one outer ball to the
other would again create a positive-volume subset outside the outer union.
Thus every radius-$k/r$ component of the inner set lies in one
radius-$k/s$ component of the outer set.  Two distinct inner components
cannot lie in the same outer ball, because that would force
\[
 2\omega_n(k/r)^n\le\omega_n(k/s)^n,
\]
contrary to $(r/s)^n<2$.  Thus the component assignment is injective and
$m_\alpha(r)\le m_\alpha(s)$.  Any two good levels can be joined, after
arbitrarily small perturbations inside the full-measure set, by a finite
chain with successive ratios below $2^{1/n}$.  Since the total sum in
\eqref{eq:total-ball-count} is constant, each nonnegative integer-valued
$m_\alpha$ is a.e.\ constant.  Every nonzero profile has a positive superlevel of positive measure and
hence has at least one component on the common good set.  Thus only finitely
many profiles can be nonzero.  Relabel them as
\[
  \phi^1,\dots,\phi^N,
  \qquad
  N:=\#\{\alpha:\phi^\alpha\not\equiv0\}<\infty.
\]

It remains to recover the prelimit superlevel globally.  The disjoint
profile decomposition gives, for $\lambda=\log(s/a)$,
\[
 \{\Phi_{a,t_j}^\varepsilon>\lambda\}
 =\left(\bigcup_{\alpha\le A_j}\{\pi_j^\alpha>\lambda\}\right)
  \mathbin{\dot\cup}\{\mathfrak r_j>\lambda\}
 \quad\text{modulo null sets}.
\]
For each fixed profile,
$\|\pi_j^\alpha(\cdot+z_j^\alpha)-\phi^\alpha\|_{L^1}\to0$ implies, after
one scalar subsequence,
\[
 |\{\pi_j^\alpha>\lambda\}
   \mathbin\triangle(z_j^\alpha+\{\phi^\alpha>\lambda\})|
 \longrightarrow0
\]
for a.e.\ $\lambda$.  The residual contribution tends to zero because
$\|\mathfrak r_j\|_{L^1}\to0$, and the profile tail is controlled by the exact
$L^1$ ledger exactly as in the proof of
\eqref{eq:log-level-exhaustion}.  Consequently, for every $s$ in one
common full-measure set,
\begin{equation}\label{eq:profile-global-symmetric-exhaustion}
 \left|
 t_jE_{st_j}^\varepsilon
 \mathbin\triangle
 \bigcup_{\alpha=1}^{N}\left(z_j^\alpha+F_s^\alpha\right)
 \right|\longrightarrow0.
\end{equation}
The right-hand side is a union of exactly $m_\varepsilon$ essentially
disjoint balls of radius $k/s$.  This proves the quantization and the
symmetric-difference exhaustion.
\end{proof}

\section{Uniform one-ball geometry and translation charge}
\label{sec:one-ball}
For each unbounded sign we now pass from subsequential logarithmic profiles
to physical high levels.

\subsection{Stable matching}
\begin{lemma}
\label{lem:frag-stable-matching}
Let \(0<r<R\) and \(2r^n>R^n\).  Suppose
\[
\mathcal U_j^{\rm match}=\bigcup_{i=1}^m B_R(a_{j,i}),\qquad
\mathcal W_j^{\rm match}=\bigcup_{i=1}^m B_r(b_{j,i}),
\]
where distinct centers in the first family are separated by
\(2R-o(1)\), distinct centers in the second family are separated by
\(2r-o(1)\), and
\[
|\mathcal W_j^{\rm match}\setminus \mathcal U_j^{\rm match}|\longrightarrow0.
\]
Then, after relabeling,
\begin{equation}
|b_{j,i}-a_{j,i}|\leq R-r+o(1),
\qquad i=1,\dots,m.
\label{eq:frag-stable-displacement}
\end{equation}
\end{lemma}

\begin{proof}
Fix an inner index \(i\), translate by \(-b_{j,i}\), and pass to a
subsequence on which every bounded sequence \(a_{j,h}-b_{j,i}\) converges.
The translated inner ball satisfies
\[
 |B_r\setminus(\mathcal U_j^{\rm match}-b_{j,i})|
 \leq |\mathcal W_j^{\rm match}\setminus \mathcal U_j^{\rm match}|\longrightarrow0.
\]
The limiting radius-\(R\) balls have disjoint interiors and cover \(B_r\)
up to a null set.  If a point of \(B_r\) lay outside the union of their
closed balls, then, since only finitely many limiting balls occur, a whole
neighborhood of that point would remain uncovered, contradicting the a.e.
coverage.  Thus the closed limiting balls cover \(B_r\) pointwise.  This
cover is supplied by one ball.  Indeed, two distinct limiting balls can
meet only at a tangency point \(p\); if their centers are \(a_1,a_2\), then
for every unit vector \(\tau\perp(a_1-a_2)\) and every \(\rho>0\),
\[
 |p+\rho\tau-a_1|^2=|p+\rho\tau-a_2|^2=R^2+\rho^2>R^2.
\]
Thus two tangent balls do not cover any neighborhood of \(p\), whereas
all other limiting balls stay a positive distance from \(p\).  Connectedness
of \(B_r\) therefore gives \(B_r\subset\overline{B_R(a)}\) for one limiting
center \(a\), and
\[
 |a|+r\leq R,\qquad |a|\leq R-r.
\]
Repeating this for the finitely many inner balls gives an assignment to
outer balls.  It is injective: if two inner balls were assigned to the same
outer ball, asymptotic disjointness and containment would give
\[
 2\omega_n r^n\leq\omega_nR^n,
\]
contrary to \(2r^n>R^n\).  Hence the assignment is a permutation and,
after relabeling, \eqref{eq:frag-stable-displacement} holds on the chosen
subsequence.  Since the initial subsequence was arbitrary, a bad-subsequence
contradiction removes the extraction and proves the assertion for the full
sequence.
\end{proof}

\subsection{Uniform physical clusters and the two-ball obstruction}
\begin{proposition}\label{prop:frag-uniform-clusters}
For an unbounded sign $\varepsilon$ there are
$m_\varepsilon\in\mathbb N$, a function $\eta_\varepsilon(t)\downarrow0$,
and centers $a_i^\varepsilon(t)$ such that
\begin{align}
 t^n\left|E_t^\varepsilon\triangle
 \bigcup_{i=1}^{m_\varepsilon}B_{k/t}(a_i^\varepsilon(t))\right|
 &\le\eta_\varepsilon(t),\label{eq:frag-uniform-cluster}\\
 t|a_i^\varepsilon(t)-a_j^\varepsilon(t)|&\ge2k-\eta_\varepsilon(t).
 \label{eq:frag-uniform-separation}
\end{align}
Moreover
\begin{equation}\label{eq:frag-center-bound}
 t|a_i^\varepsilon(t)|\le k+o(1),
\end{equation}
and $m_\varepsilon\le2$.
\end{proposition}
\begin{proof}
Fix an arbitrary $t_j\to\infty$.  For an admissible $s>1$,
Proposition~\ref{prop:profile-quantization} approximates $t_jE_{st_j}^\varepsilon$
by $m_\varepsilon$ balls of radius $k/s$.  Enlarge them concentrically to
radius $k$.  Since $E_{st_j}^\varepsilon\subset E_{t_j}^\varepsilon$ and
both volumes are fixed by the critical law, the symmetric-difference error
with $t_jE_{t_j}^\varepsilon$ is at most
$2C_\varepsilon(1-s^{-n})+o(1)$.  Let $s\downarrow1$ and diagonalize.
Thus every sequence has a subsequence with an $o(1)$ radius-$k$ cluster
approximation.  A bad-sequence contradiction uniformizes this in $t$ and
gives a common error $e_\varepsilon(t)\to0$ in
\eqref{eq:frag-uniform-cluster}--\eqref{eq:frag-uniform-separation}.
Replacing it by the tail envelope
$\eta_\varepsilon(t):=\sup_{s\ge t}e_\varepsilon(s)$ gives the stated
$\eta_\varepsilon(t)\downarrow0$.

Fix $1<q<2^{1/n}$.  Apply the matching lemma to heights $t$ and $qt$.
After compatible labeling, and after replacing the resulting $o(1)$ error
by its tail supremum,
\[
 |a_i(qt)-a_i(t)|\le
 k\left(\frac1t-\frac1{qt}\right)+\frac{\delta_q(t)}t,
 \qquad \delta_q(t)\downarrow0.
\]
The physical centers tend to the puncture: otherwise a principal ball
would remain in a compact annulus where the actual high superlevel is
empty, contradicting the symmetric-difference estimate.  Iterating the
matching inequality along $t,q t,q^2t,\dots$ and telescoping yields
\eqref{eq:frag-center-bound}.  Any subsequential scaled centers lie in
$\overline B_k$ and are pairwise at distance at least $2k$; hence there
are at most two.  If there are two they are antipodal.
\end{proof}

\begin{proposition}\label{prop:frag-signed-compactness}
Let $t_j\to\infty$.  After passage to a subsequence there are
$w_\infty\in L^1_{\mathrm{loc}}(\mathbb R^n)$ and
$Z_\infty\in L^\infty_{\mathrm{loc}}(\mathbb R^n;\mathbb R^n)$ such that
\begin{equation}\label{eq:signed-natural-compactness}
 w_{t_j}\longrightarrow w_\infty
 \quad\text{strongly in }L^1_{\mathrm{loc}}(\mathbb R^n),
 \qquad
 Z_{t_j}\weakstar Z_\infty,
 \qquad
 \operatorname{div}Z_\infty=-w_\infty.
\end{equation}
For each sign $\varepsilon\in\{+1,-1\}$ whose tail is unbounded, there is
one common full-measure set of $s>0$ such that
\begin{equation}\label{eq:signed-cluster-limit-family}
 \{\varepsilon w_\infty>s\}
 =A_s^\varepsilon
 =\bigcup_{i=1}^{m_\varepsilon}B_{k/s}(c_i^\varepsilon(s))
 \quad\text{modulo null sets},
 \qquad
 |c_i^\varepsilon(s)|\le \frac{k}{s},
\end{equation}
with pairwise disjoint interiors.  If that sign is bounded near the
puncture, $A_s^\varepsilon=\varnothing$.  Moreover, for every
$\delta,M>0$,
\begin{align}
 (Z_\infty,D T_M((w_\infty-\delta)_+))
 &=|D T_M((w_\infty-\delta)_+)|,
 \label{eq:signed-positive-calibration}\\
 (-Z_\infty,D T_M((-w_\infty-\delta)_+))
 &=|D T_M((-w_\infty-\delta)_+)|
 \label{eq:signed-negative-calibration}
\end{align}
as local Radon measures.  In addition, for each sign
$\varepsilon\in\{+1,-1\}$, each $\delta,M>0$, and every
$U\Subset\mathbb R^n$ satisfying
$|D T_M((\varepsilon w_\infty-\delta)_+)|(\partial U)=0$,
\begin{equation}\label{eq:signed-prelimit-strict-bv}
 T_M((\varepsilon w_{t_j}-\delta)_+)
 \longrightarrow T_M((\varepsilon w_\infty-\delta)_+)
 \qquad\hbox{strictly in }BV(U).
\end{equation}
The same conclusion remains true after replacing
$\varepsilon w_{t_j}$ by $\varepsilon w_{t_j}-c_j$ for any constants
$c_j\to0$, with the limiting truncation unchanged.
\end{proposition}

\begin{proof}
\emph{Step 1: simultaneous signed level-set compactness.}
For every positive rational $s$ and every unbounded sign, apply
Proposition~\ref{prop:frag-uniform-clusters} at the physical height $st_j$.  After
rescaling by $t_j$, the approximating balls have radius $k/s$, while
\eqref{eq:frag-center-bound} gives
\[
 |t_j a_i^\varepsilon(st_j)|\le \frac{k+o(1)}{s}.
\]
A diagonal extraction over the positive rationals and the finitely many
centers at each rational level therefore gives global symmetric-difference
limits $A_s^\varepsilon$ of the form displayed in
\eqref{eq:signed-cluster-limit-family}.  The nesting of the physical tail
sets and the disjointness of the positive and negative phases pass to the
limit.  Enlarging the selected set by any prescribed countable family of
continuity levels and then taking the unique right-continuous nested
representatives gives the same families for almost every $s>0$.  If one
sign is bounded near the puncture, its tail sets are empty at all
sufficiently large physical heights, so its limiting family is empty.

Define
\[
 w_\infty(y)
 :=\sup\{s>0:y\in A_s^+\}
   -\sup\{s>0:y\in A_s^-\}.
\]
For every compact $K\Subset\mathbb R^n$, the critical tail estimates give
\begin{equation}\label{eq:signed-tail-uniform-integrability}
 |K\cap\{|w_{t_j}|>s\}|\le \min\{|K|,Cs^{-n}\}
 \qquad(s>0),
\end{equation}
where a bounded sign contributes only $o(1)$ on fixed compact sets.
Convergence of the positive and negative level sets on a dense countable
family implies $w_{t_j}\to w_\infty$ in measure on $K$.  Since $n>1$,
\eqref{eq:signed-tail-uniform-integrability} gives uniform integrability:
\[
 \sup_j\int_{K\cap\{|w_{t_j}|>M\}}|w_{t_j}|\,dy
 \le C M^{1-n}\longrightarrow0.
\]
Vitali's theorem proves the strong local $L^1$ convergence in
\eqref{eq:signed-natural-compactness}.  The same cluster approximation and
center bound make every fixed selected level globally tight, so the
symmetric-difference convergence to $A_s^\varepsilon$ is global.

\emph{Step 2: the limiting divergence equation.}
Weak-star compactness gives, after a further subsequence,
$Z_{t_j}\weakstar Z_\infty$ and $|Z_\infty|\le1$.  On the punctured
rescaled domains,
\[
 \operatorname{div}Z_{t_j}=-w_{t_j}.
\]
Insert a radial cutoff which vanishes on $B_r$ and equals one outside
$B_{2r}$.  Since $|Z_{t_j}|\le1$, the cutoff flux is $O(r^{n-1})$,
uniformly in $j$.  The source error tends to zero by the strong local
$L^1$ convergence.  First let $j\to\infty$ and then $r\downarrow0$.
No point source survives, and
$\operatorname{div}Z_\infty=-w_\infty$ on all of $\mathbb R^n$.

\emph{Step 3: bounded-truncation calibration.}
On the smooth punctured rescaled domain one has the pointwise defect
\begin{equation}\label{eq:signed-pointwise-calibration-defect}
 0\le |Dw_{t_j}|-Z_{t_j}\cdot Dw_{t_j}\le t_j^{-2},
\end{equation}
and the same estimate holds for $(-w_{t_j},-Z_{t_j})$.  Fix
$\delta,M>0$ and approximate $T_M((\,\cdot-\delta)_+)$ by smooth
nondecreasing $1$-Lipschitz truncations.  Testing
\eqref{eq:signed-pointwise-calibration-defect} with a compact spatial
cutoff and a puncture cutoff gives, on every compact set, a uniform $BV$
bound for $T_M((w_{t_j}-\delta)_+)$.  The puncture term is
$O(Mr^{n-1})$ and therefore vanishes for every $n\ge2$.
Strong local $L^1$ convergence fixes the $BV$ limit.  Passing to the limit
in the bounded Anzellotti identity and using
$\operatorname{div}Z_\infty=-w_\infty$ gives, for every nonnegative
compactly supported $\chi$,
\[
 \int\chi\,|DT_M((w_\infty-\delta)_+)|
 \le \int\chi\,(Z_\infty,DT_M((w_\infty-\delta)_+))
 \le \int\chi\,|DT_M((w_\infty-\delta)_+)|.
\]
To justify the middle limit explicitly, write the prelimit pairing by
integration by parts:
\[
 \int \chi\, Z_{t_j}\cdot Df_j
 =-\int f_j Z_{t_j}\cdot D\chi
   +\int \chi f_j w_{t_j},
 \qquad f_j=T_M((w_{t_j}-\delta)_+).
\]
The first term passes by strong $L^1$ convergence of $f_j$ and weak-star
convergence of $Z_{t_j}$.  For the second, $0\le f_j\le M$, strong
$L^1$ convergence of $w_{t_j}$ and bounded convergence in measure of
$f_j$ give $f_jw_{t_j}\to f w_\infty$ in $L^1_{\rm loc}$.  Thus the
prelimit pairings converge to the Anzellotti pairing of the limit.  Hence
equality holds throughout, which is
\eqref{eq:signed-positive-calibration}.  Replacing
$(w_{t_j},Z_{t_j})$ by $(-w_{t_j},-Z_{t_j})$ proves
\eqref{eq:signed-negative-calibration}.

It remains to record the prelimit strictness used later.  Put
$Z^\varepsilon:=\varepsilon Z_\infty$, and let
$f_j=T_M((\varepsilon w_{t_j}-\delta)_+)$ and
$f=T_M((\varepsilon w_\infty-\delta)_+)$.  The pointwise calibration
defect gives, on every compact $U$,
\[
 0\le |Df_j|(U)-(Z_{t_j}^\varepsilon,Df_j)(U)
 \le C_Ut_j^{-2}.
\]
For $U$ with $|Df|(\partial U)=0$, approximate $\mathbf1_U$ from inside
and outside by smooth cutoffs in the preceding pairing identity.  The
pairing convergence already proved above and the exact limiting
calibration give
\[
 \limsup_j|Df_j|(U)\le (Z^\varepsilon,Df)(U)=|Df|(U),
\]
whereas lower semicontinuity gives the reverse inequality.  Hence the
total variations converge, and the strong $L^1(U)$ convergence yields
\eqref{eq:signed-prelimit-strict-bv}.  If $c_j\to0$, the same argument
applies to $T_M((\varepsilon w_{t_j}-c_j-\delta)_+)$: the gradients and
pointwise calibration defect are unchanged, while the truncations converge
in $L^1_{\rm loc}$ to the same $f$.  This proves the final assertion.
\end{proof}

\begin{proposition}\label{prop:two-ball-impossible}
For every unbounded sign, $m_\varepsilon=1$.  Consequently
\begin{equation}\label{eq:sharp-tail-constant}
 C_\varepsilon=\omega_nk^n.
\end{equation}
\end{proposition}
\begin{proof}
Assume $m_\varepsilon=2$ and fix a blow-down sequence.  First apply
Proposition~\ref{prop:frag-signed-compactness} to obtain the full-space signed limit
$w_\infty$, its field $Z_\infty$, and the calibrated level-set limits.
On a countable dense set of common good levels take the same subsequence so
that the two scaled centers converge.  By Proposition~\ref{prop:frag-uniform-clusters}, at level
$s$ the limiting centers are $\pm ke(s)/s$.  For nearby good levels,
nested matching gives, after the harmless interchange of the two labels,
\[
 \left|\frac{k}{r}e(r)-\frac{k}{s}e(s)\right|
 \le k\left(\frac1s-\frac1r\right).
\]
Squaring gives $e(r)\cdot e(s)\ge1$.  Finite good-level chains therefore
produce one fixed $e\in\mathbb S^{n-1}$ such that for almost every $s>0$
\[
 \{\varepsilon w_\infty>s\}
 =B_{k/s}(ke/s)\cup B_{k/s}(-ke/s).
\]
For $y\notin H=e^\perp$, one of the two tangent balls contains $y$ for
all sufficiently small positive levels.  Thus
$\varepsilon w_\infty(y)>0$ for a.e.\ $y\notin H$; any opposite-sign
part is confined to the null hyperplane $H$.  Layer cake therefore
reconstructs the whole signed limit a.e.:
\[
 v(y):=\varepsilon w_\infty(y)
 =\frac{2k|e\cdot y|}{|y|^2}.
\]
Finite-value calibration \eqref{eq:signed-positive-calibration} forces, off
$H=e^\perp$, the signed limiting field $Z:=\varepsilon Z_\infty$ to satisfy
\[
 Z(y)=\operatorname{sgn}(e\cdot y)e-
 \frac{2|e\cdot y|}{|y|^2}y.
\]
On the two sides of $H$ the traces are $e$ and $-e$.  Hence
\[
 \operatorname{div}Z=-v+2\mathcal H^{n-1}\lfloor H.
\]
There is no atom at the origin since $Z$ is bounded and
$\mathcal H^{n-1}(\partial B_r)=O(r^{n-1})$.  This contradicts the exact full-space blow-down equation
$\operatorname{div}(\varepsilon Z_\infty)=-\varepsilon w_\infty=-v$ supplied by
Proposition~\ref{prop:frag-signed-compactness}.  Thus $m_\varepsilon=1$ and
\eqref{eq:sharp-tail-constant} follows from quantization.
\end{proof}

\begin{corollary}\label{cor:one-ball-signed-compactness}
In the setting of Proposition~\ref{prop:frag-signed-compactness}, every unbounded sign
has exactly one limiting ball.  Thus, for almost every $s>0$,
\begin{equation}\label{eq:one-ball-limit-family}
 \{\varepsilon w_\infty>s\}=B_{k/s}(c_\varepsilon(s)),
 \qquad |c_\varepsilon(s)|\le k/s,
\end{equation}
and the bounded truncations are calibrated by the corresponding signed
field.  If a sign is bounded near the puncture, its limiting positive
family is empty.
\end{corollary}
\begin{proof}
Combine Propositions~\ref{prop:frag-signed-compactness} and~\ref{prop:two-ball-impossible}.
\end{proof}

\subsection{Moving balls and horizontal translation charge}
\begin{lemma}
\label{lem:bridge-moving-ball-area}
Let \(I\Subset(0,\infty)\), let \(v\) be a measurable function, finite
almost everywhere, whose clipped representatives
\[
 v_{a,b}:=\max\{a,\min\{v,b\}\}
 \quad\text{belong to }BV_{\rm loc}(\mathbb R^n)
 \qquad(0<a<b<\infty),
\]
and
suppose that for almost every \(s\in I\)
\[
\{v>s\}=B(c(s),k/s)
\]
modulo null sets.  Suppose also that \(Z\in L^\infty_{\rm loc}\),
\(|Z|\leq1\), $\operatorname{div}Z\in L^1_{\rm loc}$, and that the
Anzellotti calibration
\[
(Z,Dv)=|Dv|\qquad\text{on }\{v\in I\}
\]
holds in the sense of bounded truncations supported in the value band.
Then \(c\) has a locally Lipschitz representative on \(I\).  The map
\[
\mathfrak S(s,\theta)=c(s)+\frac{k}{s}\theta
\qquad ((s,\theta)\in I\times S^{n-1})
\]
has multiplicity one almost everywhere on the swept region and
\begin{equation}\label{eq:bridge-moving-ball-jacobian}
J_\mathfrak S(s,\theta)=
\left(\frac{k}{s}\right)^{n-1}
\left(\frac{k}{s^2}-c'(s)\cdot\theta\right)
\quad\text{for a.e.\ }(s,\theta).
\end{equation}
The last factor is nonnegative.  Consequently, for every
\(\eta\in C_c(I)\),
\begin{equation}\label{eq:bridge-moving-ball-vector-formula}
\int_{\{v\in I\}}\eta(v)Z\,dy
=\frac{|S^{n-1}|}{n}
 \int_I\eta(s)
 \left(\frac{k}{s}\right)^{n-1}c'(s)\,ds.
\end{equation}
\end{lemma}

\begin{proof}
Choose \(0<a_0<\inf I\le\sup I<b_0<\infty\).
Replacing \(v\) by \(v_{a_0,b_0}\) preserves every superlevel in \(I\),
the swept region, and the stated calibration there.  We may therefore
work with this locally \(BV\) representative throughout the proof.
For $r<s$ in the full-measure set of good levels on which the ball
representation holds, nesting gives
\[
 B(c(s),k/s)\subset B(c(r),k/r),
 \qquad
 |c(r)-c(s)|\le k(r^{-1}-s^{-1}).
\]
The a.e.\ inclusion of these open balls is automatically pointwise, since a
nonempty open difference would have positive measure.  The good set is
dense in $I$, so the estimate extends $c$ uniquely to a locally Lipschitz
representative on all of $I$, with $|c'(s)|\le k/s^2$ a.e.\  For
$\mathfrak S(s,\theta)=c(s)+(k/s)\theta$,
\[
 D_\theta\mathfrak S[\tau]=\frac{k}{s}\tau,
 \qquad
 \partial_s\mathfrak S=c'(s)-\frac{k}{s^2}\theta,
\]
so
\[
 J_\mathfrak S(s,\theta)=
 \left(\frac{k}{s}\right)^{n-1}
 \left(\frac{k}{s^2}-c'(s)\cdot\theta\right),
\]
which proves \eqref{eq:bridge-moving-ball-jacobian}.

Fix good levels $a<b$ with $[a,b]\Subset I$ and let
$S_{a,b}=B(c(a),k/a)\setminus B(c(b),k/b)$.  The map $\mathfrak S$ covers this
shell.  By the area formula,
\[
 \int_a^b\int_{S^{n-1}}J_\mathfrak S\,d\theta\,ds
 =\omega_nk^n(a^{-n}-b^{-n})=|S_{a,b}|.
\]
Since the area-formula multiplicity is at least one on the swept shell, it
must equal one a.e.

The level-set representation also identifies $v$ with the arrival time of
the nested balls: by layer cake, for a.e.\ $y\in S_{a,b}$,
\[
 v(y)-a=\int_a^b\mathbf1_{B(c(s),k/s)}(y)\,ds.
\]
We differentiate this identity directly in the sense of distributions,
which avoids any inverse-map regularity issue.  Since
$D\mathbf1_{B(c(s),k/s)}=-\theta\,\mathcal H^{n-1}
\lfloor\partial B(c(s),k/s)$, Fubini gives, for every
$\xi\in C_c^1(S_{a,b};\mathbb R^n)$,
\[
 \int_{S_{a,b}}\xi\cdot dDv
 =-\int_a^b\int_{S^{n-1}}
   \xi(\mathfrak S(s,\theta))\cdot\theta
   \left(\frac{k}{s}\right)^{n-1}d\theta\,ds.
\]
The area-formula multiplicity is one a.e.\ and
\[
 J_\mathfrak S(s,\theta)=
 \left(\frac{k}{s}\right)^{n-1}
 \left(\frac{k}{s^2}-c'(s)\cdot\theta\right).
\]
For a.e.\ fixed $s$, the zero set of the last factor is empty unless
$|c'(s)|=k/s^2$, in which case it consists of a single point of
$S^{n-1}$.  Since $n\ge2$, the tangency set has zero $n$-dimensional
parameter measure.  On every compact substrip of $I\times S^{n-1}$ the
map $\mathfrak S$ is Lipschitz, so its image is Lebesgue-null.  Thus $Dv$ is
absolutely continuous on
the swept shell and, away from the tangency set,
\begin{equation}\label{eq:bridge-arrival-gradient}
 \nabla v(\mathfrak S(s,\theta))
 =-\frac{\theta}{k/s^2-c'(s)\cdot\theta}.
\end{equation}
In particular,
\[
 \int_{S_{a,b}}|\nabla v|\,dy
 =\int_a^b\int_{S^{n-1}}
   \left(\frac{k}{s}\right)^{n-1}d\theta\,ds<\infty,
\]
so $v\in W^{1,1}_{\rm loc}$ on the swept region.  The calibration now
reduces to the pointwise identity $Z\cdot\nabla v=|\nabla v|$; since
$|Z|\le1$, \eqref{eq:bridge-arrival-gradient} implies $Z=-\theta$ a.e.

A final application of the area formula yields
\[
\begin{aligned}
 \int_{S_{a,b}}\eta(v)Z\,dy
 &=-\int_a^b\int_{S^{n-1}}\eta(s)\theta
   \left(\frac{k}{s}\right)^{n-1}
   \left(\frac{k}{s^2}-c'(s)\cdot\theta\right)d\theta\,ds\\
 &=\frac{|S^{n-1}|}{n}\int_a^b\eta(s)
   \left(\frac{k}{s}\right)^{n-1}c'(s)\,ds,
\end{aligned}
\]
using $\int_{S^{n-1}}\theta=0$ and
$\int_{S^{n-1}}\theta\otimes\theta=(|S^{n-1}|/n)\Id$.  Exhausting $I$
proves \eqref{eq:bridge-moving-ball-vector-formula}.
\end{proof}

For a regular physical positive level define
\begin{equation}\label{eq:physical-translation-flux}
 \mathcal F(t)=\int_{\partial^*E_t}\frac1{\sqrt{1+|Du|^2}}\,\nu_t\,d\mathcal H^{n-1}.
\end{equation}
The same definition, after replacing $u$ by $-u$, is used for the negative
sign.

\begin{proposition}\label{prop:zero-charge-centers}
If $\mathcal F(t)=0$ for almost every sufficiently high regular level of
an unbounded sign, then every one-ball blow-down family is centered:
$c_\varepsilon(s)=0$ for almost every $s>0$.  Consequently the whole
natural rescaling converges without extraction,
\begin{equation}\label{eq:wholeblowdown-pressure}
 \varepsilon w_t(y)\longrightarrow\frac{k}{|y|}
 \quad\text{in }L^1_{\rm loc}(\mathbb R^n).
\end{equation}
\end{proposition}
\begin{proof}
Fix a blow-down sequence $t_j\to\infty$ and let
$v=\varepsilon w_\infty$ be its one-ball limit.  For
$\eta\in C_c^1((0,\infty))$, coarea on the natural scale gives the exact
identity
\begin{equation}\label{eq:translation-flux-natural-scaling}
 \int_{\mathbb R^n}\eta(\varepsilon w_t)Z_t^\varepsilon\,dy
 =-t^{n+1}\int_0^\infty\eta(s)\mathcal F_\varepsilon(ts)\,ds.
\end{equation}
Indeed, on the rescaled level $\varepsilon w_t=s$ one has
$Z_t^\varepsilon=-\sigma\nu$ and
$|D(\varepsilon w_t)|=t^{-2}|Du|$, while
$d\mathcal H_y^{n-1}=t^{n-1}d\mathcal H_x^{n-1}$.
Thus the assumed physical zero flux makes the left-hand side of
\eqref{eq:translation-flux-natural-scaling} vanish for all large $t$.

If $\operatorname{spt}\eta\subset[a,b]\Subset(0,\infty)$, the uniform
one-ball theorem makes the set $\{\varepsilon w_{t_j}>a\}$ globally tight
in the rescaled variables.  Strong local $L^1$ convergence therefore
upgrades to
\[
 \eta(\varepsilon w_{t_j})\to\eta(v)
 \qquad\hbox{strongly in }L^1(\mathbb R^n).
\]
By Proposition~\ref{prop:frag-signed-compactness}, after the same extraction
$Z_{t_j}\weakstar Z_\infty$; set $Z:=\varepsilon Z_\infty$.  Then
$Z_{t_j}^\varepsilon\weakstar Z$.  Together with the preceding global
$L^1$ convergence this yields
\[
 \int\eta(v)Z\,dy=0.
\]
Applying Lemma~\ref{lem:bridge-moving-ball-area} on any compact value band
$I$ gives
\[
 0=\frac{|S^{n-1}|}{n}
 \int_I\eta(s)\left(\frac{k}{s}\right)^{n-1}c_\varepsilon'(s)\,ds.
\]
The weight is strictly positive; hence $c_\varepsilon'=0$ a.e.\ and the
locally Lipschitz center path is constant.  The universal center bound
$|c_\varepsilon(s)|\le k/s$ then forces this constant to be zero by
letting $s\to\infty$ through good levels.

Thus every subsequential natural limit is $k/|y|$.  If the whole family
failed to converge locally in $L^1$, a bad sequence would contain a
subsequence with a different limit, contradicting the preceding
uniqueness.  This proves \eqref{eq:wholeblowdown-pressure}.
\end{proof}

\begin{corollary}\label{cor:highdim-centering}
If $n\ge3$, every unbounded sign is centered.  Moreover the positive and
negative tails cannot both be unbounded.
\end{corollary}
\begin{proof}
The translation-stress cutoff in Lemma~\ref{lem:tail-translation} gives
$\mathcal F=0$ for each sign separately.  Apply
Proposition~\ref{prop:zero-charge-centers}.  If both signs were unbounded, for a common
blow-down and almost every $s>0$ both
$\{w_\infty>s\}$ and $\{w_\infty<-s\}$ would equal the same ball
$B_{k/s}$.  The prelimit level sets are disjoint and converge globally in
symmetric difference, a contradiction.
\end{proof}

For the positive one-tail branch there exist $M<\infty$ and $r_1>0$
such that
\begin{equation}\label{eq:lower-bound-after-twotail}
 u\ge -M\qquad\text{on }B_{r_1}\setminus\{0\}.
\end{equation}
For $n\ge3$ this follows from the preceding two-tail exclusion.  For
$n=2$ it will follow after the critical defect models are eliminated in
\cref{sec:planar-defects}.

\section{The planar critical translation defect}
\label{sec:planar-defects}
In this section $n=2$ and $k=1$.  The finite-height stress cutoff is only
$O(1)$, so horizontal translation charge need not vanish before the axis
geometry is understood.  We show that its only possible equality models are
$H_e$ and $D_e$, and that neither can arise from an actual punctured
capillary graph.

\subsection{Cluster-set topology}
\begin{lemma}\label{lem:cluster-interval}
For every continuous $u:B_R\setminus\{0\}\to\mathbb R$ in dimension two,
\[
 \mathcal C_0(u)=[\liminf_{x\to0}u(x),\limsup_{x\to0}u(x)]
\]
as an interval in the extended real line.  In particular
$\mathcal C_0(u)=\{+\infty\}$ if and only if $u(x)\to+\infty$ uniformly,
and similarly for $-\infty$.
\end{lemma}
\begin{proof}
For each $r$, the image of the connected punctured disk $B_r\setminus\{0\}$
is an interval, and its closure in the compact extended real line is a
compact interval.  These intervals are nested as $r\downarrow0$, and their
intersection is the cluster set.  The endpoint and uniform-divergence
statements follow immediately.
\end{proof}

\subsection{Finite half-line clusters and the half-dipole}
Assume for contradiction that
\begin{equation}\label{eq:halfline-hyp}
 \mathcal C_0(u)=[\ell,+\infty],\qquad \ell\in\mathbb R.
\end{equation}
In the half-line branch write
\[
 \mathcal L_\ell:=\{0\}\times[\ell,+\infty),\qquad p_*:=(0,\ell),
 \qquad B_r^*:=B_r\setminus\{0\}.
\]
Then $u$ is bounded below near the puncture.  Choose $R_1\in(0,R)$ and
$M<\infty$ such that
\begin{equation}\label{eq:halfline-lower-bound}
 u\ge -M\qquad\text{on }B_{R_1}\setminus\{0\}.
\end{equation}
The following two planar lemmas are the only topological input in the
half-line exclusion.
\begin{lemma}
\label{lem:no-trapped-cold-component}
There is \(c_M>0\) with the following property.  Let \(t\ge1\) be a
regular value, and let \(D\) be a connected component of
\(\{u<t\}\cap B_{R_1}^*\) whose closure does not meet
\(\partial B_{R_1}\).  Then

\begin{equation*}
 |D|\ge c_M.
\end{equation*}

The conclusion also allows \(0\in\overline D\).
\end{lemma}

\begin{proof}
Set $v=(t-u)\mathbf1_D$.  For a.e.\ sufficiently small $\varepsilon>0$,
put \(D_\varepsilon=D\setminus\overline B_\varepsilon\); these radii
are admissible for the finite-perimeter Gauss--Green formulas below.  The
outer normal of \(D_\varepsilon\) on its inner circular boundary is
\(-e_r\), while \(v\) has zero trace on the ordinary level boundary of
\(D\).  Testing
\(\operatorname{div}\mathcal A(\nabla u)=-u\) by \(t-u\) on
\(D_\varepsilon\) therefore gives
\[
 \begin{split}
 \int_{D_\varepsilon}
   \frac{|\nabla u|^2}{\sqrt{1+|\nabla u|^2}}\,dx
 &=\int_{D_\varepsilon}(u^2-tu)\,dx\\
 &\quad+
 \int_{D\cap\partial B_\varepsilon}
   (t-u)\mathcal A(\nabla u)\cdot e_r\,d\mathcal H^1.
 \end{split}
\]
The last term has absolute value at most
\(2\pi\varepsilon(t+M)\).  Letting \(\varepsilon\downarrow0\) gives

\begin{equation*}
 \int_D\frac{|\nabla u|^2}{\sqrt{1+|\nabla u|^2}}\,dx
 =\int_D(u^2-tu)\,dx.
\end{equation*}

For Gauss--Green, the outer normal of \(D\) on the ordinary level
boundary is \(\nabla u/|\nabla u|\).  Hence
\[
 \begin{split}
 -\int_{D_\varepsilon}u\,dx
 &=\int_{\partial^*D\setminus B_\varepsilon}
   \frac{|\nabla u|}{\sqrt{1+|\nabla u|^2}}\,d\mathcal H^1\\
 &\quad-
 \int_{D\cap\partial B_\varepsilon}
   \mathcal A(\nabla u)\cdot e_r\,d\mathcal H^1.
 \end{split}
\]
The last integral is bounded by \(2\pi\varepsilon\).  Letting
\(\varepsilon\downarrow0\) gives

\begin{equation}
 -\int_Du\,dx
 =\int_{\partial^*D}
   \frac{|\nabla u|}{\sqrt{1+|\nabla u|^2}}\,d\mathcal H^1
 \ge0.
 \label{eq:cold-gauss-green}
\end{equation}

On \(0\le u<t\), \(u^2-tu\le0\); on \(-M\le u<0\), it is at most
\(M^2+tM\).  Thus

\begin{equation*}
 0\le
 \int_D\frac{|\nabla u|^2}{\sqrt{1+|\nabla u|^2}}\,dx
 \le(M^2+tM)|D|.
\end{equation*}

Since

\[
 |\nabla u|\le1+
 \frac{|\nabla u|^2}{\sqrt{1+|\nabla u|^2}},
\]

the zero extension of \(v\) belongs to \(BV(\mathbb R^2)\): the trace
of $t-u$ vanishes on the ordinary level boundary, so no jump term is
created there, and the puncture is $\mathcal H^1$-null.  Moreover it
satisfies

\begin{equation*}
 |Dv|(\mathbb R^2)\le(1+M^2+tM)|D|.
\end{equation*}

The planar \(BV\)-Sobolev inequality and
\eqref{eq:cold-gauss-green} now yield

\[
 t|D|
 \le\int_D(t-u)
 \le |D|^{1/2}\|v\|_{L^2}
 \le C_{BV}(1+M^2+tM)|D|^{3/2}.
\]

For \(t\ge1\),

\begin{equation*}
 |D|\ge
 \frac{1}{C_{BV}^2(1+M+M^2)^2}
 =:c_M>0.
\end{equation*}
\end{proof}

Choose \(0<R_0<R_1\) so small that \(\pi R_0^2<c_M\), and choose a
regular value

\begin{equation*}
 t_0>\max\left\{1,\ell,\max_{\partial B_{R_0}}u\right\}.
\end{equation*}

For every regular \(t\ge t_0\), the set
\(\{u<t\}\cap B_{R_0}^*\) is connected.  Indeed, a collar of
\(\partial B_{R_0}\) lies in \(\{u<t\}\).  If \(C\) were another
component, let \(D\) be the component of
\(\{u<t\}\cap B_{R_1}^*\) containing \(C\).  If \(D\) left \(B_{R_0}\),
a path in the connected open set \(D\) would first meet
\(\partial B_{R_0}\) in the collar and would connect \(C\) to the collar
component, a contradiction.  Thus \(D\subset B_{R_0}\), contrary to
Lemma~\ref{lem:no-trapped-cold-component}, because
\[
 |D|\le |B_{R_0}|=\pi R_0^2<c_M.
\]

For arbitrary \(t>t_0\), choose regular values \(t_k\uparrow t\).  Given
\(x,y\in\{u<t\}\cap B_{R_0}^*\), for all sufficiently large \(k\),
\[
 \max\{u(x),u(y)\}<t_k<t.
\]
Thus \(x\) and \(y\) belong to the same connected set
\(\{u<t_k\}\cap B_{R_0}^*\), which is contained in
\(\{u<t\}\cap B_{R_0}^*\).  This proves connectedness for every
\(t\ge t_0\).  Since \(\ell<t\) is a cluster value, the connected open set
accumulates at \(0\); joining such points to the outer collar gives paths
that cross every intervening annulus.

\begin{lemma}
\label{lem:zero-phase-every-annulus}
Let $r_j\downarrow0$ be a sequence for which the shifted natural rescalings
\begin{equation}\label{eq:halfline-shifted-natural-scaling}
 U_j(y):=r_j\bigl(u(r_jy)-\ell\bigr)
\end{equation}
converge strongly in $L^1_{\mathrm{loc}}(\mathbb R^2)$ to a limit $U$,
with the bounded positive truncations converging locally strictly in $BV$.
Then, for every $0<a<b<\infty$,
\[
 \operatorname*{ess\,inf}_{\{a<|y|<b\}}U=0.
\]
\end{lemma}

\begin{proof}
Since $U_j=w_{t_j}-\ell/t_j$, the constant shift leaves the gradient and
natural flux unchanged; the shifted strict-$BV$ conclusion in
\eqref{eq:signed-prelimit-strict-bv} supplies the stated subsequence and
truncation convergence.  The lower bound \eqref{eq:halfline-lower-bound} gives
$U_j\ge-r_j(M+|\ell|)$ on each fixed compact set for all large $j$; hence
$U\ge0$.  We first record the level-set consequence of the assumed
strict convergence of bounded positive truncations.  Fix a bounded positive
value interval $[\alpha,\beta]$ and a bounded open
$K\Subset\mathbb R^2$ such that
$|D T_{\beta-\alpha}((U-\alpha)_+)|(\partial K)=0$; such sets may be
chosen to form an exhaustion.  Coarea and perimeter lower semicontinuity give
\[
 P(\{U>s\};K)\le\liminf_jP(\{U_j>s\};K)
 \quad\text{for a.e.\ }s,
\]
while strict $BV$ convergence of
$T_{\beta-\alpha}((U_j-\alpha)_+)$ gives convergence of the integrals of
these perimeters over $(\alpha,\beta)$.  On such a $K$, write $P_j(s)=P(\{U_j>s\};K)$ and $P(s)=P(\{U>s\};K)$.
Lower semicontinuity gives $(P-P_j)_+\to0$ a.e., while
$0\le(P-P_j)_+\le P\in L^1(ds)$.  Dominated convergence and the
convergence of the coarea integrals yield
\[
 \int(P-P_j)_+\,ds\longrightarrow0,\qquad
 \int(P_j-P)_+\,ds
 =\int(P_j-P)\,ds+\int(P-P_j)_+\,ds\longrightarrow0.
\]
Thus $\int|P_j-P|\,ds\to0$.  After one scalar diagonal extraction
over such an open exhaustion and bounded positive value intervals, the
level indicators converge in $L^1_{\rm loc}$ and their perimeters
converge for a full-measure set of $s>0$; hence
\[
 \mathbf1_{\{U_j>s\}}\longrightarrow\mathbf1_{\{U>s\}}
 \quad\text{locally strictly in }BV.
\]
Suppose now that $U\ge m>0$ almost everywhere on $A_{a,b}$, and choose
\(s\in(0,m)\) from this full-measure set.  Next choose \(a<a_1<b_1<b\) so that
\(|D\mathbf1_{\{U>s\}}|\) does not charge either boundary circle.  With
\(E_j=\{U_j>s\}\), strict convergence gives

\begin{equation}
 |A_{a,b}\setminus E_j|\longrightarrow0,
 \qquad
 \Per(E_j;A_{a_1,b_1})\longrightarrow0.
 \label{eq:cold-strict-bv}
\end{equation}

For large \(j\), the physical threshold
\(\ell+s/r_j\) exceeds \(t_0\), and

\[
 \{U_j<s\}=r_j^{-1}\{u<\ell+s/r_j\}
\]

contains a path crossing \(A_{a_1,b_1}\).  Let \(B_j\) be the set of
radii \(\rho\in(a_1,b_1)\) for which
\(E_j\cap\partial B_\rho=\varnothing\).  Polar Fubini and the first
limit in \eqref{eq:cold-strict-bv} give
\[
 2\pi a_1|B_j|
 \le \int_{B_j}\mathcal H^1(\partial B_\rho)\,d\rho
 \le |A_{a,b}\setminus E_j|\longrightarrow0,
\]
so \(|B_j|\to0\).  For almost every
remaining radius the circle contains a nonempty open arc of each phase,
so its one-dimensional perimeter is at least \(2\).  The \(BV\) slicing
inequality gives

\[
 \Per(E_j;A_{a_1,b_1})
 \ge2(b_1-a_1-|B_j|),
\]

contradicting \eqref{eq:cold-strict-bv}.
\end{proof}

\begin{proposition}\label{prop:halfline-half-dipole-blowdown}
Every sequence $r_j\downarrow0$ has a subsequence and $e\in\mathbb S^1$
such that
\[
 r_j\bigl(u(r_j\,\cdot)-\ell\bigr)\to H_e
 \quad\text{strongly in }L^1_{\rm loc}(\mathbb R^2).
\]
The bounded positive truncations converge locally strictly in $BV$.
\end{proposition}
\begin{proof}
Fix $r_j\downarrow0$, put $t_j=r_j^{-1}$, and write
\[
 U_j=w_{t_j}-\ell/t_j.
\]
By Proposition~\ref{prop:frag-signed-compactness} and
\eqref{eq:signed-prelimit-strict-bv}, after a subsequence
$U_j\to U$ strongly in $L^1_{\rm loc}$ and every bounded positive
truncation converges locally strictly in $BV$.  The lower bound
\eqref{eq:halfline-lower-bound} gives $U\ge0$.

Fix a positive continuity level $s$ for the limiting family and set
\[
 q_j(s):=\ell+s t_j.
\]
Then exactly
\[
 \{U_j>s\}=t_jE_{q_j(s)}^+.
\]
Apply the uniform one-ball theorem at the physical height $q_j(s)$.  Since
$q_j(s)/t_j=s+\ell/t_j\to s$, after multiplication by the spatial factor
$t_j$ the approximating radius is
\[
 t_j/q_j(s)=\frac1{s+\ell/t_j}\longrightarrow\frac1s,
\]
and the sharp center bound gives
\[
 t_j|a^+(q_j(s))|
 \le \frac{1+o(1)}{s+\ell/t_j}.
\]
Diagonalizing over a countable dense family of such levels and using
nesting therefore gives, for almost every $s>0$,
\[
 \{U>s\}=B_{1/s}(c(s)),\qquad |c(s)|\le1/s.
\]
By the annular essential-infimum lemma above, strict inequality
$|c(s)|<1/s$ is impossible: such a disk contains a nonempty annulus
centered at the origin.  Hence $c(s)=e(s)/s$.  Nested matching at nearby
good levels gives
\[
 |e(r)/r-e(s)/s|\le1/s-1/r,
\]
and squaring forces $e(r)=e(s)$.  Thus $e$ is independent of level, and
layer cake gives $U=H_e$.
\end{proof}

\subsection{Current completion and the axis-supported first variation}
\begin{lemma}
For every \(z_0\in\mathbb R\), there are \(C_{z_0},\rho_0>0\) such that

\begin{equation}
 \mathcal H^2\bigl(\Sigma\cap B_\rho^3((0,z_0))\bigr)
 \le C_{z_0}\rho^2
 \qquad(0<\rho<\rho_0).
 \label{eq:halfline-quadratic-area}
\end{equation}

Moreover \(E\) has locally finite perimeter in
\(\mathcal O:=B_R\times\mathbb R\), and the relative current

\begin{equation*}
 T:=\partial[\![E]\!]
\end{equation*}

is locally integral and satisfies \(\partial T=0\) in every ambient ball
compactly contained in \(\mathcal O\).
\end{lemma}

\begin{proof}
The ambient ball in \eqref{eq:halfline-quadratic-area} projects into
\(B_\rho\) and restricts graph height to
\([z_0-\rho,z_0+\rho]\).  Thus
\eqref{common:shrinking-band} gives
\eqref{eq:halfline-quadratic-area}.

Fix a compact cylinder $K=B_\rho\times(a,b)\Subset\mathcal O$, and put
$E_\varepsilon=E\cap\{|x|>\varepsilon\}$.  Its artificial boundary is
contained in
$S_\varepsilon=\{|x|=\varepsilon\}\times(a,b)$, so
\begin{equation}
 \mathcal H^2(S_\varepsilon)=2\pi\varepsilon(b-a),
 \qquad
 |(E_\varepsilon\mathbin\triangle E)\cap K|
 \le\pi\varepsilon^2(b-a).
 \label{eq:halfline-puncture-cylinder-cost}
\end{equation}
The part of $\partial E_\varepsilon\cap K$ lying on the graph is bounded
uniformly by a finite cover of $K$ and
\eqref{eq:halfline-quadratic-area}.  Hence
\[
 \sup_{0<\varepsilon<\rho/2}\Per(E_\varepsilon;K)<\infty,
 \qquad
 \mathbf1_{E_\varepsilon}\to\mathbf1_E
 \quad\hbox{in }L^1(K).
\]
Lower semicontinuity therefore gives $E\in BV_{\mathrm{loc}}(\mathcal O)$.
Away from the axis, $\partial^*E$ is the original smooth graph with its
subgraph normal; the axis is $\mathcal H^2$-null and contributes no
perimeter by \eqref{eq:halfline-puncture-cylinder-cost}.  Consequently
$T=\partial[\![E]\!]$ is precisely integration over that graph with the
subgraph orientation, and
\[
 \partial T=\partial^2[\![E]\!]=0
 \quad\hbox{in every }U\Subset\mathcal O.
\]
\end{proof}

Let \(V\) denote the integral varifold associated with the graph.  Off the
axis, let \(\mathbf H_\Sigma\) be its classical mean-curvature vector; in a
fixed finite-height cylinder, \(|\mathbf H_\Sigma|=|z|\).

\begin{proposition}
\label{prop:halfline-axis-first-variation}
The first variation of \(V\) is the sum of its classical
prescribed-mean-curvature part and a vector Radon measure supported on
\(\mathcal L_\ell\).  More precisely,

\begin{equation}
 \delta V(X)
 =-\int \mathbf H_\Sigma\cdot X\,d\|V\|
   +\mu_{\mathrm{ax}}(X),
 \qquad
 \mu_{\mathrm{ax}}
 =\mathbf g_{\rm ax}(z)\mathcal H^1\llcorner\mathcal L_\ell,
 \label{eq:halfline-axis-density}
\end{equation}

for some \(\mathbf g_{\rm ax}\in L^\infty_{\mathrm{loc}}(\mathcal L_\ell;\mathbb R^3)\).  In particular,
\(\mu_{\mathrm{ax}}\) has no atom at \(p_*=(0,\ell)\).
\end{proposition}

\begin{proof}
Let \(\chi_\eta(x)\) vanish on \(B_\eta\), equal \(1\) outside
\(B_{2\eta}\), and satisfy \(|D\chi_\eta|\le C/\eta\).  The smooth
first-variation identity for \(\chi_\eta X\) is
\[
 \int_\Sigma\chi_\eta\operatorname{div}_\Sigma X\,d\mathcal H^2
 +\int_\Sigma\nabla_\Sigma\chi_\eta\cdot X\,d\mathcal H^2
 =-\int_\Sigma\chi_\eta\mathbf H_\Sigma\cdot X\,d\mathcal H^2.
\]
Moreover,
\[
 \begin{aligned}
 &\left|\int_\Sigma(1-\chi_\eta)
       \operatorname{div}_\Sigma X\,d\mathcal H^2\right|
 \le2\|DX\|_\infty\mathcal H^2(\Sigma\cap\{|x|<2\eta\}),\\
 &\left|\int_\Sigma(1-\chi_\eta)
       \mathbf H_\Sigma\cdot X\,d\mathcal H^2\right|
 \le\|\mathbf H_\Sigma\|_\infty\|X\|_\infty
       \mathcal H^2(\Sigma\cap\{|x|<2\eta\}).
 \end{aligned}
\]
Both right-hand sides tend to zero by
\eqref{eq:halfline-quadratic-area}.  Letting \(\eta\downarrow0\) therefore
gives

\begin{equation}
 \mu_{\mathrm{ax}}(X)
 :=\delta V(X)+\int\mathbf H_\Sigma\cdot X\,d\|V\|
 =-\lim_{\eta\downarrow0}
   \int_\Sigma\nabla_\Sigma\chi_\eta\cdot X\,d\mathcal H^2.
 \label{eq:halfline-axis-cutoff}
\end{equation}

Fix a compact height interval
\(I_0\Subset\mathbb R\).  If \(X\) is supported in a sufficiently small
base cylinder over a subinterval \(J\subset I_0\), then

\begin{align*}
 |\mu_{\mathrm{ax}}(X)|
 &\le \|X\|_\infty
 \liminf_{\eta\downarrow0}\frac{C}{\eta}
 \int_{B_{2\eta}\cap\{u\in J\}}
       \sqrt{1+|\nabla u|^2}\,dx \\
 &\le C_{I_0}|J|\,\|X\|_\infty.
\end{align*}

Here the direct application of
\eqref{common:shrinking-band} first gives
\(C_{I_0}|J|+O(\eta)\), and the last term vanishes as
\(\eta\downarrow0\).

Taking \(J=I_0\) shows that the distribution in
\eqref{eq:halfline-axis-cutoff} is of order zero on every compact
cylinder.  The Riesz representation theorem therefore makes it a vector
Radon measure.  Applying the dual characterization of total variation to
fields supported over \(J\) gives

\[
 |\mu_{\mathrm{ax}}|(\{0\}\times J)\le C_{I_0}|J|.
\]

The distribution is supported on the finite-height axis cluster set,
which is exactly \(\mathcal L_\ell\) under
\eqref{eq:halfline-hyp}.  Radon--Nikodym now gives
\eqref{eq:halfline-axis-density}.  A singleton has zero
\(\mathcal H^1\)-measure, so no endpoint atom is possible.
\end{proof}

\subsection{The height-localized Bregman identity}

Fix \(z_0\ge\ell\) and \(\rho_j\downarrow0\).  Define

\begin{equation*}
 U_j(X):=\rho_j\bigl(u(\rho_jX)-\ell\bigr),
 \qquad
 \zeta_j(X):=\frac{u(\rho_jX)-z_0}{\rho_j}.
\end{equation*}

Then

\begin{equation}
 \rho_j^2\zeta_j=U_j-\rho_j(z_0-\ell).
 \label{eq:halfline-scale-relation}
\end{equation}

Put

\begin{equation*}
 W_j:=\sqrt{1+|D\zeta_j|^2},
 \qquad F_j:=\frac{D\zeta_j}{W_j},
 \qquad \gamma_j^{\rm src}:=\rho_jz_0+\rho_j^2\zeta_j=\rho_ju(\rho_jX).
\end{equation*}

Since \(D\zeta_j(X)=Du(\rho_jX)\),

\begin{equation}
 \operatorname{div}F_j=-\gamma_j^{\rm src},
 \qquad
 DU_j=\rho_j^2Du(\rho_jX),
 \qquad
 \frac{DU_j}{\sqrt{\rho_j^4+|DU_j|^2}}=F_j.
 \label{eq:halfline-exact-environmental-flux}
\end{equation}

The shrinking-height estimate rescales to

\begin{equation}
 \int_{B_A\cap\{a<\zeta_j<b\}}W_j\,dX
 \le C\bigl(A|b-a|+A^2\bigr)
 \label{eq:halfline-environmental-area}
\end{equation}

for bounded \(A\), compact \([a,b]\), and all large \(j\).

After a subsequence,
Proposition~\ref{prop:halfline-half-dipole-blowdown} gives
\(U_j\to H_e\).  Rotate so that \(e=e_1\), and write
\(\Pi_-=\{X_1<0\}\), \(\Pi_+=\{X_1>0\}\).

\begin{lemma}
\label{lem:halfline-inactive-flux}
On every compact subset of \(\Pi_-\),

\begin{equation}
 F_j\longrightarrow e_1\quad\text{strongly in }L^2,
 \qquad
 W_j^{-1}\longrightarrow0\quad\text{strongly in }L^2.
 \label{eq:halfline-inactive-strong}
\end{equation}
\end{lemma}

\begin{proof}
Every weak-star limit is the calibrated field \(F\) from
Proposition~\ref{prop:frag-signed-compactness}.  On \(\Pi_+\), calibration of
the bounded positive truncations gives

\begin{equation*}
 F=\frac{DH_{e_1}}{|DH_{e_1}|}
  =e_1-H_{e_1}(X)X.
\end{equation*}

Its first component has weak normal trace $1$ on $\{X_1=0\}$.  Here
\[
 DM^\infty_{\mathrm{loc}}(\Omega)
 :=\{G\in L^\infty_{\mathrm{loc}}(\Omega;\mathbb R^2):
          \operatorname{div}G\in\mathcal M_{\mathrm{loc}}(\Omega)\}.
\]
Since $F\in DM^\infty_{\mathrm{loc}}$ and
$\operatorname{div}F=-H_{e_1}\in L^1_{\mathrm{loc}}$ has no line part,
the Gauss--Green formula for bounded divergence-measure fields
\cite[Theorem~\refnum{2.2}]{ChenFrid1999} applies on bounded rectangles in each
half-plane.  Apply it separately on the two sides and subtract the two
identities.  Since $\operatorname{div}F=-H_{e_1}\,dX$ has no
$\mathcal H^1$-component on the interface, the only possible interface
term is the jump of the normal trace; exhausting the rectangles to the
interface therefore gives
\begin{equation}
 \bigl(\operatorname{Tr}_{\Pi_+}F-
       \operatorname{Tr}_{\Pi_-}F\bigr)\cdot e_1=0
 \quad\mathcal H^1\text{-a.e.\ on }\{X_1=0\}.
 \label{eq:halfline-normal-trace-jump}
\end{equation}
Thus the trace from $\Pi_-$ is also $1$.  On $\Pi_-$,
$\operatorname{div}F=0$.  Hence there is a
stream function
\(\psi_{\rm str}\in W^{1,\infty}_{\mathrm{loc}}(\Pi_-)\) such that

\[
 D\psi_{\rm str}=(-F_2,F_1),\qquad |D\psi_{\rm str}|\le1.
\]
Indeed,
\[
 \operatorname{curl}(-F_2,F_1)
 =\partial_{X_1}F_1+\partial_{X_2}F_2
 =\operatorname{div}F=0,
\]
and the half-plane $\Pi_-$ is simply connected.

On the interface, the tangential derivative of the stream-function trace is
\[
 \partial_y\operatorname{Tr}\psi_{\rm str}
 =\operatorname{Tr}_{\Pi_-}F_1=1
 \quad\text{in distributions}.
\]
After an additive normalization, therefore,
\(\operatorname{Tr}\psi_{\rm str}(0,y)=y\).  Since
\(\|D\psi_{\rm str}\|_{L^\infty(\Pi_-)}\le1\), its global \(1\)-Lipschitz
representative satisfies, for \(x<0\) and every \(a\in\mathbb R\),

\[
 a-\sqrt{x^2+(y-a)^2}
 \le\psi_{\rm str}(x,y)\le
 a+\sqrt{x^2+(y-a)^2}.
\]

Letting \(a\to+\infty\) in the left inequality and
\(a\to-\infty\) in the right gives \(\psi_{\rm str}(x,y)=y\), hence \(F=e_1\) on
\(\Pi_-\).

Every subsequential weak-star limit is therefore the extreme point
\(e_1\).  Since

\[
 |F_j-e_1|^2\le2(1-e_1\cdot F_j),
 \qquad
 W_j^{-2}=1-|F_j|^2\le2(1-e_1\cdot F_j),
\]

weak-star convergence proves \eqref{eq:halfline-inactive-strong}.
\end{proof}

\begin{lemma}
Let \(0\le\varphi\in C_c^1(\mathbb R^2)\) and
\(0\le\chi\in C_c^1(\mathbb R)\).  Then

\begin{equation}
 \int\varphi\chi(\zeta_j)(W_j-e_1\cdot D\zeta_j)\,dX
 \longrightarrow0.
 \label{eq:halfline-bregman}
\end{equation}

Equivalently, with

\begin{equation*}
 N_j=(-F_j,W_j^{-1}),
 \qquad N_0=(-e_1,0),
 \qquad d\mu_j=W_j\,dX,
\end{equation*}

one has

\begin{equation}
 \frac12\int\varphi\chi(\zeta_j)|N_j-N_0|^2\,d\mu_j
 \longrightarrow0.
 \label{eq:halfline-normal-defect}
\end{equation}
\end{lemma}

\begin{proof}
Set

\[
 \Gamma(s):=-\int_s^\infty\chi(\tau)\,d\tau,
 \qquad \Gamma'=\chi.
\]

By \eqref{eq:halfline-environmental-area},
\(\varphi\chi(\zeta_j)D\zeta_j\) is integrable on its finite height window.
Thus \(\varphi\Gamma(\zeta_j)\) is an admissible bounded \(W^{1,1}\) test after
inserting a radial cutoff \(\xi_\eta\), which vanishes on \(B_\eta\), equals
one outside \(B_{2\eta}\), and satisfies
\(|D\xi_\eta|\le C/\eta\).  Since
\(\operatorname{div}F_j=-\gamma_j^{\rm src}\), the complete test identity is
\[
 \begin{aligned}
 &\int \xi_\eta\varphi\chi(\zeta_j)F_j\cdot D\zeta_j
 +\int \xi_\eta\Gamma(\zeta_j)F_j\cdot D\varphi\\
 &\qquad
 +\int \varphi\Gamma(\zeta_j)F_j\cdot D\xi_\eta
 =\int \gamma_j^{\rm src}\xi_\eta\varphi\Gamma(\zeta_j).
 \end{aligned}
\]
The cutoff term satisfies
\[
 \left|\int \varphi\Gamma(\zeta_j)F_j\cdot D\xi_\eta\right|
 \le \frac C\eta|B_{2\eta}\setminus B_\eta|=O(\eta),
\]
because \(|F_j|\le1\) and \(\Gamma\) is bounded.  Dominated convergence
handles the remaining cutoff factors.  Letting \(\eta\downarrow0\) gives

\begin{equation}
 \int\varphi\chi(\zeta_j)F_j\cdot D\zeta_j
 +\int\Gamma(\zeta_j)F_j\cdot D\varphi
 =\int \gamma_j^{\rm src}\varphi\Gamma(\zeta_j).
 \label{eq:halfline-bregman-test}
\end{equation}

Since \(e_1\) is constant,

\begin{equation}
 \int\varphi\chi(\zeta_j)e_1\cdot D\zeta_j
 +\int\Gamma(\zeta_j)e_1\cdot D\varphi=0.
 \label{eq:halfline-constant-field-test}
\end{equation}

Using \(F_j\cdot D\zeta_j=W_j-W_j^{-1}\) and subtracting
\eqref{eq:halfline-constant-field-test} from
\eqref{eq:halfline-bregman-test}, we obtain the exact identity

\begin{align}
 \int\varphi\chi(\zeta_j)(W_j-e_1\cdot D\zeta_j)
 &=\int\varphi\chi(\zeta_j)W_j^{-1}
   +\int \gamma_j^{\rm src}\varphi\Gamma(\zeta_j) \\
 &\quad-\int\Gamma(\zeta_j)(F_j-e_1)\cdot D\varphi.
 \label{eq:halfline-exact-bregman}
\end{align}

Fix \(\delta>0\) and split the support of \(\varphi\) into
\({X_1\le-\delta}\), \({|X_1|<\delta}\), and
\({X_1\ge\delta}\).  On the inactive part, the first and third terms on
the right tend to zero by Lemma~\ref{lem:halfline-inactive-flux}: if
$K=\operatorname{spt}\varphi$, then Cauchy--Schwarz gives
\begin{align*}
 \int_{K\cap\{X_1\le-\delta\}}
   |\varphi\chi(\zeta_j)|W_j^{-1}
 &\le C_K\|W_j^{-1}\|_{L^2(K\cap\{X_1\le-\delta\})}\to0,\\
 \left|\int_{K\cap\{X_1\le-\delta\}}
   \Gamma(\zeta_j)(F_j-e_1)\cdot D\varphi\right|
 &\le C_K\|F_j-e_1\|_{L^2(K\cap\{X_1\le-\delta\})}\to0.
\end{align*}
On $K\cap\{X_1\ge\delta\}$, the half-dipole $H_{e_1}$ has a positive
lower bound and \eqref{eq:halfline-scale-relation} gives
$\zeta_j\to+\infty$ in measure.  Since $W_j^{-1}\le1$, $|F_j-e_1|\le2$,
and $\chi,\Gamma$ are bounded with compact upper support, dominated
convergence makes the same two integrals tend to zero.  Finally,
\begin{align*}
 \int_{K\cap\{|X_1|<\delta\}}|\varphi\chi(\zeta_j)|W_j^{-1}
 &\le C_K\,|K\cap\{|X_1|<\delta\}|,\\
 \left|\int_{K\cap\{|X_1|<\delta\}}
 \Gamma(\zeta_j)(F_j-e_1)\cdot D\varphi\right|
 &\le C_K\,|K\cap\{|X_1|<\delta\}|,
\end{align*}
and both right-hand sides are $O_K(\delta)$.

If \(\Gamma(\zeta_j)\ne0\), then \(\zeta_j\le b\) for a fixed \(b\).  On a fixed
compact base set and for large \(j\),

\[
 \ell-1\le u(\rho_jX)\le z_0+\rho_jb,
\]

where the lower bound follows from
\eqref{eq:halfline-hyp}.  Thus
\(|\gamma_j^{\rm src}|=\rho_j|u(\rho_jX)|\le C\rho_j\) on the source support, and the
source term tends to zero.  Let first \(j\to\infty\) and then
\(\delta\downarrow0\) in \eqref{eq:halfline-exact-bregman}.  This proves
\eqref{eq:halfline-bregman}.  Finally,

\[
 W_j-e_1\cdot D\zeta_j
 =W_j(1-e_1\cdot F_j)
 =\frac12W_j|N_j-N_0|^2,
\]

which gives \eqref{eq:halfline-normal-defect}.
\end{proof}

\subsection{Environmental planes and vanishing of the axis-supported first variation}

\begin{lemma}
\label{lem:halfline-global-constant-polar}
Let \(G\subset\mathbb R^3\) be a set of locally finite perimeter which
is neither null nor conull.  If, for a constant unit vector \(N_0\) and
\(\sigma\in\{-1,1\}\),

\begin{equation}
 D\mathbf1_G=\sigma N_0|D\mathbf1_G|
 \qquad\text{in }\mathbb R^3,
 \label{eq:halfline-constant-polar}
\end{equation}

then \(G\) agrees almost everywhere with a half-space, and
\(\partial[\![G]\!]\) is its multiplicity-one boundary plane.
\end{lemma}

\begin{proof}
Every distributional derivative tangent to \(N_0^\perp\) vanishes.
After an orthogonal change of coordinates with \(s=X\cdot N_0\), Fubini
therefore gives
\[
 \mathbf1_G(X)=h(s)\quad\text{for a.e.\ }X,
 \qquad h\in BV_{\mathrm{loc}}(\mathbb R;\{0,1\}).
\]
Equation \eqref{eq:halfline-constant-polar} reduces to
\[
 Dh=\sigma|Dh|.
\]
Thus \(h\) is monotone.  Since \(G\) is nontrivial and a
\(\{0,1\}\)-valued monotone function has exactly one jump, there is
\(s_0\in\mathbb R\) such that, up to null sets,
\[
 h=\mathbf1_{(s_0,\infty)}
 \quad\text{or}\quad
 h=\mathbf1_{(-\infty,s_0)}.
\]
Hence \(G\) is a half-space, and the reduced boundary of a
finite-perimeter set has unit current multiplicity.
\end{proof}

\begin{proposition}
\label{prop:halfline-environmental-tangents}
Fix \(z_0>\ell\).  There are \(x_j\to0\) such that \(u(x_j)=z_0\).  Put
\(\rho_j=|x_j|\), and dilate the graph about \(p_0=(0,z_0)\) by
\(\rho_j^{-1}\).  Every convergent subsequence has a further subsequence
whose current and varifold limits are one and the same multiplicity-one
vertical plane.
\end{proposition}

\begin{proof}
Choose \(\delta>0\) with
\(z_0\pm\delta\in(\ell,+\infty)\).  Values below and above \(z_0\) occur
arbitrarily close to the puncture.  Connectedness of every punctured disk
and the intermediate value theorem give exact-level points
\(x_j\to0\) with \(u(x_j)=z_0\).  After a subsequence,
\(x_j/\rho_j\to\omega\in\mathbb S^1\).

Let

\[
 \eta_j(x,z)=\left(\frac{x}{\rho_j},\frac{z-z_0}{\rho_j}\right).
\]

Denote the dilated varifold and subgraph set by \(\mathcal V_j\) and \(\mathcal E_j\), and
write

\[
 \mathcal T_j=\partial[\![\mathcal E_j]\!]
\]

for the corresponding relative boundary current.  The quadratic area
estimate and Proposition~\ref{prop:halfline-axis-first-variation} give
uniform local mass and first-variation bounds, hence varifold, current,
and Caccioppoli compactness.  The rescaled graph contains
\((x_j/\rho_j,0)\to(\omega,0)\).  Choose \(r>0\) so that
\(B_{2r}((\omega,0))\) stays away from the axis.  The rescaled mean
curvature tends uniformly to zero there, and bounded-mean-curvature
monotonicity at the smooth density-one points
\((x_j/\rho_j,0)\) gives, for all large \(j\),
\[
 \|\mathcal V_j\|(B_r((\omega,0)))\ge c r^2
\]
with \(c>0\) independent of \(j\).  Thus every varifold limit is nonzero.

After another subsequence, the natural blow-down is \(H_e\) by
Proposition~\ref{prop:halfline-half-dipole-blowdown}.  Finite sums of
nonnegative product cutoffs \(\varphi(X)\chi(Z)\) approximate every
nonnegative compactly supported ambient cutoff.  Thus
\eqref{eq:halfline-normal-defect} and the local mass bounds give

\begin{equation}
 \int\Phi|N_j-N_0|^2\,d\|\mathcal V_j\|\longrightarrow0,
 \qquad N_0=(-e,0),
 \label{eq:halfline-ambient-normal-defect}
\end{equation}

for every \(0\le\Phi\in C_c(\mathbb R^3)\).

Pass to a subsequence such that

\[
 \mathbf1_{\mathcal E_j}\to\mathbf1_{\mathcal E_\infty}
 \quad\text{in }L^1_{\mathrm{loc}},
 \qquad
 \mathcal T_j\rightharpoonup \mathcal T_\infty=\partial[\![\mathcal E_\infty]\!].
\]

The Hodge-dual vector measure of \(\mathcal T_j\) is
\(N_j\|\mathcal V_j\|\).  Cauchy--Schwarz and
\eqref{eq:halfline-ambient-normal-defect} show that
\(N_j\|\mathcal V_j\|-N_0\|\mathcal V_j\|\to0\) locally in total variation, because for
every compact \(K\Subset\mathbb R^3\),
\[
 \int_K|N_j-N_0|\,d\|\mathcal V_j\|
 \le
 \left(\int_K|N_j-N_0|^2\,d\|\mathcal V_j\|\right)^{1/2}
 \|\mathcal V_j\|(K)^{1/2}\longrightarrow0.
\]
With the
convention \(D\mathbf1_E=-\nu_E|D\mathbf1_E|\), passage to the limit
gives

\begin{equation}
 -D\mathbf1_{\mathcal E_\infty}=N_0\|\mathcal V_\infty\|.
 \label{eq:halfline-vector-measure-identity}
\end{equation}

Taking total variations yields

\begin{equation}
 |D\mathbf1_{\mathcal E_\infty}|=\|\mathcal V_\infty\|.
 \label{eq:halfline-no-hidden-multiplicity}
\end{equation}

The limit is global on the expanding ambient domains and is nonzero.
Lemma~\ref{lem:halfline-global-constant-polar} therefore makes
\(\mathcal E_\infty\) a half-space.  Equation
\eqref{eq:halfline-no-hidden-multiplicity} shows that its current and
varifold boundaries are the same multiplicity-one plane.  Since \(N_0\)
is horizontal, that plane is vertical.
\end{proof}

\begin{proposition}
\label{prop:halfline-axis-force-vanishes}
The density \(\mathbf g_{\rm ax}\) in \eqref{eq:halfline-axis-density} vanishes almost
everywhere on \((\ell,+\infty)\).  Consequently,

\begin{equation}
 \mu_{\mathrm{ax}}=0
 \label{eq:halfline-no-axis-force}
\end{equation}
throughout \(\mathcal O=B_R\times\mathbb R\).
\end{proposition}

\begin{proof}
Let \(z_0>\ell\) be a Lebesgue point of \(\mathbf g_{\rm ax}\), and use the exact-level
scales in Proposition~\ref{prop:halfline-environmental-tangents}.  For a
compactly supported smooth ambient field \(Y\), scaling
\eqref{eq:halfline-axis-density} gives

\begin{align*}
 \delta \mathcal V_j(Y)
 &=-\rho_j\int
 \mathbf H_\Sigma((0,z_0)+\rho_jQ)\cdot Y(Q)\,d\|\mathcal V_j\|(Q) \\
 &\quad+\int \mathbf g_{\rm ax}(z_0+\rho_jZ)\cdot Y(0,Z)\,dZ.
\end{align*}

The first term tends to zero by the local mass bound.  The Lebesgue-point
property gives

\[
 \mathbf g_{\rm ax}(z_0+\rho_j\,\cdot)\longrightarrow \mathbf g_{\rm ax}(z_0)
 \quad\text{in }L^1_{\mathrm{loc}}(\mathbb R).
\]

On the other hand, after a subsequence \(\mathcal V_j\) converges to the
stationary multiplicity-one plane \(\mathcal V_\infty\) by
Proposition~\ref{prop:halfline-environmental-tangents}.  Since
\((Q,S)\mapsto\operatorname{div}_S Y(Q)\) is continuous and compactly
supported on the Grassmann bundle, varifold convergence gives
\[
 \delta\mathcal V_j(Y)
 =\int\operatorname{div}_S Y(Q)\,d\mathcal V_j(Q,S)
 \longrightarrow
 \int\operatorname{div}_S Y(Q)\,d\mathcal V_\infty(Q,S)=0.
\]
Hence

\[
 \int \mathbf g_{\rm ax}(z_0)\cdot Y(0,Z)\,dZ=0
\]

for every \(Y\), so \(\mathbf g_{\rm ax}(z_0)=0\).  Since almost every
\(z_0\in(\ell,+\infty)\) is a Lebesgue point, $\mathbf g_{\rm ax}=0$
\(\mathcal H^1\)-almost everywhere on that interval.  By
\eqref{eq:halfline-axis-density},
\[
 \mu_{\mathrm{ax}}
 =\mathbf g_{\rm ax}\,\mathcal H^1\llcorner
   \bigl(\{0\}\times(\ell,+\infty)\bigr)
 +\mu_{\mathrm{ax}}(\{p_*\})\,\delta_{p_*}.
\]
The preceding conclusion gives
$\mathbf g_{\rm ax}\,\mathcal H^1\llcorner(\{0\}\times(\ell,+\infty))=0$, while
Proposition~\ref{prop:halfline-axis-first-variation} gives
$\mu_{\mathrm{ax}}(\{p_*\})=0$.  The same proposition places the support
of $\mu_{\mathrm{ax}}$ in
$\mathcal L_\ell=\{0\}\times[\ell,+\infty)$.  Hence
\eqref{eq:halfline-no-axis-force} holds on all of $\mathcal O$.
\end{proof}

\subsection{The full-cluster model}
Assume now
\begin{equation}\label{eq:fullcluster-hyp}
 \mathcal C_0(u)=[-\infty,+\infty].
\end{equation}
Both signs are unbounded and, by the uniform one-ball theorem, have one
ball at every good level.  For a common blow-down the positive and negative
balls are disjoint and their centers satisfy $|c_\pm(s)|\le1/s$.  Hence
\[
 2/s\le |c_+(s)-c_-(s)|\le |c_+(s)|+|c_-(s)|\le2/s.
\]
Equality throughout gives, after fixing an orientation by nested matching,
\[
 c_+(s)=e/s,\qquad c_-(s)=-e/s.
\]
Layer cake therefore yields the full dipole
\begin{equation}\label{eq:full-dipole-limit}
 U=D_e=\frac{2e\cdot x}{|x|^2}.
\end{equation}
Let
\[
 \mathcal L_0=\{0\}\times\mathbb R\subset\mathbb R^3.
\]
The full-cluster hypothesis implies $\mathcal L_0\subset\overline\Sigma$.
Estimate \eqref{common:ambient-area} shows that adding the
$\mathcal H^2$-null set $\mathcal L_0$ produces a locally integral, multiplicity-one
varifold $V$ associated with
\[
 \widehat\Sigma:=\Sigma\cup \mathcal L_0.
\]

Fix $p_0=(0,z_0)\in\mathcal L_0$ and scales $\rho_j\downarrow0$.  Set
\begin{equation*}
 U_j(X)=\rho_j u(\rho_jX),\qquad
 \varepsilon_j=\rho_j^2,\qquad c_j=\rho_jz_0,
\end{equation*}
and
\begin{equation*}
 \zeta_j(X)=\frac{U_j(X)-c_j}{\varepsilon_j}
       =\frac{u(\rho_jX)-z_0}{\rho_j}.
\end{equation*}
The ambient blow-up of $\widehat\Sigma$ about $p_0$ by $\rho_j^{-1}$ is
the completed graph of $\zeta_j$.  After a subsequence, the signed compactness theorem gives a natural
blow-down limit, and the preceding full-cluster one-ball argument identifies
it with a full dipole.  Thus
\begin{equation}
 U_j\to D_e\quad\text{in }L^1_{\mathrm{loc}}(\mathbb R^2)
 \label{fc:natural-limit-at-axis}
\end{equation}
for some $e\in\mathbb S^1$.

We first establish first-variation compactness before invoking any
regularity theorem.

The rescaled graph has scalar mean curvature
\begin{equation*}
 H_j(X,\zeta)=U_j(X)=c_j+\varepsilon_j\zeta.
\end{equation*}
Write \(\mathbf H_j=-H_jN_j\) for its mean-curvature vector, with the
upward subgraph normal \(N_j\) and the convention
\(\delta \mathcal V_j(Y)=-\int\mathbf H_j\cdot Y\,d\|\mathcal V_j\|\) off the axis.
Scaling \eqref{common:shrinking-band} gives, for bounded $L$,
\begin{equation}
 \mathcal H^2\{|X|<2\eta,\ |\zeta_j|<L\}
 \le C_L(\eta L+\eta^2).
 \label{fc:scaled-strip-area}
\end{equation}
Let $\chi_\eta(X)$ vanish for $|X|<\eta$, equal one for
$|X|>2\eta$, and satisfy $|\nabla\chi_\eta|\le C/\eta$.  For
$Y\in C_c^1(\mathbb R^3;\mathbb R^3)$ supported in a fixed height slab,
the first-variation identity away from the axis, applied to
\(\chi_\eta Y\), is
\[
 \int \chi_\eta\operatorname{div}_{T}Y\,d\mathcal V_j(X,T)
 +\int \nabla_T\chi_\eta\cdot Y\,d\mathcal V_j(X,T)
 =-\int \chi_\eta\mathbf H_j\cdot Y\,d\|\mathcal V_j\|.
\]
Since
\[
 \delta \mathcal V_j(Y)
 =\int \chi_\eta\operatorname{div}_{T}Y\,d\mathcal V_j
  +\int(1-\chi_\eta)\operatorname{div}_{T}Y\,d\mathcal V_j,
\]
letting \(\eta\downarrow0\) gives
\begin{equation}
 \begin{split}
 |\delta \mathcal V_j(Y)|
 &\le \int |H_j||Y|\,d\|\mathcal V_j\|\\
 &\quad+\liminf_{\eta\downarrow0}
   \|DY\|_\infty
   \mathcal H^2\{|X|<2\eta,\ |\zeta_j|<L\}\\
 &\quad+\liminf_{\eta\downarrow0}
   \frac{C\|Y\|_\infty}{\eta}
   \mathcal H^2\{|X|<2\eta,\ |\zeta_j|<L\}.
 \end{split}
 \label{fc:cutoff-first-variation}
\end{equation}
By \eqref{fc:scaled-strip-area}, the term containing \(DY\) tends to zero
with $\eta$, while the final cutoff term in
\eqref{fc:cutoff-first-variation} is bounded uniformly in \(j\).  Hence
\begin{equation}
 \sup_j\bigl(\|\mathcal V_j\|(K)+|\delta \mathcal V_j|(K)\bigr)<\infty
 \label{fc:first-variation-compactness}
\end{equation}
on every ambient compact set.  The varifold compactness theorem and its
integrality conclusion
\cite[Chapter~\refnum{8}, Theorem~\refnum{5.8} and Remark~\refnum{5.9}]{SimonGMT2018} give an integral
varifold subsequence $\mathcal V_j\to \mathcal V_\infty$.

We next use the shrinking-height estimate to identify the limiting slices.

Define
\[
 F_j=\frac{\nabla U_j}
          {\sqrt{\varepsilon_j^2+|\nabla U_j|^2}},
 \qquad m_j=|F_j|.
\]
We next identify the limiting horizontal flux.  Fix a compact
$K\Subset\mathbb R^2\setminus\{0\}$ and first remove a strip
$\{|e\cdot X|<\delta\}$.  On each remaining component the dipole has a
fixed sign and is bounded away from zero, so the finite-value strict-$BV$
calibration for the corresponding sign determines the weak-star limit of
$F_j$ as the unit polar field of $D_e$.  Since that field has unit length,
the extreme-point identity
$|F_j-n_e|^2\le2(1-n_e\cdot F_j)$ upgrades the convergence to strong
$L^2$ there.  The omitted strip has area $O_K(\delta)$ and both fields are
bounded by one; letting $\delta\downarrow0$ gives, on every such compact,
\begin{equation}
 F_j\longrightarrow
 n_e:=\frac{\nabla D_e}{|\nabla D_e|}
 \quad\text{strongly in }L^2_{\rm loc},
 \label{fc:flux-to-ne}
\end{equation}
and hence also weak-star in $L^\infty_{\rm loc}$.
Direct polar-coordinate calculation gives
\begin{equation}
 \operatorname{div}n_e=-D_e
 \quad\text{in }\mathcal D'(\mathbb R^2).
 \label{fc:ne-divergence}
\end{equation}
There is no Dirac mass at zero: the flux of $n_e$ through every centered
circle is zero.
Indeed, after rotating \(e=e_1\) and writing
\(X=r(\cos\theta,\sin\theta)\),
\[
 D_e=\frac{2\cos\theta}{r},\qquad
 n_e\cdot e_r=-\cos\theta,\qquad
 n_e\cdot e_\theta=-\sin\theta,
\]
and therefore
\[
 \operatorname{div}n_e
 =\frac1r\partial_r\!\left(r\,n_e\cdot e_r\right)
  +\frac1r\partial_\theta(n_e\cdot e_\theta)
 =-\frac{2\cos\theta}{r}=-D_e.
\]
Moreover,
\(\int_{\partial B_r}n_e\cdot e_r\,d\mathcal H^1
=-r\int_0^{2\pi}\cos\theta\,d\theta=0\), which rules out an atom at
the origin.

For bounded $\zeta$ put
\[
 s_j(\zeta)=c_j+\varepsilon_j\zeta.
\]
For almost every $\zeta$ and nonnegative
$\phi\in C_c^1(\mathbb R^2)$, put
\(E_{j,\zeta}=\{U_j>s_j(\zeta)\}\).  Its outer normal on the reduced
boundary is \(-\nabla U_j/|\nabla U_j|\), and hence
\(F_j\cdot\nu_{E_{j,\zeta}}=-m_j\).  On
\(E_{j,\zeta}\setminus\overline B_\eta\), whose inner circular normal is
\(-e_r\), Gauss--Green for \(\phi F_j\) gives
\[
 \begin{aligned}
 \int_{E_{j,\zeta}\setminus B_\eta}
       (F_j\cdot\nabla\phi-\phi U_j)\,dX
 &=-\int_{\partial^*E_{j,\zeta}\setminus B_\eta}
       \phi m_j\,d\mathcal H^1\\
 &\quad-\int_{E_{j,\zeta}\cap\partial B_\eta}
       \phi F_j\cdot e_r\,d\mathcal H^1.
 \end{aligned}
\]
The last integral is bounded by
\(2\pi\eta\|\phi\|_\infty\).  Letting \(\eta\downarrow0\) gives
\begin{equation}
 \int_{\partial^*\{U_j>s_j(\zeta)\}}\phi m_j\,d\mathcal H^1
 =
 \int_{\{U_j>s_j(\zeta)\}}
       (\phi U_j-F_j\cdot\nabla\phi)\,dX.
 \label{fc:moving-slice-flux}
\end{equation}
The moving thresholds converge uniformly on every bounded
$\zeta$-interval, locally in the plane.  More precisely, for every compact
$K\Subset\mathbb R^2$, for $|\zeta|\le L$, and once
$|s_j(\zeta)|\le\delta$,
\[
 K\cap\bigl(\{U_j>s_j(\zeta)\}\triangle\{D_e>0\}\bigr)
 \subset
 K\cap\bigl(\{|D_e|\le2\delta\}
              \cup\{|U_j-D_e|>\delta\}\bigr).
\]
Consequently,
\begin{equation}
 \sup_{|\zeta|\le L}
 \left|K\cap\bigl(
 \{U_j>s_j(\zeta)\}\triangle\{D_e>0\}
 \bigr)\right|\longrightarrow0
 \qquad(K\Subset\mathbb R^2).
 \label{fc:moving-threshold-uniformity}
\end{equation}
The common Morrey estimate $\int_{B_r}|u|\le Cr$ gives, after this scaling, $\int_{B_\delta}|U_j|\le C\delta$, and hence supplies uniform integrability at zero, while
\eqref{fc:flux-to-ne} handles the flux term away from zero.  Passing to the
limit in \eqref{fc:moving-slice-flux}, taking in
\eqref{fc:moving-threshold-uniformity} a compact $K$ containing
$\operatorname{spt}\phi$, and using
\eqref{fc:ne-divergence}, gives in the essential $L^\infty$ sense on
bounded $\zeta$-intervals
\begin{equation}
 \int_{\partial^*\{U_j>s_j(\zeta)\}}\phi m_j\,d\mathcal H^1
 \longrightarrow
 \int_{e^\perp}\phi\,d\mathcal H^1.
 \label{fc:slice-calibration-limit}
\end{equation}

\begin{proposition}
\label{fc:environmental-plane}
Every convergent subsequence of the ambient blow-ups at any
$p_0\in \mathcal L_0$ has a further subsequence satisfying
\begin{equation*}
 \mathcal V_j\longrightarrow |P_e|,
 \qquad P_e=e^\perp\times\mathbb R,
\end{equation*}
for some $e\in\mathbb S^1$.  The multiplicity is exactly one.
\end{proposition}

\begin{proof}
First identify the support.  Let
$K\Subset\mathbb R^2\setminus e^\perp$, fix $L,M>0$, and put
\[
 A_j=K\cap\{|\zeta_j|\le L\}.
\]
On $A_j$, $U_j=c_j+\varepsilon_j\zeta_j\to0$ uniformly, whereas $|D_e|$ is
bounded below on $K$.  Hence \eqref{eq:full-dipole-limit} gives
$|A_j|\to0$.  Let $q_j=|\nabla \zeta_j|$, so
$m_j=q_j/\sqrt{1+q_j^2}$.  On $\{q_j\le M\}$,
\[
 \int_{A_j\cap\{q_j\le M\}}\sqrt{1+q_j^2}\,dX
 \le\sqrt{1+M^2}\,|A_j|\to0.
\]
On $\{q_j>M\}$,
\[
 \sqrt{1+q_j^2}
 \le(1+M^{-2})\frac{q_j^2}{\sqrt{1+q_j^2}}.
\]
Choose $\phi\in C_c^1(\mathbb R^2\setminus e^\perp)$ equal to one on
$K$.  Coarea and \eqref{fc:slice-calibration-limit} give
\[
 \int_{A_j\cap\{q_j>M\}}
   \frac{q_j^2}{\sqrt{1+q_j^2}}\,dX
 \le
 \int_{-L}^{L}
 \int_{\partial^*\{U_j>s_j(\zeta)\}}\phi m_j\,d\mathcal H^1\,d\zeta
 \longrightarrow0.
\]
Thus graph mass over $K\times[-L,L]$ tends to zero and
$\operatorname{spt}\mathcal V_\infty\subset P_e$.

Since $\mathcal V_\infty$ is integral and supported in a plane, its tangent plane is
$P_e$ almost everywhere.  If $\nu_{j,3}$ denotes the vertical component
of the graph normal, varifold convergence gives
\begin{equation}
 \int\Phi\,\nu_{j,3}^2\,d\|\mathcal V_j\|\longrightarrow0
 \label{fc:vertical-angle-vanishes}
\end{equation}
for every nonnegative compactly supported $\Phi$.  Graph coarea gives
\begin{equation*}
 \int_{\operatorname{graph}\zeta_j}\phi(X)\chi(\zeta)\,d\mathcal H^2
 =
 \int_{\mathbb R}\chi(\zeta)
 \int_{\partial^*\{U_j>s_j(\zeta)\}}
       \frac{\phi}{m_j}\,d\mathcal H^1\,d\zeta.
\end{equation*}
Moreover, exactly,
\begin{equation*}
 \int_{\mathbb R}\chi(\zeta)
 \int_{\partial^*\{U_j>s_j(\zeta)\}}
 \phi\left(\frac1{m_j}-m_j\right)d\mathcal H^1\,d\zeta
 =
 \int_{\operatorname{graph}\zeta_j}
       \phi(X)\chi(\zeta)\nu_{j,3}^2\,d\mathcal H^2.
\end{equation*}
The right side tends to zero by
\eqref{fc:vertical-angle-vanishes}; the calibrated $m_j$ term tends by
\eqref{fc:slice-calibration-limit}.  Therefore
\[
 \int\phi(X)\chi(\zeta)\,d\|\mathcal V_\infty\|
 =
 \left(\int_{e^\perp}\phi\,d\mathcal H^1\right)
 \left(\int_{\mathbb R}\chi\,d\zeta\right),
\]
which is precisely the multiplicity-one plane.
\end{proof}

\begin{proposition}\label{prop:fullcluster-axis-force-vanishes}
In the full-cluster branch, the completed varifold has no axis-supported
first-variation defect.  Equivalently, after subtracting its smooth
prescribed-mean-curvature part, the remaining vector Radon measure is zero.
\end{proposition}
\begin{proof}
Away from $\mathcal L_0$, the original varifold $V$ has its smooth
prescribed-mean-curvature first variation.  The same cutoff estimate used
for \eqref{fc:first-variation-compactness} proves that $\delta V$ is a
Radon measure on ambient compact sets.  Subtract the smooth
mean-curvature part and call the remainder $\mu$; then
$\operatorname{spt}\mu\subset\mathcal L_0$.

Fix a compact vertical interval $J$.  For every subinterval $I\subset J$,
choose a smooth field \(Y_\delta\) with
\(\|Y_\delta\|_\infty\le1\), whose trace on \(\mathcal L_0\) agrees with the prescribed
test field over \(I\) and is supported in an arbitrarily small enlargement
\(I_\delta\).  Let \(\psi_r(x)\) equal one on \(B_r\), vanish outside
\(B_{2r}\), and satisfy \(|D\psi_r|\le C/r\).  Since
\(\operatorname{spt}\mu\subset\mathcal L_0\),
\[
 \begin{aligned}
 \mu(Y_\delta)
 &=\mu(\psi_rY_\delta)\\
 &=\int_\Sigma \psi_r\operatorname{div}_\Sigma Y_\delta
       \,d\mathcal H^2
   +\int_\Sigma \nabla_\Sigma\psi_r\cdot Y_\delta
       \,d\mathcal H^2
   +\int_\Sigma \psi_r\mathbf H_\Sigma\cdot Y_\delta
       \,d\mathcal H^2.
 \end{aligned}
\]
The first and third terms are bounded by
\[
 C_{J,\delta}\mathcal H^2\bigl(
 \Sigma\cap\{|x|<2r,\ u\in I_\delta\}\bigr)
 =O_{J,\delta}(r(|I|+2\delta)+r^2)
\]
and vanish as $r\downarrow0$.  The middle term is bounded by
\[
 \frac Cr\,
 \mathcal H^2\bigl(
 \Sigma\cap\{|x|<2r,\ u\in I_\delta\}\bigr)
 \le C_J(|I|+2\delta+r).
\]
First let $r\downarrow0$ and then $\delta\downarrow0$.  Thus, for every
smooth axis test field $\varphi$ supported in $I$ with
$\|\varphi\|_\infty\le1$, a smooth ambient extension gives
$|\mu(\varphi)|\le C_J|I|$.  Approximating continuous axis fields by smooth
ones and taking the Radon-measure dual supremum yields
\begin{equation*}
 |\mu|(\{0\}\times I)\le C_J|I|.
\end{equation*}
Hence
\begin{equation*}
 \mu=\mathbf g_{\rm ax}(z)\mathcal H^1\lfloor\mathcal L_0
\end{equation*}
for a locally bounded vector-valued density $\mathbf g_{\rm ax}$.

At a Lebesgue point $z_0$ of $\mathbf g_{\rm ax}$, blow up about $(0,z_0)$.  If
$\eta_{p_0,\rho}(p)=(p-p_0)/\rho$ and
$V_\rho=(\eta_{p_0,\rho})_\#V$, then
\[
 \delta V_\rho(Y)=
 \rho^{-1}\delta V(Y\circ\eta_{p_0,\rho}).
\]
For a test field supported in a fixed ball, the smooth part satisfies
\[
 \left|\rho^{-1}\int
   \mathbf H_\Sigma\cdot(Y\circ\eta_{p_0,\rho})\,d\|V\|\right|
 \le C\rho^{-1}\|V\|(B_{C\rho}(p_0))=O(\rho),
\]
whereas the axis-supported part becomes
\[
 \rho^{-1}\int_{\mathcal L_0} \mathbf g_{\rm ax}(z)\cdot
 Y\!\left(0,\frac{z-z_0}{\rho}\right)d\mathcal H^1(z)
 =\int_\mathbb R \mathbf g_{\rm ax}(z_0+\rho\zeta)\cdot Y(0,\zeta)\,d\zeta
 \longrightarrow
 \int_\mathbb R \mathbf g_{\rm ax}(z_0)\cdot Y(0,\zeta)\,d\zeta.
\]
Thus the rescaled defect converges to
$\mathbf g_{\rm ax}(z_0)\mathcal H^1\lfloor\mathcal L_0$.  On the other hand,
Proposition~\ref{fc:environmental-plane} makes every such tangent a
multiplicity-one plane, whose first variation is zero.  For a fixed smooth
test field $Y$, the function $(x,S)\mapsto\operatorname{div}_S Y(x)$ is
continuous and compactly supported on the Grassmann bundle; hence varifold
convergence itself gives convergence of $\delta V_\rho(Y)$.  Therefore
$\mathbf g_{\rm ax}(z_0)=0$.  Thus
\begin{equation}\label{eq:fullcluster-axis-defect-zero}
 \mu=0.
\end{equation}
\end{proof}

\subsection{Axis defect equals the derivative of translation flux}
\begin{proposition}\label{prop:axis-flux-derivative}
In either planar exceptional branch, fix a working ball
\(B_{r_0}\Subset B_R\).  Choose \(T_0>0\) above the maximum of \(u\)
on the closed collar \(\overline B_{r_0}\setminus B_{r_0/2}\), so that
all levels above \(T_0\), localized to \(B_{r_0}\), lie in \(B_{r_0/2}\).
Let \(I\subset(T_0,\infty)\) be an open interval and suppose that,
for test fields compactly supported in \(B_{r_0}\times I\),
\[
 \delta V(Y)=-\int \mathbf H_\Sigma\cdot Y\,d\|V\|
 +\int_I \mathbf g_{\rm ax}(z)\cdot Y(0,z)\,dz.
\]
Let $\mathbf g_{\rm ax,h}$ denote the horizontal component and let $\mathcal F$ be the
horizontal level flux from \eqref{eq:physical-translation-flux}.  Then
\begin{equation}\label{eq:axis-flux-identity}
 D_s\mathcal F=\mathbf g_{\rm ax,h}
\end{equation}
in distributions on $I$.  In particular, if $\mathbf g_{\rm ax}=0$ on a high half-line,
then $\mathcal F=0$ there.
\end{proposition}
\begin{proof}
Fix a constant horizontal vector $e\in\mathbb R^2$ and
$\chi\in C_c^1(I)$.  Choose \(\eta\in C_c^1(B_{r_0})\) equal to one
on \(B_{r_0/2}\), and use the compactly supported ambient vector field
\[
 Y(x,z)=\eta(x)\chi(z)(e,0).
\]
On every graph point with height in \(\operatorname{spt}\chi\), one has
\(\eta=1\) and \(D\eta=0\), by the choice of \(T_0\).  Thus the
spatial cutoff creates no graph term, and \(\eta(0)=1\) preserves the
axis term.
On the smooth graph, with upward normal $N=(N_h,\alpha)$,
\[
 \nabla_\Sigma z=e_3-\alpha N.
\]
Since the ambient horizontal field $(e,0)$ is constant,
\[
 \operatorname{div}_\Sigma Y
 =\chi'(z)(e,0)\cdot\nabla_\Sigma z
 =-\chi'(z)\alpha N_h\cdot e.
\]
Writing $\sigma=|N_h|$ and $N_h=\sigma\nu_s$ on the horizontal slice
$z=s$, graph coarea therefore gives
\begin{equation}\label{eq:axis-flux-coarea}
 \delta V(Y)
 =-\int_I\chi'(s)\mathcal F(s)\cdot e\,ds.
\end{equation}
Here the orientation agrees with \eqref{eq:physical-translation-flux}:
$\nu_s$ is the outer normal of the superlevel $\{u>s\}$.

The classical mean-curvature contribution contains no horizontal source.
Indeed \(\mathbf H_\Sigma=-sN\) on the slice \(z=s\), and graph coarea
rewrites \(-\int\mathbf H_\Sigma\cdot Y\,d\|V\|\) as
\[
 \int_I s\chi(s)\,e\cdot
 \left(\int_{\partial^*E_s}\nu_s\,d\mathcal H^1\right)ds,
\]
which vanishes by finite-perimeter Gauss--Green at almost every such high level,
applied to constant vector fields.  Comparing \eqref{eq:axis-flux-coarea}
with the axis term in the assumed first-variation formula gives
\[
 -\int_I\chi'(s)\mathcal F(s)\cdot e\,ds
 =\int_I\chi(s)\mathbf g_{\rm ax,h}(s)\cdot e\,ds.
\]
Since $e$ and $\chi$ are arbitrary, this is exactly
$D_s\mathcal F=\mathbf g_{\rm ax,h}$.

If $\mathbf g_{\rm ax}=0$ on $(T,\infty)$ with \(T\ge T_0\), then $\mathcal F$ is distributionally constant
there.  On regular levels,
\[
 |\mathcal F(s)|
 \le\int_{\partial^*E_s}\alpha\,d\mathcal H^1
 \le\int_{\partial^*E_s}\frac\alpha\sigma\,d\mathcal H^1
 =-V'(s).
\]
Because $-V'\in L^1(T,\infty)$, the flux itself belongs to
$L^1(T,\infty)$.  The only constant vector in this space is zero, so
$\mathcal F\equiv0$ a.e.\ on the high tail.
\end{proof}

\begin{lemma}\label{lem:shifted-zero-flux-centering}
Assume $n=2$.  Let $r_j\downarrow0$, put $t_j=r_j^{-1}$, and let
\[
 U_j(y)=r_j\bigl(u(r_jy)-\ell\bigr),\qquad
 F_j(y)=\frac{DU_j}{\sqrt{r_j^4+|DU_j|^2}}.
\]
Suppose, after a subsequence,
\[
 U_j\to U\quad\text{strongly in }L^1_{\rm loc}(\mathbb R^2),\qquad
 F_j\weakstar F,
\]
and for almost every $s>0$
\[
 \{U>s\}=B_{1/s}(c(s)),\qquad |c(s)|\le s^{-1},
\]
with the bounded positive truncations calibrated by $F$.  If the physical
positive translation flux $\mathcal F(q)$ from
\eqref{eq:physical-translation-flux} vanishes for almost every sufficiently
large $q$, then $c\equiv0$.
\end{lemma}
\begin{proof}
The limiting field satisfies the divergence hypothesis required by
Lemma~\ref{lem:bridge-moving-ball-area}.  Indeed, the rescaled equation gives
\[
 \operatorname{div}F_j=-r_j u(r_j\,\cdot)=-(U_j+r_j\ell).
\]
Since $U_j\to U$ strongly in $L^1_{\rm loc}$ and $F_j\weakstar F$, passage
to distributions yields
\[
 \operatorname{div}F=-U\in L^1_{\rm loc}(\mathbb R^2).
\]
For every $\eta\in C_c^1((0,\infty))$, coarea gives the exact shifted
scaling identity
\begin{equation}\label{eq:shifted-translation-flux-scaling}
 \int_{\mathbb R^2}\eta(U_j)F_j\,dy
 =-t_j^3\int_0^\infty
   \eta(s)\,\mathcal F(\ell+t_js)\,ds.
\end{equation}
Indeed, on $\{U_j=s\}$ the corresponding physical level is
$u=\ell+t_js$, while
$|DU_j|=t_j^{-2}|Du|$, $F_j=-\sigma\nu$, and
$d\mathcal H_y^1=t_j\,d\mathcal H_x^1$.  Thus the factor in
\eqref{eq:shifted-translation-flux-scaling} is exactly $t_j^3$.

If $\operatorname{spt}\eta\subset[a,b]\Subset(0,\infty)$, the uniform
one-ball theorem at the neighboring physical levels
$\ell+t_js\sim t_js$ gives global tightness of $\{U_j>a\}$.  Hence
$\eta(U_j)\to\eta(U)$ strongly in global $L^1(\mathbb R^2)$.  Since
$\ell+t_js$ is above the zero-flux threshold for every
$s\in\operatorname{spt}\eta$ and all large $j$, the right-hand side of
\eqref{eq:shifted-translation-flux-scaling} is zero.  Passing to the limit
therefore gives
\[
 \int\eta(U)F\,dy=0.
\]
Applying the moving-ball formula
\eqref{eq:bridge-moving-ball-vector-formula} on an arbitrary compact
positive value band yields
\[
 0=\pi\int \eta(s)\frac{c'(s)}{s}\,ds.
\]
Thus $c'=0$ almost everywhere.  The nested-ball estimate makes $c$ locally
Lipschitz, so $c$ is constant; the bound $|c(s)|\le s^{-1}$ forces this
constant to be zero as $s\to\infty$ through good levels.
\end{proof}

\begin{theorem}\label{thm:planar-exceptional-exclusion}
Neither a finite half-line cluster nor the full extended-line cluster can
occur in dimension two.
\end{theorem}
\begin{proof}
For the half-line branch, Proposition~\ref{prop:halfline-axis-force-vanishes} and
Proposition~\ref{prop:axis-flux-derivative} give zero high translation flux.
Apply Lemma~\ref{lem:shifted-zero-flux-centering} to the shifted natural
rescalings from \eqref{eq:halfline-shifted-natural-scaling}.  It forces
the center path to be constant and equal to zero, whereas
Proposition~\ref{prop:halfline-half-dipole-blowdown} gives $c(s)=e/s$.  This is a
contradiction.

For the full-cluster branch,
Propositions~\ref{prop:fullcluster-axis-force-vanishes} and~\ref{prop:axis-flux-derivative} give
zero positive translation flux on
the high tail.  Apply the unshifted centering theorem
Proposition~\ref{prop:zero-charge-centers} to any natural blow-down sequence.  It
forces the positive one-ball center path to be identically zero, whereas
the positive phase of \eqref{eq:full-dipole-limit} has
$c_+(s)=e/s$.  This is again a contradiction.  The negative finite
half-line is obtained by the symmetry $u\mapsto-u$.
\end{proof}

\begin{corollary}\label{cor:planar-genuine-one-tail}
If $n=2$ and the positive tail is unbounded while the negative tail is not,
then $u$ is bounded below near the puncture, the positive blow-down is
centered, and \eqref{eq:wholeblowdown-pressure} holds with $k=1$.  The
analogous statement holds for the negative tail.
\end{corollary}
\begin{proof}
By Lemma~\ref{lem:cluster-interval} and Theorem~\ref{thm:planar-exceptional-exclusion}, an
unbounded positive tail that is not accompanied by an unbounded negative
tail has cluster set $\{+\infty\}$; thus the convergence is uniform.
Hence, for any sufficiently high regular levels $t<T$, the finite-height
slab $\{t<u<T\}$ avoids the puncture.  Integrating $\operatorname{div}\mathsf S=0$
on this slab gives
\[
 \int_{\partial^*E_t}\left(\frac{t^2}{2}-\alpha\right)\nu_t\,d\mathcal H^1
 =\int_{\partial^*E_T}\left(\frac{T^2}{2}-\alpha\right)\nu_T\,d\mathcal H^1.
\]
Since each $E_s$ is a bounded finite-perimeter set,
$\int_{\partial^*E_s}\nu_s\,d\mathcal H^1=0$.  Therefore the physical
translation flux $\mathcal F(s)=\int_{\partial^*E_s}\alpha\nu_s$ is
constant on the high tail.  On the common full-measure set of regular
levels,
\[
 |\mathcal F(s)|\le\int_{\partial^*E_s}\frac{\alpha}{\sigma_s}\,d\mathcal H^1
 =-V'(s),
 \qquad
 \int_T^\infty[-V'(s)]\,ds<\infty.
\]
The only constant vector satisfying this bound is zero.  Thus
$\mathcal F=0$ almost everywhere on the high tail, and
Proposition~\ref{prop:zero-charge-centers} applies.
\end{proof}
 \section{Pressure graphs and the smooth one-sheeted end}
\label{sec:pressure-complete}

We work in the positive one-tail branch and retain the localized sets
$E_s=E^+_{s,r_0}$, their volumes and radii, and the natural rescaling on
$sB_{r_0}$ from the preceding sections.  Every pressure graph below is
built only from points $0<|x|<r_0$; consequently every level-set conclusion
in this section is a conclusion inside $B_{r_0}$.  The preceding one-tail analysis (and, in dimension two, Theorem~\ref{thm:planar-exceptional-exclusion}) gives numbers $M>0$ and $r_1>0$ such that
\begin{equation}
u\geq-M\qquad\hbox{on }B_{r_1}\setminus\{0\}.
\end{equation}
The whole-family cluster theorem also gives
\begin{equation}
w_t(y)=t^{-1}u(y/t)\longrightarrow \frac{k}{|y|}
\quad\hbox{in }L^1_{\rm loc}(\R^n),
\qquad C_*=\omega_nk^n.
\end{equation}
The purpose of this section is to upgrade this measure-theoretic statement to a smooth,
one-sheeted foliation by all sufficiently high localized levels.  The only new issue relative to a filled
punctured end is the possible vertical-axis boundary of the pressure graph.

Put
\begin{equation}\label{eq:pressure-definitions}
z_t(y)=t^2(w_t(y)-1),\qquad
a_t=(t^{-4}+|Dw_t|^2)^{1/2},\qquad
N_t=\left(-\frac{Dw_t}{a_t},\frac{t^{-2}}{a_t}\right).
\end{equation}
The vector field $N_t$ is extended constantly in the last variable.  It has unit length and,
by the rescaled equation,
\begin{equation}\label{eq:extended-normal-divergence}
\operatorname{div}_{y,z}N_t=-\operatorname{div}_y Z_t=w_t,
\qquad Z_t=\frac{Dw_t}{a_t}.
\end{equation}

\begin{lemma}\label{lem:axis-mass}
For every $L>0$ there are $t_L$ and $C_L$, independent of $t\geq t_L$ and
$0<\eta\leq1$, such that
\begin{equation}\label{eq:uniform-axis-mass}
\mathcal H^n\bigl(\widehat\Sigma_t^L\cap
(B_\eta\times(-L,L))\bigr)
\leq C_L(\eta^{n-1}+\eta^n),
\end{equation}
where
\begin{equation}\label{eq:pressure-slab-definition}
\widehat\Sigma_t^L=
\{(tx,t(u(x)-t)):0<|x|<r_0, |u(x)-t|<L/t\}.
\end{equation}
\end{lemma}

\begin{proof}
Fix $0<\delta<\eta\leq1$ and let
\[
D_{\delta,\eta}=B_\eta\setminus\overline{B_\delta},
\qquad \bar z_t(y)=\max\{-L,\min\{z_t(y),L\}\}.
\]
For large $t$, the points $y/t$ with $y\in B_1$ lie in $B_{r_1}$, so
\eqref{eq:lower-bound-after-twotail} gives $(w_t)_-\leq M/t$ there.  Apply the divergence
theorem to \eqref{eq:extended-normal-divergence} on the Lipschitz region
\[
\mathcal R_{\delta,\eta}
=\{(y,z):y\in D_{\delta,\eta},\ \bar z_t(y)<z<L\}.
\]
On the part of the lower boundary on which $|z_t|<L$, the outward normal is
$(Dz_t,-1)/(1+|Dz_t|^2)^{1/2}$.  Since $Dz_t=t^2Dw_t$,
\[
 \sqrt{1+|Dz_t|^2}=t^2a_t,
 \qquad
 N_t\cdot\frac{(Dz_t,-1)}{\sqrt{1+|Dz_t|^2}}
 =\left(-\frac{Dw_t}{a_t},\frac{t^{-2}}{a_t}\right)
 \cdot\left(\frac{Dw_t}{a_t},-\frac{t^{-2}}{a_t}\right),
\]
and hence
\begin{equation}\label{eq:bottom-calibration}
N_t\cdot\frac{(Dz_t,-1)}{(1+|Dz_t|^2)^{1/2}}=-1.
\end{equation}
Let
\[
 \Gamma_{\delta,\eta}
 :=\widehat\Sigma_t^L\cap
 ((B_\eta\setminus\overline B_\delta)\times(-L,L)).
\]
Writing \(F_{\rm top},F_{\rm out},F_{\rm in}\), and \(F_{\rm clip}\)
for the remaining boundary fluxes, the divergence theorem reads
\[
 \int_{\mathcal R_{\delta,\eta}}w_t
 =F_{\rm top}+F_{\rm out}+F_{\rm in}
   -\mathcal H^n(\Gamma_{\delta,\eta})+F_{\rm clip}.
\]
On the clipped floor \(\bar z_t=-L\),
\[
 F_{\rm clip}
 =-\int_{D_{\delta,\eta}\cap\{z_t\leq-L\}}
   \frac{t^{-2}}{a_t}\,dy\leq0,
\]
while
\[
 F_{\rm top}\leq|D_{\delta,\eta}|,
 \qquad
 |F_{\rm out}|+|F_{\rm in}|
 \leq2L\{\mathcal H^{n-1}(\partial B_\eta)
          +\mathcal H^{n-1}(\partial B_\delta)\}.
\]
Consequently
\[
\begin{aligned}
\mathcal H^n(\Gamma_{\delta,\eta})
&\leq |D_{\delta,\eta}|
 +2L\{\mathcal H^{n-1}(\partial B_\eta)
          +\mathcal H^{n-1}(\partial B_\delta)\}
 -\int_{\mathcal R_{\delta,\eta}}w_t.
\end{aligned}
\]
The source term is controlled by
\[
-\int_{\mathcal R_{\delta,\eta}}w_t\,dy\,dz
=-\int_{D_{\delta,\eta}}w_t(L-\bar z_t)\,dy
\leq2L\int_{D_{\delta,\eta}}(w_t)_-\,dy
\leq C_Lt^{-1}\eta^n.
\]
Substitution gives
\begin{equation}\label{eq:annular-axis-mass}
\mathcal H^n\bigl(\widehat\Sigma_t^L\cap
((B_\eta\setminus B_\delta)\times(-L,L))\bigr)
\leq C_L(\eta^{n-1}+\delta^{n-1}+\eta^n).
\end{equation}
Letting $\delta\downarrow0$ and using monotone convergence proves
\eqref{eq:uniform-axis-mass}.  Notice that no regular-value assertion has been used.
\end{proof}

The next lemma closes the pressure current across the missing vertical axis;
this must be done before any slicing argument.

\begin{lemma}
\label{lem:axis-current-closed}
Fix $L>0$ and take $t$ sufficiently large.  In the open slab
\[
 \mathcal O_L:=\mathbb R^n\times(-L,L),
\]
the pressure graph has a canonical closure as a locally integral
$n$-current $\mathcal T_t^L$ with
\begin{equation}\label{eq:axis-current-zero-boundary}
 \partial\mathcal T_t^L=0\qquad\text{in }\mathcal O_L.
\end{equation}
This conclusion holds for every $n\ge2$; no $r^{n-2}$ cutoff is used.
\end{lemma}

\begin{proof}
Write $\Omega_t=tB_{r_0}$ and recall
\[
 z_t(y)=t^2(w_t(y)-1)=t\bigl(u(y/t)-t\bigr)
 \qquad (y\in\Omega_t\setminus\{0\}).
\]
Let
\[
 \bar z_t(y):=\max\{-L,\min\{z_t(y),L\}\}.
\]
For fixed $t$ the finite-height estimate applied to the physical band
$t-L/t<u<t+L/t$ gives
\[
 \int_{\Omega_t\cap\{|z_t|<L\}}|D z_t|\,dy<\infty,
\]
because $D_yz_t(y)=Du(y/t)$.  Since $\bar z_t$ is bounded, an inner
sphere integration by parts has boundary cost $O(L\eta^{n-1})$; hence the
bounded truncation extends across $y=0$ as a function in
$W^{1,1}_{\rm loc}(\Omega_t)$.  Moreover, for large $t$ the function $u$
is bounded on a collar of $\partial B_{r_0}$ while $t-L/t\to\infty$.
Consequently $z_t<-L$ on a collar of $\partial\Omega_t$, so that
$\bar z_t=-L$ there.  Extending $\bar z_t$ by the constant $-L$ outside
$\Omega_t$ therefore produces a globally defined function
\[
 \bar z_t\in W^{1,1}_{\rm loc}(\mathbb R^n),
 \qquad \bar z_t+L\ \hbox{compactly supported}.
\]

Define the relative subgraph in the open slab by
\begin{equation}\label{eq:pressure-relative-subgraph}
 \mathcal E_t^L
 :=\{(y,z)\in\mathcal O_L:-L<z<\bar z_t(y)\}.
\end{equation}
The standard subgraph formula for a $W^{1,1}$ function shows that
$\mathcal E_t^L$ has locally finite perimeter in $\mathcal O_L$.  Since a
$W^{1,1}$ function has no jump part, its relative reduced boundary in the
open slab is precisely the graph of $\bar z_t$ at those points where
$-L<\bar z_t<L$.  The regions where $\bar z_t=L$ meet the top face
$z=L$, which lies outside the relative boundary of $\mathcal O_L$, while
$\bar z_t=-L$ contributes no interior boundary.  Thus, modulo an
$\mathcal H^n$-null set,
\[
 \partial^*\mathcal E_t^L\cap\mathcal O_L
 =\widehat\Sigma_t^L.
\]
On this graph the measure-theoretic outward unit normal of the subgraph is
\[
 \frac{(-D z_t,1)}{\sqrt{1+|D z_t|^2}}
 =\left(-\frac{Dw_t}{a_t},\frac{t^{-2}}{a_t}\right)=N_t,
\]
so the induced boundary orientation is exactly the upward pressure-graph
orientation used below.

Now define the current intrinsically by
\begin{equation}\label{eq:pressure-current-as-boundary}
 \mathcal T_t^L:=\partial[[\mathcal E_t^L]]\qquad\hbox{in }\mathcal O_L.
\end{equation}
Then $\mathcal T_t^L\in\mathbf I_{n,\rm loc}(\mathcal O_L)$ and, by the boundary of
boundary identity for relative currents,
\[
 \partial\mathcal T_t^L=\partial^2[[\mathcal E_t^L]]=0
 \qquad\hbox{in }\mathcal O_L.
\]
Thus the missing vertical axis produces no current boundary in any
ambient dimension $n\ge2$.
\end{proof}

For fixed \(L>0\), let
\begin{equation}\label{eq:pressure-varifold-notation}
 \mathcal V_t^L:=\mathbf v(\widehat\Sigma_t^L,1),\qquad
 \mu_t^L:=\|\mathcal V_t^L\|
 =\mathcal H^n\llcorner\widehat\Sigma_t^L,
 \qquad \widehat N_t:=N_t|_{\widehat\Sigma_t^L}.
\end{equation}
Thus \(\mathcal V_t^L\) is the multiplicity-one integral varifold
associated with the pressure graph and \(\mu_t^L\) is its weight measure;
the closure across the axis changes neither measure.  When \(L\) is fixed
we suppress it from the notation.  For a sequence \(t_j\to\infty\) we write
\(\mathcal V_j=\mathcal V_{t_j}\), \(\mu_j=\mu_{t_j}\), and
\(\widehat N_j=\widehat N_{t_j}\).

For $-L<z<L$ define
\begin{equation}\label{eq:pressure-slices}
s_t(z)=t+\frac zt,
\qquad \mathcal S_t(z)=tE_{s_t(z)}.
\end{equation}
To justify this moving threshold, observe first that
\[
 \mathcal S_t(z)=\frac{t}{s_t(z)}\bigl(s_t(z)E_{s_t(z)}\bigr),
 \qquad \frac{t}{s_t(z)}\longrightarrow1.
\]
By \eqref{eq:wholeblowdown-pressure}, $w_s\to k/|y|$ locally in $L^1$.
The value $1$ is a continuity level of the limit because
$\{k/|y|=1\}=\partial B_k$ has zero Lebesgue measure.  Therefore
\[
 \mathbf1_{sE_s}
 =\mathbf1_{\{y\in sB_{r_0}\setminus\{0\}:w_s(y)>1\}}
 \longrightarrow\mathbf1_{B_k}
 \quad\hbox{locally in }L^1(\mathbb R^n).
\]
The critical volume law also gives
\[
 |sE_s|=s^nV(s)\longrightarrow\omega_nk^n=|B_k|.
\]
Local convergence plus convergence of total volumes is global symmetric-difference
convergence, because for every \(R>k\),
\[
 |sE_s\mathbin\triangle B_k|
 \leq |(sE_s\cap B_R)\mathbin\triangle B_k|
      +\bigl||sE_s|-|sE_s\cap B_R|\bigr|\longrightarrow0.
\]
Writing \(\lambda_t=t/s_t(z)\to1\), dilation gives
\[
 |\mathcal S_t(z)\mathbin\triangle B_k|
 \leq \lambda_t^n|s_t(z)E_{s_t(z)}\mathbin\triangle B_k|
      +|\lambda_t B_k\mathbin\triangle B_k|\longrightarrow0,
\]
which is
\begin{equation}\label{eq:all-slice-L1}
\mathbf1_{\mathcal S_t(z)}\longrightarrow\mathbf1_{B_k}
\quad\hbox{in }L^1(\R^n)
\end{equation}
for every fixed \(z\), completing the moving-threshold argument.

\begin{proposition}\label{prop:pressure-spacetime-strict-bv}
With $\mathcal E_t^L$ defined by \eqref{eq:pressure-relative-subgraph},
\begin{equation}\label{eq:pressure-spacetime-L1}
 \mathbf1_{\mathcal E_t^L}
 \longrightarrow\mathbf1_{B_k\times(-L,L)}
 \qquad\hbox{in }L^1(\mathcal O_L),
\end{equation}
and
\begin{equation}\label{eq:pressure-spacetime-perimeter}
 \operatorname{Per}(\mathcal E_t^L;\mathcal O_L)
 \longrightarrow2L\,\mathcal H^{n-1}(\partial B_k).
\end{equation}
Hence the convergence is strict in $BV(\mathcal O_L)$ and
\begin{equation}\label{eq:pressure-subgraph-current-limit}
 \partial[[\mathcal E_t^L]]
 \rightharpoonup
 \partial[[B_k\times(-L,L)]]
 \qquad\hbox{as currents in }\mathcal O_L,
\end{equation}
with the outward horizontal orientation on the limiting cylinder.
\end{proposition}
\begin{proof}
For a.e.\ fixed $z\in(-L,L)$ the horizontal section of
$\mathcal E_t^L$ is $\mathcal S_t(z)=tE_{t+z/t}$.  The moving-threshold convergence
\eqref{eq:all-slice-L1} gives
$\mathbf1_{\mathcal S_t(z)}\to\mathbf1_{B_k}$ in global $L^1(\mathbb R^n)$.
For $|z|\le L$, the critical volume law at the neighboring level
$t+z/t\sim t$ gives a uniform bound $|\mathcal S_t(z)|\le C_L$.  Fubini and dominated
convergence therefore yield \eqref{eq:pressure-spacetime-L1}.

The relative reduced boundary of $\mathcal E_t^L$ in $\mathcal O_L$ is
exactly the pressure graph by Lemma~\ref{lem:axis-current-closed}; hence
\[
 \operatorname{Per}(\mathcal E_t^L;\mathcal O_L)
 =\mathcal H^n(\widehat\Sigma_t^L).
\]
The finite-band coarea identity \eqref{eq:finite-graph-band-identity} gives the exact
formula
\[
 \mathcal H^n(\widehat\Sigma_t^L)
 =t^n\int_{t-L/t}^{t+L/t}
    \left(J(s)+\frac1n\mathcal R_{\rm def}(s)\right)\,ds.
\]
On this interval $s/t\to1$ uniformly and
$s^{n-1}J(s)\to P_k:=\mathcal H^{n-1}(\partial B_k)$, so the $J$-term
tends to $2LP_k$.  The defect term tends to zero by
\eqref{eq:narrow-R-vanishing}.  Hence
\[
 \mathcal H^n(\widehat\Sigma_t^L)\longrightarrow2LP_k,
\]
which is exactly the perimeter of the relative cylinder
$B_k\times(-L,L)$ in the open slab.  Together with
\eqref{eq:pressure-spacetime-L1}, this proves strict $BV$ convergence.

For the current convergence no separate Reshetnyak argument is needed:
$L^1$ convergence of the characteristic functions gives
$[[\mathcal E_t^L]]\rightharpoonup[[B_k\times(-L,L)]]$, and continuity of
the boundary operator gives \eqref{eq:pressure-subgraph-current-limit}.
The subgraph outward normal along the top graph is $N_t$, so the limiting
orientation is the outward horizontal orientation of the cylinder.
\end{proof}

\begin{proposition}\label{prop:slab-current-varifold}
For every $L>0$,
\begin{equation}\label{eq:slab-varifold-limit}
\widehat\Sigma_t^L\longrightarrow
\partial B_k\times(-L,L)
\end{equation}
as \(t\to\infty\), as integral varifolds with multiplicity one and as
oriented currents with outward horizontal orientation.  For every
\(0<L'<L\),
\[
d_H\!\left(
 \operatorname{spt}\widehat\Sigma_t^L\cap(\mathbb R^n\times[-L',L']),
 \partial B_k\times[-L',L']
\right)\longrightarrow0.
\]
\end{proposition}

\begin{proof}
Fix an arbitrary sequence $t_j\to\infty$.

\noindent\textbf{Step 1: current and varifold convergence from strict $BV$.}
Proposition~\ref{prop:pressure-spacetime-strict-bv} gives
\[
 \mathbf1_{\mathcal E_{t_j}^L}\to
 \mathbf1_{B_k\times(-L,L)}
 \quad\text{strictly in }BV(\mathcal O_L).
\]
The oriented boundary currents already converge by
\eqref{eq:pressure-subgraph-current-limit}.  Let $\nu_j$ denote the
measure-theoretic outward normal of $\mathcal E_{t_j}^L$ and let
$\nu_\infty$ be that of the solid cylinder.  Strict $BV$ convergence and
Reshetnyak continuity imply, for every compactly supported continuous
function $\Phi(x,S)$ on the Grassmann bundle,
\[
 \int \Phi(x,\nu_j(x)^\perp)\,d|D\mathbf1_{\mathcal E_{t_j}^L}|(x)
 \longrightarrow
 \int \Phi(x,\nu_\infty(x)^\perp)\,d|D\mathbf1_{B_k\times(-L,L)}|(x).
\]
Since the relative reduced boundaries in the slab are exactly the pressure
graphs, this is precisely convergence to the multiplicity-one cylinder as
integral varifolds.  No separate slice reconstruction is needed.

\noindent\textbf{Step 2: first variation across the vertical axis.}
Let $\chi_\eta(y)$ vanish on $B_\eta$, equal one outside $B_{2\eta}$,
and satisfy $|D\chi_\eta|\le C/\eta$.  For a compactly supported vector
field $X$ in a smaller slab, the smooth first-variation identity for
$\chi_\eta X$ contains three errors near the axis: the cutoff-gradient
term, the omitted $\operatorname{div}_{\Sigma}X$ term, and the omitted
mean-curvature term.  Lemma~\ref{lem:axis-mass} gives
\[
 \frac{C}{\eta}\,
 \mathcal H^n\bigl(\widehat\Sigma_t^L\cap
 (B_{2\eta}\times(-L,L))\bigr)
 \le C_L(\eta^{n-2}+\eta^{n-1}),
\]
while the other two errors are bounded by a constant times the same local
mass.  Hence, when $n\ge3$, all three errors vanish as
$\eta\downarrow0$, and the closed varifold satisfies
\[
 \delta\mathcal V_t(X)=-\int H_t\cdot X\,d\mu_t
\]
with no singular contribution on the axis.  Since
$w_t=1+z_t/t^2$, the graph equation gives
\begin{equation}\label{eq:rescaled-mean-curvature}
 \widehat H_t=1+\frac{z}{t^2},\qquad H_t=-\widehat H_t\widehat N_t.
\end{equation}
Thus $|H_t|$ is uniformly bounded on every fixed slab.  If $n=2$, the
exceptional-cluster analysis has already shown $u(x)\to+\infty$ uniformly;
for fixed $t,L$ the band $|u-t|<L/t$ therefore misses a sufficiently small
punctured disk, so the same first-variation formula holds without an axis
cutoff.

\noindent\textbf{Step 3: Hausdorff convergence of the supports.}
Fix $0<L'<L''<L$.  Suppose that points
$x_j\in\operatorname{spt}\widehat\Sigma_{t_j}^L$ with $|z(x_j)|\le L'$
stayed a fixed distance from $\partial B_k\times[-L'',L'']$.  Choose a
small fixed $r>0$.  Since $x_j$ lies in the support, choose a rectifiable
density point $x_j^\ast\in B_{r/4}(x_j)$ with density at least one.  The
bounded-mean-curvature monotonicity formula gives
\[
 \mu_j(B_r(x_j))\ge c r^n
\]
with $c>0$ independent of $j$.  If $(x_j)$ is horizontally bounded, this
contradicts Step~1.  If it escapes horizontally, apply
Proposition~\ref{prop:pressure-spacetime-strict-bv} with the intermediate
slab $L''$: its total boundary mass tends to $2L''P_k$, while local
varifold convergence captures arbitrarily close to that full cylinder
mass inside a sufficiently large horizontal ball.  Thus the mass in
$(\mathbb R^n\setminus B_R)\times[-L',L']$ is uniformly small for large
$R$, contradicting the preceding lower bound.  This proves one Hausdorff
inclusion.

Conversely, if a point of $\partial B_k\times[-L',L']$ were missed by the
supports in a fixed ball, that ball would have zero $\mu_j$-mass for all
large $j$, whereas the limit cylinder has positive mass there.  The reverse
inclusion follows.  Since the initial sequence was arbitrary, all
convergences hold along the full family.
\end{proof}

We shall use the following boundary-null transfer from fixed to moving
Allard balls.

\begin{lemma}\label{lem:allard-generic-radius}
Let $V_j\to V$ as varifolds in an open subset of $\mathbb R^{n+1}$, let
$x_j\to p$, and let $P$ be an
$n$-plane.  If a neighborhood of $\overline B_{r+\delta}(p)$ is contained
in the open set and $\|V\|(\partial B_r(p))=0$, then
\begin{align}
\|V_j\|(B_r(x_j))&\longrightarrow\|V\|(B_r(p)),
\label{eq:generic-radius-mass-transfer}\\
\int_{B_r(x_j)\times G(n+1,n)}|S-P|^2\,dV_j(x,S)
&\longrightarrow
\int_{B_r(p)\times G(n+1,n)}|S-P|^2\,dV(x,S).
\label{eq:generic-radius-tilt-transfer}
\end{align}
\end{lemma}

\begin{proof}
Put \(\Delta_j=|x_j-p|\).  For every fixed \(0<\varepsilon<\delta\) and all
large \(j\),
\[
B_{r-\varepsilon}(p)\subset B_r(x_j)\subset B_{r+\varepsilon}(p),
\qquad \Delta_j<\varepsilon.
\]
The Portmanteau theorem for weakly convergent Radon measures \cite[Chapter~\refnum{1}]{AFP2000} therefore gives
\[
\begin{aligned}
\|V\|(B_{r-\varepsilon}(p))
&\leq\liminf_{j\to\infty}\|V_j\|(B_r(x_j))\\
&\leq\limsup_{j\to\infty}\|V_j\|(B_r(x_j))
\leq\|V\|(\overline B_{r+\varepsilon}(p)).
\end{aligned}
\]
Letting \(\varepsilon\downarrow0\) and using
\(\|V\|(\partial B_r(p))=0\) proves
\eqref{eq:generic-radius-mass-transfer}.

For the tilt term define Radon measures on the base by
\[
\lambda_j(A)=\int_{A\times G(n+1,n)}|S-P|^2\,dV_j(x,S),
\qquad
\lambda(A)=\int_{A\times G(n+1,n)}|S-P|^2\,dV(x,S).
\]
Since the Grassmannian integrand is bounded and continuous,
\(\lambda_j\rightharpoonup\lambda\), and
\[
\lambda(\partial B_r(p))
\leq\sup_{S\in G(n+1,n)}|S-P|^2\,\|V\|(\partial B_r(p))=0.
\]
The same moving-ball squeeze for \(\lambda_j\) proves
\eqref{eq:generic-radius-tilt-transfer}.
\end{proof}

\begin{proposition}\label{prop:allard-one-sheet}
For every \(0<L'<L\), there exists \(t_0=t_0(L',L)\) such that, for
every \(t\ge t_0\), there is a smooth function \(f_t\) satisfying
\[
\operatorname{spt}\widehat\Sigma_t^L
\cap(\mathbb R^n\times[-L',L'])
=\{((k+f_t(\omega,z))\omega,z):
(\omega,z)\in S^{n-1}\times[-L',L']\},
\]
and \(f_t\to0\) in \(C^m\) for every \(m\) as \(t\to\infty\).
Consequently, after increasing the threshold if necessary, for every such
$t$ the localized level
\[
 \Gamma_t:=\{x\in B_{r_0}\setminus\{0\}:u(x)=t\}
\]
is one smooth strictly convex hypersurface surrounding the origin; no other point of the level
$u=t$ lies in $B_{r_0}\setminus\{0\}$.  The sufficiently high localized levels are smoothly nested.
\end{proposition}

\begin{proof}
Fix an arbitrary sequence $t_j\to\infty$.

\noindent\textbf{Step 1: local smooth graphs.}
Let $p$ lie on the limiting cylinder
$\partial B_k\times(-L,L)$ and put
$P=T_p(\partial B_k\times\mathbb R)$.  For sufficiently small $\rho$,
smoothness of the cylinder gives
\[
 \frac{\|\mathcal V_\infty\|(B_{2\rho}(p))}{\omega_n(2\rho)^n}
 =1+O(\rho^2),\qquad
 (2\rho)^{-n}\int_{B_{2\rho}(p)}|S-P|^2\,d\mathcal V_\infty=O(\rho^2).
\]
Choose $\rho$ from the boundary-null radii and small enough for Allard's
interior theorem.  Proposition~\ref{prop:slab-current-varifold} gives
$x_j\in\operatorname{spt}\mathcal V_j$ with $x_j\to p$, and
Lemma~\ref{lem:allard-generic-radius} transfers the mass and tilt estimates
to $B_{2\rho}(x_j)$.  Step~2 of Proposition~\ref{prop:slab-current-varifold} supplies an
absolutely continuous first variation.  Fix any exponent $p_{\mathrm A}>n$.  On the
fixed pressure slab, \eqref{eq:rescaled-mean-curvature} and the transferred
mass-ratio bound give
\[
 (2\rho)^{1-n/p_{\mathrm A}}\,
 \|H_j\|_{L^{p_{\mathrm A}}(B_{2\rho}(x_j),\mu_j)}\le C_L\rho.
\]
Choose first $\rho$ so small that this quantity and the limiting-cylinder
mass and tilt excesses are below the Allard threshold, and then take $j$
large.  Since $\mathcal V_j$ is integral, its multiplicity is at least one
$\mu_j$-a.e.\  Thus the hypotheses of Allard's interior regularity theorem
\cite[Chapter~\refnum{5}, Theorem~\refnum{5.2}]{SimonGMT2018} are satisfied, and on a smaller ball the
support is a single $C^{1,\gamma}$ graph, $\gamma>0$, converging to the
cylinder graph.

The graphical prescribed-mean-curvature equation is uniformly elliptic on
these small-slope graphs.  Interior Schauder estimates give uniform
$C^{m,\alpha}$ bounds on still smaller balls for every $m$.  Every
subsequence therefore has a $C^m$-convergent subsubsequence, and its limit
must be the already identified $C^1$ cylinder graph.  Hence the whole
sequence converges to the cylinder in $C^m$.

\noindent\textbf{Step 2: the local graphs cover the full support.}
Choose intermediate slabs
$L'<L_0<L_1<L$.  Compactness of
$\partial B_k\times[-L_1,L_1]$ provides finitely many of the preceding
Allard charts.  By the Hausdorff support convergence in
Proposition~\ref{prop:slab-current-varifold}, after shrinking the charts
slightly they cover the entire support in $|z|\le L_0$, not merely a
selected component.  Consequently the $C^m$ estimates above are uniform
there.

\noindent\textbf{Step 3: global one-sheetedness.}
Fix $0<\tau<k/2$ and consider the radial projection
\[
 \pi(y,z)=\left(k\frac{y}{|y|},z\right)
\]
from the tubular neighborhood $\{||y|-k|<\tau\}$ onto the cylinder.  For
large $j$, full-support convergence places the support over the smaller
slab inside this tube, and the finite smooth graph cover makes the
restriction of $\pi$ a local diffeomorphism.  Over a compact
subcylinder it is proper; its image is open by local diffeomorphism and
closed by properness.  Since the target cylinder is connected and the
image is nonempty, $\pi$ is onto and hence a finite covering.  Write its
constant degree as $\deg(\pi_j)\ge1$.

Let $C'=\partial B_k\times(-L',L')$.  Since the local graphs converge in
$C^1$ to the cylinder, the tangential Jacobian of $\pi$ on
$\pi^{-1}(C')$ is at most $1+o(1)$.  The area formula and the exact mass
convergence from Proposition~\ref{prop:slab-current-varifold} give
\[
 \deg(\pi_j)\,\mathcal H^n(C')
 \le (1+o(1))\mathcal H^n(\widehat\Sigma_{t_j}^{L'})
 =(1+o(1))\mathcal H^n(C').
\]
Thus $\deg(\pi_j)=1$ for all large $j$.  The entire support over the smaller slab
is therefore one normal graph; oriented-current convergence fixes the
outward horizontal orientation.

Since the original sequence was arbitrary, the one-sheeted graph exists
for all sufficiently large $t$ and its height tends to zero in every
$C^m$.  Writing the normal graph in cylinder coordinates gives
\[
 \operatorname{spt}\widehat\Sigma_t^L\cap
 (\mathbb R^n\times[-L',L'])
 =\{((k+f_t(\omega,z))\omega,z):
   (\omega,z)\in S^{n-1}\times[-L',L']\},
 \qquad \|f_t\|_{C^m}\to0.
\]
The slice $z=z_0$ is exactly the rescaled level
$u=t+z_0/t$.  Hence, uniformly for $|z_0|\le L'$, these slices converge in
$C^2$ to the sphere of radius $k$; their outward shape operators converge
to $k^{-1}\Id$, and the vertical normal component tends to zero.  Thus the
slices are strictly convex and $|Du|>0$ for all large $t$.  At $z=0$,
\[
 t\Gamma_t=\{(k+f_t(\omega,0))\omega:\omega\in S^{n-1}\},
\]
which is a single smooth strictly convex hypersurface surrounding the
origin.  Since the pressure graph contains every point in the localized
height band, there is no additional component of the level $u=t$ in
$B_{r_0}\setminus\{0\}$.  The high levels are smoothly nested; coverage of
a full punctured neighborhood is established below.
\end{proof}

We now prove uniform blow-up at the puncture and identify every sufficiently
high localized superlevel.  Let
$\Gamma_t\subset B_{r_0}\setminus\{0\}$ be the single localized high level supplied by
Proposition~\ref{prop:allard-one-sheet}, and let $K_t\Subset B_{r_0}$ be the convex body bounded by it.
At \(z=0\), Proposition~\ref{prop:allard-one-sheet} gives
\[
t\Gamma_t
=\{(k+f_t(\omega,0))\omega:\omega\in S^{n-1}\},
\qquad \|f_t(\cdot,0)\|_{L^\infty}\longrightarrow0.
\]
Since \(K_t\) is the convex body bounded by this radial graph, for all
large \(t\),
\[
B_{k/2}\subset tK_t\subset B_{3k/2},
\qquad
B_{k/(2t)}\subset K_t\subset B_{3k/(2t)}.
\]
In particular,
\begin{equation}\label{eq:inner-ball-level}
B_{k/(2t)}\subset K_t
\end{equation}
for all large $t$.  The set
$\Omega_t^{\rm in}=\operatorname{int}K_t\setminus\{0\}$ is connected for $n\geq2$ and contains
no other point of the level $u=t$.  On \(\Gamma_t=\partial K_t\), the
outward orientation gives
\[
\frac{Du}{|Du|}=-\nu_{K_t},
\qquad
\partial_{\nu_{K_t}}u=-|Du|<0.
\]
Thus, for \(y\in\Gamma_t\) and small \(s>0\),
\[
u(y-s\nu_{K_t}(y))
=u(y)-s\,\partial_{\nu_{K_t}}u(y)+o(s)>t,
\]
so an inner collar lies in \(E_t\).  Moreover,
\[
\Omega_t^{\rm in}
=\bigl(\Omega_t^{\rm in}\cap\{u>t\}\bigr)
\mathbin{\dot\cup}
\bigl(\Omega_t^{\rm in}\cap\{u<t\}\bigr).
\]
The two sets are relatively open and the first is nonempty; connectedness
therefore gives
\begin{equation}\label{eq:inside-superlevel}
E_t\cap\Omega_t^{\rm in}=\Omega_t^{\rm in}.
\end{equation}
The upper radial inclusion already gives \(K_t\Subset B_{r_0}\) for large
\(t\).  Also
\[
E_t\subset B_{\rho(t)},\qquad \rho(t)\downarrow0,
\]
so
\[
E_t\cap(B_{r_0}\setminus\overline B_{r_0/2})=\varnothing
\]
for large \(t\).  The same relatively open/closed decomposition on the
connected exterior is
\[
B_{r_0}\setminus K_t
=\bigl((B_{r_0}\setminus K_t)\cap\{u>t\}\bigr)
\mathbin{\dot\cup}
\bigl((B_{r_0}\setminus K_t)\cap\{u<t\}\bigr).
\]
Its second part contains the outer annulus, so the first part is empty:
\[
E_t\cap(B_{r_0}\setminus K_t)=\varnothing.
\]
Together with \eqref{eq:inside-superlevel}, this gives
\begin{equation}\label{eq:exact-high-superlevel}
E_t=\operatorname{int}K_t\setminus\{0\}.
\end{equation}
Given \(T>0\), choose \(t>\max\{T,t_0\}\) and put
\(\delta_T=k/(2t)\).  Then
\[
0<|x|<\delta_T
\Longrightarrow x\in\operatorname{int}K_t\setminus\{0\}=E_t
\Longrightarrow u(x)>t>T.
\]
Equivalently,
\[
\forall T>0\ \exists\delta_T>0:\qquad
0<|x|<\delta_T\Longrightarrow u(x)>T,
\]
which is
\begin{equation}\label{eq:uniform-positive-blowup}
 u(x)\longrightarrow+\infty\qquad\hbox{uniformly as }x\to0.
\end{equation}
Fix a sufficiently high \(T_0\).  Nesting and
\eqref{eq:exact-high-superlevel} give the exact high-value region
\[
 \operatorname{int}K_{T_0}\setminus\{0\}
 =E_{T_0}=\bigcup_{t>T_0}\Gamma_t,
 \qquad
 \Gamma_{t_1}\cap\Gamma_{t_2}=\varnothing\quad(t_1\ne t_2).
\]
This region contains the punctured ball
\(B_{k/(2T_0)}\setminus\{0\}\) by \eqref{eq:inner-ball-level}.
Every point of it lies on the unique leaf \(\Gamma_{u(x)}\); together
with \(|Du|>0\), the implicit-function theorem makes the high levels a
smooth foliation covering a full punctured neighborhood.

The preceding construction used the fixed working radius $r_0$.  We record the quantifier needed
in the theorem.  Fix $0<r<R$ and choose
$0<\delta<\min\{r,r_0\}$.  Choose \(t_\delta\) so that
\(K_t\Subset B_\delta\) for \(t\geq t_\delta\), and put
\[
M_{r,\delta}:=\max_{\overline B_r\setminus B_\delta}u<\infty,
\qquad
T_r:=\max\{t_\delta,M_{r,\delta}+1\}.
\]
Then, for \(t\geq T_r\),
\[
\{x\in B_r\setminus\{0\}:u(x)>t\}
 \cap(\overline B_r\setminus B_\delta)=\varnothing,
\]
and \eqref{eq:exact-high-superlevel} proves
\begin{equation}\label{eq:localized-positive-level-derived}
 \{x\in B_r\setminus\{0\}:u(x)>t\}
 =\operatorname{int}K_t\setminus\{0\},\qquad t\geq T_r.
\end{equation}
Set \(K_{t,r}:=K_t\).  Then \eqref{eq:localized-positive-level-derived}
is \eqref{eq:localized-positive-level-main}, and
\(\partial K_{t,r}=\Gamma_t=\{u=t\}\cap B_r\).  The threshold may depend
on $r$; no uniform control as $r\uparrow R$ is claimed.
Because deleting the interior point $0$ does not change a
finite-perimeter set up to null sets, its reduced boundary remains
$\partial K_{t,r}$.  In the relative topology of
$B_r\setminus\{0\}$, the deleted point is not present, so the relative
topological boundary is also exactly $\partial K_{t,r}$.
The negative statement follows by applying the same proof to $-u$.
 \section{Gauss coordinates and the asymptotic expansion}\label{sec:round-complete}

Throughout this section \(k=n-1\).  By the preceding pressure section, there are
\(\delta>0\) and \(t_0\) such that
\[
B_\delta\setminus\{0\}\subset E_{t_0}
=\mathop{\dot\bigcup}_{t>t_0}\Gamma_t,\qquad
\Gamma_t=\partial^* E^+_{t,r_0}
=\{x\in B_{r_0}\setminus\{0\}:u(x)=t\},\qquad |Du|>0,
\]
and each \(\Gamma_t\) is a smooth strictly convex \(k\)-sphere surrounding
the origin.  For every fixed \(L>0\), apply
Proposition~\ref{prop:allard-one-sheet} with an outer slab, say \(2L\),
and inner slab \(L\).  Since the pressure graph itself is independent of
the clipping parameter on \( |z|<L\), this yields the normal graph
representation
\[
\operatorname{spt}\widehat\Sigma_T^L
=\{((k+f_T(\omega,z))\omega,z):
(\omega,z)\in\mathbb S^k\times(-L,L)\},
\qquad f_T\longrightarrow0\quad\hbox{in }C^\infty_{\rm loc}
\quad\text{as }T\to\infty.
\]
No level component outside \(B_{r_0}\) is used or controlled here.

The argument proceeds in four stages.  We first derive an exact support-function
system in Gauss coordinates and use translation balance to control its degree-one
modes.  We then prove a uniform interior estimate for the higher spherical
harmonics and use the limiting spectral gap to obtain exponential angular decay.
The remaining mean modes satisfy a two-dimensional perturbed radial system; a
coercive Hamiltonian barrier yields its algebraic asymptotics.  Finally we return
to Euclidean polar coordinates.

\subsection{Gauss coordinates and the exact support system}

Let $\omega\in\mathbb S^k$ be the outer unit normal of $\Gamma_t$, and let
$X(t,\omega)$ be the inverse Gauss map.  Its support function is
\[
h(t,\omega)=X(t,\omega)\cdot\omega.
\]
All covariant derivatives and contractions below are taken with respect to
the unit round metric $g$ on $\mathbb S^k$.
Let \(P_0,P_1,P_{\geq2}\) denote the orthogonal spherical-harmonic
projections of degrees \(0,1,\geq2\), with
\[
 P_1f=b_f\cdot\omega,\qquad
 b_f=\frac{k+1}{|\mathbb S^k|}
       \int_{\mathbb S^k}f(\omega)\omega\,d\omega.
\]
We write
\(C^m_{\geq2}(\mathbb S^k)
 :=\{f\in C^m(\mathbb S^k):P_0f=P_1f=0\}\).

The support parametrization and its normal velocity are recorded first.

\begin{lemma}
\label{lem:round-support-geometry}
One has
\begin{equation}\label{eq:round-support-map}
X=h\omega+\nabla h,\qquad
B_h:=\nabla^2h+hg>0,\qquad
dA_{\Gamma_t}=\det(B_h)\,d\omega.
\end{equation}
At fixed Gauss normal,
\begin{equation}\label{eq:round-normal-velocity}
h_t=-\frac1{|Du|}<0.
\end{equation}
\end{lemma}

\begin{proof}
The standard support-function identities for a strictly convex hypersurface
give
\[
 X=h\omega+\nabla h,\qquad
 X_a=(h_{;ab}+h\delta_{ab})e_b,
\]
so $B_h=\nabla^2h+hg>0$ and
$dA_{\Gamma_t}=\det(B_h)\,d\omega$.  At fixed Gauss normal,
$u(X(t,\omega))=t$, $Du=-|Du|\omega$, and
$X_t=h_t\omega+\nabla h_t$.  Hence
\[
 1=Du\cdot X_t=-|Du|h_t,
\]
which proves \eqref{eq:round-normal-velocity}.
\end{proof}

Set
\begin{equation}\label{eq:round-variables}
Z=\frac{t^2}{2},\qquad q(Z,\omega)=t h(t,\omega),\qquad
-h_t=\tan\phi,
\qquad B_q=\nabla^2q+qg.
\end{equation}
Angular differentiation keeps \(t\) fixed, so
\[
B_q=\nabla^2(th)+thg=tB_h>0,\qquad
\phi=-\arctan h_t\in(0,\pi/2),
\]
where the indicated branch of \(\arctan\) is used.

We record the coordinate bridge because the estimates below use mixed
$Z$--angular derivatives.  Fix a large base height $T$, put
$Z_T=T^2/2$, and use the pressure coordinate $z=T(t-T)$.  Then
\begin{equation}\label{eq:round-pressure-Z-bridge}
\begin{aligned}
Z(t)-Z_T&=\frac{t^2-T^2}{2}
=z+\frac{z^2}{2T^2},&
\frac tT&=1+\frac z{T^2},\\
\frac{d(Z-Z_T)}{dz}&=\frac tT,&
\partial_Z&=\frac Tt\,\partial_z.
\end{aligned}
\end{equation}
The one-sheeted pressure surface supplied by
Proposition~\ref{prop:allard-one-sheet} is first written in radial cylinder
parameters as
\[
Y_T(\theta,z)=r_T(\theta,z)\theta,\qquad
r_T(\theta,z)=k+f_T(\theta,z),
\qquad f_T\longrightarrow0\quad\text{in }C^\infty
\]
on each fixed $z$-slab.  Here $\theta$ is the normal of the limiting
cylinder, not yet the Gauss normal of the actual slice.

For fixed $z$ let
$G_T(\theta,z)$ be the outward Gauss map of $Y_T(\cdot,z)$.  For large
$T$,
\[
G_T(\theta,z)
=\frac{r_T(\theta,z)\theta-\nabla_\theta r_T(\theta,z)}
{\sqrt{r_T(\theta,z)^2+|\nabla_\theta r_T(\theta,z)|^2}}.
\]
Therefore the parameter-dependent inverse-function theorem gives
\[
G_T\longrightarrow\operatorname{id},\qquad
G_T^{-1}\longrightarrow\operatorname{id}\quad\hbox{in }C^\infty
\]
with all mixed $z$ and angular derivatives on bounded slabs.
Reparametrizing by $\theta=G_T^{-1}(\omega,z)$, the support function in
the rescaled $y$ variables is
\[
\widehat h_T(z,\omega)
=Y_T(G_T^{-1}(\omega,z),z)\cdot\omega
=T h(T+z/T,\omega).
\]
The complete scaling identities at fixed Gauss normal are
\[
\begin{aligned}
\partial_z\widehat h_T(z,\omega)
 &=h_t(T+z/T,\omega),\\
q(Z(t),\omega)&=\frac tT\widehat h_T(z,\omega),\\
B_q(Z(t),\omega)
 &=\frac tT\bigl(\nabla_\omega^2\widehat h_T
      +\widehat h_Tg\bigr),\\
\phi(Z(t),\omega)
 &=-\arctan\bigl(\partial_z\widehat h_T(z,\omega)\bigr).
\end{aligned}
\]
Since \(\widehat h_T\to k\) with all mixed derivatives and
\(\partial_Z=(T/t)\partial_z\), translating \(Z_T\) to zero gives
\begin{equation}\label{eq:round-cylinder-limit}
B_q\longrightarrow kg,\qquad \phi\longrightarrow0
\quad\hbox{in }C^\infty_{\rm loc}
\end{equation}
in the mixed \(Z\)--angular sense.  Moreover, for
\(\ell_a(\omega)=a\cdot\omega\),
\[
\nabla^2\ell_a+\ell_ag=0,\qquad B_{q+\ell_a}=B_q.
\]
Thus \(B_q\to kg\) alone does not control \(P_1q\).  Here the centered
pressure graph already gives \(q(Z_T,\cdot)\to k\), hence
\(P_1q(Z_T,\cdot)\to0\); the translation balance below supplies the
intrinsic quantitative centering estimate used later.

The capillary equation becomes the following exact first-order support
system.

\begin{proposition}\label{prop:round-exact-system}
The variables in \eqref{eq:round-variables} satisfy
\begin{align}
q_Z&=\frac{q}{2Z}-\tan\phi,\label{eq:round-q}\\
\phi_Z&=\sec\phi-\operatorname{tr}(B_q^{-1})
-\sec^2\phi\,
  \langle B_q^{-1}\nabla\phi,\nabla\phi\rangle.
\label{eq:round-phi}
\end{align}
Moreover, horizontal translation balance is exactly
\begin{equation}\label{eq:round-flux}
\int_{\mathbb S^k}\det(B_q)\sin\phi\,\omega\,d\omega=0.
\end{equation}
\end{proposition}

\begin{proof}
Since $dZ/dt=t$ and $q=th$,
\[
 q_Z=\frac1t(th)_t=\frac{q}{2Z}-\tan\phi,
\]
which is \eqref{eq:round-q}.

For the second equation put
$v=h_t$, $Q=(1+v^2)^{1/2}$, $B=B_h$, and
$F(t,\omega)=(X(t,\omega),t)$.  In the splitting $(t,\omega)$,
\[
 F_t=v\omega+\nabla v+e_{n+1},\qquad
 F_a=B_{ab}e_b,\qquad
 N=\frac{\omega-ve_{n+1}}{Q}.
\]
The angular block of the induced metric is $B^2$, and its Schur
complement is $Q^2$.  Contracting the second fundamental form with the
resulting inverse metric gives
\begin{equation}\label{eq:round-mean-curvature}
 H_\Sigma
 =\frac1Q\operatorname{tr}(B^{-1})
 +\frac1{Q^3}
   \bigl(\langle B^{-1}\nabla v,\nabla v\rangle-v_t\bigr).
\end{equation}
For completeness, the only nonzero components needed in this contraction
are $A_{ab}=Q^{-1}B_{ab}$ and $A_{tt}=-Q^{-1}v_t$; the mixed components
vanish.  Thus the angular-gradient term in \eqref{eq:round-mean-curvature}
comes from the Schur-complement correction to the angular inverse metric.

The capillary equation gives $H_\Sigma=t$.  Using
\[
 B=t^{-1}B_q,\qquad v=-\tan\phi,\qquad Q=\sec\phi,
 \qquad \nabla v=-\sec^2\phi\nabla\phi,
 \qquad v_t=-t\sec^2\phi\,\phi_Z,
\]
substitution in \eqref{eq:round-mean-curvature} yields exactly
\eqref{eq:round-phi}.

Finally, the horizontal translation flux vanishes on every sufficiently
high level of the actual pole branch.  Lemma~\ref{lem:tail-translation} for
$n\ge3$, and Corollary~\ref{cor:planar-genuine-one-tail} for $n=2$, first
give vanishing on a full-measure set of sufficiently high regular levels.
By Proposition~\ref{prop:allard-one-sheet}, however, every sufficiently high
level is smooth, $|Du|>0$, and the high levels form a smooth foliation.  For
any two such levels $t<T$, uniform blow-up makes the finite-height slab
$\{t<u<T\}$ avoid the puncture.  Integrating $\operatorname{div}\mathsf S=0$ on
that slab and using $\int_{\Gamma_s}\nu_s\,dA=0$ shows that the
translation flux is pointwise constant throughout the high tail.  Since that
constant is zero on a full-measure subset,
\[
 \int_{\Gamma_t}\frac{\nu}{\sqrt{1+|Du|^2}}\,dA=0
\]
for every sufficiently high $t$.  In Gauss coordinates
$\nu=\omega$,
$dA=t^{-k}\det(B_q)\,d\omega$, and
$(1+|Du|^2)^{-1/2}=\sin\phi$.  Removing the positive factor $t^{-k}$
gives \eqref{eq:round-flux}.
\end{proof}

\subsection{Translation balance and centering}

Define the mean radius and the Steiner point of the level body by
\begin{equation}\label{eq:round-center-def}
R_h(t)=\frac1{|\mathbb S^k|}\int_{\mathbb S^k}h(t,\omega)\,d\omega,
\qquad
c(t)=\frac{k+1}{|\mathbb S^k|}
      \int_{\mathbb S^k}h(t,\omega)\omega\,d\omega.
\end{equation}

The following intrinsic calculation supplies the quantitative centering
estimate used below.

\begin{lemma}\label{lem:round-centering}
As $t\to\infty$,
\begin{equation}\label{eq:round-centering-conclusion}
tR_h(t)\longrightarrow k,\qquad c(t)=o(R_h(t)).
\end{equation}
Consequently
\begin{equation}\label{eq:round-full-limit}
q(Z,\cdot)\longrightarrow k,\qquad \phi(Z,\cdot)\longrightarrow0
\quad\hbox{in }C^\infty(\mathbb S^k).
\end{equation}
\end{lemma}

\begin{proof}
Since \(B_q=tB_h\),
\[
\begin{aligned}
\operatorname{tr}_gB_q&=\Delta q+kq,\\
\frac1{|\mathbb S^k|}\int_{\mathbb S^k}
\operatorname{tr}_gB_q\,d\omega
&=\frac{k}{|\mathbb S^k|}\int_{\mathbb S^k}q\,d\omega
=ktR_h(t).
\end{aligned}
\]
The left side tends to \(k^2\) by \eqref{eq:round-cylinder-limit}; hence
\[
tR_h(t)\longrightarrow k,\qquad
\frac{B_h}{R_h}=\frac{B_q}{tR_h}\longrightarrow g
\quad\hbox{in }C^\infty.
\]
Pressure-slab convergence also gives \(h_t\to0\) in \(C^\infty\).

Using $\sin\phi=-h_t/(1+h_t^2)^{1/2}$, the flux identity becomes
\begin{equation}\label{eq:round-support-flux}
\int_{\mathbb S^k}
\frac{h_t\det(B_h)}{(1+h_t^2)^{1/2}}\,\omega\,d\omega=0.
\end{equation}
Put
\[
\varepsilon_{\rm cen}(t)=
\left\|
\frac{\det(B_h)}{R_h(t)^k(1+h_t^2)^{1/2}}-1
\right\|_{L^\infty(\mathbb S^k)}.
\]
Then $\varepsilon_{\rm cen}(t)\to0$.  Subtracting \eqref{eq:round-support-flux} from the
derivative of \eqref{eq:round-center-def} is made explicit by first
dividing the zero flux by \(R_h(t)^k\):
\[
0=\int_{\mathbb S^k}h_t
\frac{\det(B_h)}
 {R_h(t)^k(1+h_t^2)^{1/2}}\,\omega\,d\omega.
\]
Therefore, using
\(R_h'(t)=|\mathbb S^k|^{-1}\int h_t\,d\omega\) and \(h_t<0\),
\[
\begin{aligned}
c'(t)
&=\frac{k+1}{|\mathbb S^k|}\int_{\mathbb S^k}h_t\omega\,d\omega\\
&=\frac{k+1}{|\mathbb S^k|}\int_{\mathbb S^k}h_t
\left(1-\frac{\det(B_h)}
 {R_h(t)^k(1+h_t^2)^{1/2}}\right)\omega\,d\omega,\\
|c'(t)|
&\leq\frac{k+1}{|\mathbb S^k|}\varepsilon_{\rm cen}(t)
\int_{\mathbb S^k}|h_t|\,d\omega
=(k+1)\varepsilon_{\rm cen}(t)(-R_h'(t)).
\end{aligned}
\]
The nested convex bodies shrink to the puncture, hence $c(t)\to0$ and
$R_h(t)\to0$.  With $\varepsilon_{\rm cen,*}(t)=\sup_{s\geq t}\varepsilon_{\rm cen}(s)$, integration from
$t$ to infinity yields
\[
\begin{aligned}
|c(t)|&\leq\int_t^\infty|c'(s)|\,ds\\
&\leq(k+1)\varepsilon_{\rm cen,*}(t)
\int_t^\infty(-R_h'(s))\,ds\\
&=(k+1)\varepsilon_{\rm cen,*}(t)R_h(t)=o(R_h(t)).
\end{aligned}
\]

Thus
\[
P_0q=tR_h(t)\longrightarrow k,\qquad
P_1q=t\,c(t)\cdot\omega=o(tR_h(t))\longrightarrow0.
\]
Put \(\eta=P_{\geq2}q\) and
\(\mathcal L\eta=\nabla^2\eta+\eta g\).  Since
\(\mathcal L(P_1q)=0\),
\[
\mathcal L\eta=B_q-(P_0q)g\longrightarrow0.
\]
Taking the trace gives
\[
 (\Delta+k)\eta=\operatorname{tr}_g(\mathcal L\eta).
\]
Since $P_0\eta=P_1\eta=0$, the scalar operator $\Delta+k$ is invertible on
the degree-$\ge2$ subspace.  The standard Schauder estimate on the fixed
sphere therefore gives
\[
 \|\eta\|_{C^{m+2,\alpha}(\mathbb S^k)}
 \le C_{m,\alpha}\|\operatorname{tr}_g(\mathcal L\eta)\|_{C^{m,\alpha}}
 \le C_{m,\alpha}\|\mathcal L\eta\|_{C^{m,\alpha}}.
\]
Hence \(\eta\to0\) with all derivatives.  This proves the first limit in
\eqref{eq:round-full-limit}; the second is already part of
\eqref{eq:round-cylinder-limit}.
\end{proof}

\subsection{Elimination of the flux mode}

Decompose
\begin{equation}\label{eq:round-decomposition}
q=\bar q+a\cdot\omega+\eta,\qquad
\phi=\bar\phi+b\cdot\omega+\psi,\qquad
P_{\geq2}\eta=\eta,\quad P_{\geq2}\psi=\psi.
\end{equation}

On the open set \(\bar q g+\nabla^2\eta+\eta g>0\), define
\[
\begin{aligned}
G:\;&\mathbb R^{k+1}\times\mathbb R^2
\times C^2_{\geq2}(\mathbb S^k)\times C^0_{\geq2}(\mathbb S^k)
\longrightarrow\mathbb R^{k+1},\\
G(b;\bar q,\bar\phi,\eta,\psi)
&:=\int_{\mathbb S^k}
\det(\bar q g+\nabla^2\eta+\eta g)
\sin(\bar\phi+b\cdot\omega+\psi)\,\omega\,d\omega.
\end{aligned}
\]
The flux equation eliminates the degree-one component of the angle.

\begin{lemma}\label{lem:round-flux-ift}
There are fixed neighborhoods \(\mathcal U\) of
\((k,0,0,0)\) in
\[
\mathbb R^2\times C^2_{\geq2}(\mathbb S^k)
\times C^0_{\geq2}(\mathbb S^k)
\]
and \(\mathcal V\) of \(0\) in \(\mathbb R^{k+1}\), and a unique smooth
map \(\beta:\mathcal U\to\mathcal V\), such that
\[
G(b;\bar q,\bar\phi,\eta,\psi)=0
\quad\Longleftrightarrow\quad
b=\beta(\bar q,\bar\phi,\eta,\psi).
\]
For every radial pair \((\bar q,\bar\phi)\) near \((k,0)\),
\begin{equation}\label{eq:round-beta-vanish}
\beta(\bar q,\bar\phi,0,0)=0,\qquad
D_{(\eta,\psi)}\beta(\bar q,\bar\phi,0,0)=0.
\end{equation}
In particular, uniformly near the limiting radial state,
\begin{equation}\label{eq:round-beta-quadratic}
|b|\leq C\bigl(\|\eta\|_{C^2}+\|\psi\|_{C^0}\bigr)^2.
\end{equation}
The degree-one coefficient $a$ does not enter this equation.
\end{lemma}

\begin{proof}
The displayed map is smooth: the determinant is polynomial in the
\(C^0\) Hessian entries and the sine
Nemytskii map is smooth on \(C^0\).  Moreover
\[
\nabla^2(a\cdot\omega)+(a\cdot\omega)g=0,\qquad
B_q=\bar q g+\nabla^2\eta+\eta g,
\]
so \(a\) is absent from \(G\).  At a radial pair,
\[
G(0;\bar q,\bar\phi,0,0)
=\bar q^k\sin\bar\phi\int_{\mathbb S^k}\omega\,d\omega=0,
\]
and
\[
\begin{aligned}
D_bG
&=\bar q^k\cos\bar\phi\int_{\mathbb S^k}
\omega\otimes\omega\,d\omega
=\frac{\bar q^k\cos\bar\phi|\mathbb S^k|}{k+1}\Id,\\
D_\eta G[\dot\eta]
&=\bar q^{k-1}\sin\bar\phi
\int_{\mathbb S^k}(\Delta+k)\dot\eta\,\omega\,d\omega\\
&=\bar q^{k-1}\sin\bar\phi
\int_{\mathbb S^k}\dot\eta(\Delta+k)\omega\,d\omega=0,\\
D_\psi G[\dot\psi]
&=\bar q^k\cos\bar\phi\int_{\mathbb S^k}\dot\psi\,\omega\,d\omega=0.
\end{aligned}
\]
Thus \(D_bG\) is uniformly invertible near \((k,0)\), and the
Banach-space implicit-function theorem gives \(\beta\) with
\[
\beta(\bar q,\bar\phi,0,0)=0,\qquad
D_{(\eta,\psi)}\beta(\bar q,\bar\phi,0,0)
=-(D_bG)^{-1}D_{(\eta,\psi)}G=0.
\]
On a fixed smaller neighborhood, \(D^2\beta\) is uniformly bounded, so
Taylor's theorem gives
\[
|b|=|\beta(\bar q,\bar\phi,\eta,\psi)|
\leq C\bigl(\|\eta\|_{C^2}+\|\psi\|_{C^0}\bigr)^2.
\]
By \eqref{eq:round-full-limit}, the actual solution lies in
\(\mathcal U\times\mathcal V\) for all sufficiently large \(Z\).
\end{proof}

\subsection{The weighted interior estimate and the spectral gap}

For $L>0$ write
\[
\mathcal C_L(Z_0):=[Z_0-L,Z_0+L]\times\mathbb S^k.
\]

The next lemma is the uniform interior estimate for the high spherical
harmonics.

\begin{lemma}
\label{lem:round-DN-estimate}
For every \(L\geq2\) and integer \(j\geq1\), there are
$Z_{j,L}$ and $C_{j,L}$ such that, for $Z_0\geq Z_{j,L}$,
\begin{equation}\label{eq:round-DN-Cj}
\|\eta\|_{C^j(\mathcal C_{L-1}(Z_0))}
+\|\psi\|_{C^{j-1}(\mathcal C_{L-1}(Z_0))}
\leq C_{j,L}\| (\eta,\psi)\|_{L^2(\mathcal C_L(Z_0))}.
\end{equation}
\end{lemma}

\begin{proof}
 \noindent\textbf{Step 1: projected path linearization.}
Substitute \(b=\beta(\bar q,\bar\phi,\eta,\psi)\), put
\[
\phi=\bar\phi+\beta(\bar q,\bar\phi,\eta,\psi)\cdot\omega+\psi,
\qquad B_q=\bar q g+\nabla^2\eta+\eta g,
\]
and define
\[
\begin{aligned}
\mathfrak G_1(\bar q,\bar\phi,\eta,\psi)
&:=\eta_Z-\frac{\eta}{2Z}+P_{\geq2}(\tan\phi),\\
\mathfrak G_2(\bar q,\bar\phi,\eta,\psi)
&:=\psi_Z-P_{\geq2}\!\left[
\sec\phi-\operatorname{tr}(B_q^{-1})
-\sec^2\phi\,
\langle B_q^{-1}\nabla\phi,\nabla\phi\rangle
\right].
\end{aligned}
\]
Here the degree-one part of \(q\) disappears because
\(\nabla^2(a\cdot\omega)+(a\cdot\omega)g=0\).  The projected exact system
is \(\mathfrak G(\bar q,\bar\phi,\eta,\psi)=0\), while
\(\mathfrak G(\bar q,\bar\phi,0,0)=0\) for every radial pair.  Thus, with
\(\mathbf w=(\eta,\psi)\),
\[
\begin{aligned}
0&=\mathfrak G(\bar q,\bar\phi,\eta,\psi)
  -\mathfrak G(\bar q,\bar\phi,0,0)\\
 &=\int_0^1D_{(\eta,\psi)}
\mathfrak G(\bar q,\bar\phi,s\eta,s\psi)[\eta,\psi]\,ds
 =(\mathcal L_Z+\mathcal K_Z)\mathbf w.
\end{aligned}
\]
In particular,
\begin{equation}\label{eq:round-LK}
\mathcal L_Z\mathbf w+\mathcal K_Z\mathbf w=0.
\end{equation}
Here $\mathcal L_Z$ consists of the local differential terms and
$\mathcal K_Z$ of the finite-rank terms produced by $P_0,P_1$ and by the
derivative of the solved parameter $b$.

\noindent\textbf{Step 2: exact path linearization and the finite-rank correction.}
We use two distinct fundamental-theorem-of-calculus representations.  The
local differential part is obtained from the straight path in the
physical variables $(B,\phi)$ joining the radial endpoint to the actual
endpoint, whereas the implicit-function path in $(\eta,\psi)$ is used only
to represent the endpoint vector $b$.  These two identities are combined
only after both have been evaluated at the same endpoint.

Put \(\Pi:=P_0+P_1\) and \(\Pi_{\rm hi}:=I-\Pi\).  Shrinking the
implicit-function neighbourhood if necessary and increasing the tail
threshold, we may assume that it is star-shaped in the
\((\eta,\psi)\)-variables.  Put
\[
 b:=\beta(\bar q,\bar\phi,\eta,\psi),\qquad \zeta:=b\cdot\omega+\psi.
\]
For \(0\leq s\leq1\), use the straight path in the two physical
variables
\[
 B_s:=\bar q g+sC_\eta,\qquad C_\eta:=\nabla^2\eta+\eta g,
 \qquad \phi_s:=\bar\phi+s\zeta.
\]
We also write
\[
 M_s:=B_s^{-1},\qquad \mathbf h_s:=\nabla\phi_s,\qquad
 \mathcal Q_s:=\langle M_s\mathbf h_s,\mathbf h_s\rangle,
\qquad \dot\phi_s=\zeta=\psi+b\cdot\omega.
\]
All these objects are evaluated at the same \(Z\).  By the already
proved mixed \(Z\)--angular convergence to the round family, this
claim does not use the estimate currently being proved: indeed,
\[
 \bar q=\frac1{k|\mathbb S^k|}\int_{\mathbb S^k}
       \operatorname{tr}_gB_q\,d\omega,\qquad
 C_\eta=B_q-\bar q g,
\]
and the fixed-sphere elliptic inverse of
\(\eta\mapsto\nabla^2\eta+\eta g\) on the high-harmonic subspace gives
mixed convergence of \(\eta\) from that of \(B_q\).  The fixed
projections \(P_0,P_1,P_{\geq2}\) give the corresponding convergence
of \(\bar\phi,b,\psi\) from that of \(\phi\).  Hence, uniformly for
\(s\in[0,1]\),
\begin{equation}\label{eq:DN-path-convergence}
 B_s\longrightarrow kg,\qquad
 \phi_s\longrightarrow0,\qquad
 \nabla\phi_s\longrightarrow0
 \quad\text{in every mixed }C^m\text{ norm}.
\end{equation}

Let
\[
 \mathcal G(B,\phi):=
 \sec\phi-\operatorname{tr}(B^{-1})
 -\sec^2\phi\,
       \langle B^{-1}\nabla\phi,\nabla\phi\rangle.
\]
Since \(\dot M_s=-M_sC_\eta M_s\), direct differentiation gives
\begin{equation}\label{eq:DN-exact-linearization}
\begin{split}
 \frac d{ds}\mathcal G(B_s,\phi_s)
 ={}&
 A_s^{ab}\bigl(\nabla_{ab}\eta+\eta g_{ab}\bigr)
  +d_s^a\nabla_a\psi+c_s\psi\\
 &+
  c_s(b\cdot\omega)
  +d_s^a\nabla_a(b\cdot\omega),
\end{split}
\end{equation}
where
\begin{equation}\label{eq:DN-local-coefficients}
\begin{split}
 A_s^{ab}
 &:=(M_s^2)^{ab}
   +\sec^2\phi_s\,(M_s\mathbf h_s)^a(M_s\mathbf h_s)^b,\\
 d_s^a&:=-2\sec^2\phi_s\,(M_s\mathbf h_s)^a,\\
 c_s&:=\sec\phi_s\tan\phi_s
       -2\sec^2\phi_s\tan\phi_s\,\mathcal Q_s.
\end{split}
\end{equation}
Here \(M_s^2\) is the positive \(g\)-self-adjoint square of \(M_s\).
Similarly,
\begin{equation}\label{eq:DN-tan-linearization}
 \frac d{ds}\tan\phi_s
   =\sec^2\phi_s\,\psi
     +\sec^2\phi_s(b\cdot\omega).
\end{equation}
Define the path-averaged local coefficients
\begin{equation}\label{eq:DN-averaged-coefficients}
 e:=\int_0^1\sec^2\phi_s\,ds,\qquad
 A^{ab}:=\int_0^1A_s^{ab}\,ds,\qquad
 d^a:=\int_0^1d_s^a\,ds,\qquad
 c:=\int_0^1c_s\,ds.
\end{equation}
Thus the local differential part of the exact difference system is
\begin{equation}\label{eq:DN-local-operator}
\begin{split}
 (\mathcal L_Z\mathbf w)_1
   &:=\eta_Z-\frac{\eta}{2Z}+e\psi,\\
 (\mathcal L_Z\mathbf w)_2
   &:=\psi_Z
      -A^{ab}\nabla_{ab}\eta
      -(\operatorname{tr}_g A)\eta
      -d^a\nabla_a\psi-c\psi ,
\end{split}
\qquad \mathbf w=(\eta,\psi).
\end{equation}
Notice in particular that no derivative of \(\mathbf w\) has been suppressed
in \eqref{eq:DN-local-operator}.
At the round limit, \(\mathcal L_Z\mathbf w=0\) becomes
\begin{equation}\label{eq:round-linear-limit}
 \eta_Z+\psi=0,\qquad
 \psi_Z-k^{-2}(\Delta+k)\eta=0.
\end{equation}

We next express the solved vector \(b\) as a finite-dimensional linear
functional of \(\mathbf w\).  The straight physical path keeps this contribution
in a fixed finite-dimensional range, so the resulting nonlocal terms are
genuinely finite rank.  Let
\[
 \mathfrak F(b;\bar q,\bar\phi,\eta,\psi)
 :=\int_{\mathbb S^k}
       \det(\bar q g+C_\eta)\,
       \sin(\bar\phi+b\cdot\omega+\psi)\,\omega\,d\omega
\]
be the flux map used to define \(b=\beta(\bar q,\bar\phi,\eta,\psi)\).
Set
\[
 \widehat b_s:=\beta(\bar q,\bar\phi,s\eta,s\psi),\qquad
 \widehat\phi_s:=\bar\phi+\widehat b_s\cdot\omega+s\psi,
\]
so \(\widehat b_s\to0\) and \(\widehat\phi_s\to0\) in every required
mixed norm, uniformly in \(s\).  In the formulas below the subscript
\(s\) on a derivative of \(\mathfrak F\) means evaluation at
\((\widehat b_s;\bar q,\bar\phi,s\eta,s\psi)\).
For an arbitrary high-mode pair \(\mathbf v=(\xi,\upsilon)\), define
\[
 \mathfrak b_s[\mathbf v]
 :=D_{(\eta,\psi)}\beta(\bar q,\bar\phi,s\eta,s\psi)[\xi,\upsilon].
\]
The implicit differentiation formula is
\begin{equation}\label{eq:DN-beta-derivative}
 \mathfrak b_s[\mathbf v]
 =-\bigl(D_b\mathfrak F_s\bigr)^{-1}
   \left\{
   \int_{\mathbb S^k}
       \operatorname{cof}(B_s)^{ab}
       (\nabla_{ab}\xi+\xi g_{ab})
       \sin\widehat\phi_s\,\omega\,d\omega
   +\int_{\mathbb S^k}
       \det(B_s)\cos\widehat\phi_s\,\upsilon\,\omega\,d\omega
   \right\}.
\end{equation}
The matrix \(D_b\mathfrak F_s\) is uniformly invertible by the flux
implicit-function theorem and \eqref{eq:DN-path-convergence}.  Two
integrations by parts on \(\mathbb S^k\) rewrite
\eqref{eq:DN-beta-derivative} as
\begin{equation}\label{eq:DN-beta-order-zero}
 \mathfrak b_s[\mathbf v]
 =-\bigl(D_b\mathfrak F_s\bigr)^{-1}
   \left\{
      \int_{\mathbb S^k}\mathcal U_s\,\xi\,d\omega
      +\int_{\mathbb S^k}\mathcal V_s\,\upsilon\,d\omega
   \right\},
\end{equation}
where the vector-valued smooth kernels are
\begin{equation}\label{eq:DN-beta-kernels}
\begin{split}
 \mathcal U_s
 &:=
 \nabla_a\nabla_b
   \bigl(\operatorname{cof}(B_s)^{ab}
          \sin\widehat\phi_s\,\omega\bigr)
 +\operatorname{tr}_g(\operatorname{cof}(B_s))
          \sin\widehat\phi_s\,\omega,\\
 \mathcal V_s
 &:=\det(B_s)\cos\widehat\phi_s\,\omega.
\end{split}
\end{equation}
Consequently
\begin{equation}\label{eq:DN-B-functional}
 \mathfrak B[\mathbf v]:=\int_0^1\mathfrak b_s[\mathbf v]\,ds
\end{equation}
is an angular operator of order zero with range in
\(\mathbb R^{k+1}\), and the fundamental theorem of calculus gives
\begin{equation}\label{eq:DN-B-on-U}
 \mathfrak B[\mathbf w]
 =\beta(\bar q,\bar\phi,\eta,\psi)-\beta(\bar q,\bar\phi,0,0)=b.
\end{equation}

For \(\mathbf v=(\xi,\upsilon)\), set
\[
 \mathcal R[\mathbf v]
 :=A^{ab}(\nabla_{ab}\xi+\xi g_{ab})
       +d^a\nabla_a\upsilon+c\upsilon.
\]
Extend \(\mathcal L_Z\) from \eqref{eq:DN-local-operator} to an arbitrary
high-mode pair \(\mathbf v=(\xi,\upsilon)\) in the evident way, and define
\begin{equation}\label{eq:DN-exact-LK-splitting}
\begin{split}
 (\mathcal K_Z\mathbf v)_1
 &:=-\Pi(e\upsilon)
    +\Pi_{\rm hi}\!\left(e(\mathfrak B[\mathbf v]\cdot\omega)\right),\\
 (\mathcal K_Z\mathbf v)_2
 &:=\Pi\mathcal R[\mathbf v]
   -\Pi_{\rm hi}\!\left(
       c(\mathfrak B[\mathbf v]\cdot\omega)
       +d^a\nabla_a(\mathfrak B[\mathbf v]\cdot\omega)\right).
\end{split}
\end{equation}
Equations \eqref{eq:DN-exact-linearization},
\eqref{eq:DN-tan-linearization}, and
\eqref{eq:DN-B-on-U} show \emph{exactly} that
\begin{equation}\label{eq:DN-exact-P-equation}
 \mathcal P_Z\mathbf w=(\mathcal L_Z+\mathcal K_Z)\mathbf w=0.
\end{equation}
In particular, \(\mathcal K_Z\) is pointwise in \(Z\), contains no
\(Z\)-derivative of its argument \(\mathbf v\), and commutes with every scalar cutoff
\(\chi(Z)\).

We claim more precisely that \(\mathcal K_Z\) is small in the norms used
below.  At a radial state \(B=\bar q g\), \(\phi=\bar\phi\), the coefficients
of \(\mathcal R[\mathbf v]\) are independent of \(\omega\), and hence
\(\Pi\mathcal R[\mathbf v]=0\) for \(P_{\ge2}\xi=\xi,\ P_{\ge2}\upsilon=\upsilon\).  Moreover,
\[
 \int_{\mathbb S^k}(\Delta+k)\eta\,\omega\,d\omega=0,
 \qquad
 \int_{\mathbb S^k}\psi\,\omega\,d\omega=0,
\]
so \(D_{(\eta,\psi)}\beta(\bar q,\bar\phi,0,0)=0\) on
\(\operatorname{Ran}P_{\ge2}\times\operatorname{Ran}P_{\ge2}\).  For the variable-coefficient terms in
\(\Pi\mathcal R[\mathbf v]\), integrate the two derivatives of \(\xi\), and the one
derivative of \(\upsilon\), onto the fixed degree \(0\) or \(1\) harmonic
and the coefficient.  Together with
\eqref{eq:DN-beta-order-zero}, this shows that every component of
\(\mathcal K_Z\) is a finite sum of terms of the form
\begin{equation}\label{eq:DN-finite-rank-form}
 a_\nu(Z,\omega)
 \int_{\mathbb S^k}
    \bigl\{\alpha_\nu(Z,\omega')\xi(Z,\omega')
           +\beta_\nu(Z,\omega')\upsilon(Z,\omega')\bigr\}
 \,d\omega'.
\end{equation}
The mixed \(C^m\) seminorms of the relevant high-mode parts of
\(\alpha_\nu,\beta_\nu\), and of the nonradial parts producing
\(\Pi\mathcal R[\mathbf v]\), tend
to zero as \(Z\to\infty\).  Equivalently, since \(\xi,\upsilon\) are high
modes, one may replace the kernels in
\eqref{eq:DN-beta-order-zero} by their $P_{\ge2}$-projections, and
for \(0\leq r\leq m+1\) one has
\[
 \bigl|\partial_Z^r\mathfrak B[\mathbf v](Z)\bigr|
 \leq\varepsilon_{m,L}(Z_0)
 \sum_{\ell=0}^r\left(
  \|\partial_Z^\ell \xi(Z,\cdot)\|_{L^2(\mathbb S^k)}
 +\|\partial_Z^\ell \upsilon(Z,\cdot)\|_{L^2(\mathbb S^k)}
 \right)
\]
on \(\mathcal C_L(Z_0)\), with \(\varepsilon_{m,L}(Z_0)\to0\).
Therefore, for each fixed \(m\) and fixed cylinder length \(L\),
\begin{equation}\label{eq:DN-norms}
\begin{split}
 \mathcal X_m(\mathbf v;Q)&:=
 \|\xi\|_{H^{m+2}(Q)}
 +\|\upsilon\|_{H^{m+1}(Q)},\\
 \mathcal Y_m(\mathbf f;Q)&:=
 \|\mathbf f_1\|_{H^{m+1}(Q)}
 +\|\mathbf f_2\|_{H^m(Q)}
\end{split}
\end{equation}
satisfy
\begin{equation}\label{eq:round-K-small}
 \mathcal Y_m(\mathcal K_Z\mathbf v;\mathcal C_L(Z_0))
 \leq \varepsilon_{m,L}(Z_0)\,\mathcal X_m(\mathbf v;\mathcal C_L(Z_0)),
 \qquad
 \varepsilon_{m,L}(Z_0)\longrightarrow0.
\end{equation}
This follows directly from \eqref{eq:DN-finite-rank-form},
differentiating under the angular integrals; all coefficient
derivatives are uniformly controlled by
\eqref{eq:DN-path-convergence}.  Thus the finite-rank correction loses no derivatives.

\noindent\textbf{Step 3: Douglis--Nirenberg ellipticity and the interior estimate.}
We now check the local system to which the Douglis--Nirenberg estimate is applied.  First, after increasing
\(Z_*\), the eigenvalues of every \(B_s\) lie, for example, in
\([k/2,3k/2]\).  Thus, for every covector \(\theta\),
\begin{equation}\label{eq:DN-positive-A}
 A_s^{ab}\theta_a\theta_b
 =|M_s\theta|_g^2
  +\sec^2\phi_s
       \bigl\langle M_s\mathbf h_s,\theta\bigr\rangle_g^2
 \geq \frac{4}{9k^2}|\theta|_g^2.
\end{equation}
Also \(e\geq1\), and \(e\to1\),
\(A\to k^{-2}g^{-1}\), \(d\to0\), \(c\to0\) in every
mixed coefficient norm.  Assign column weights
\[
 (t_\eta,t_\psi)=(2,1)
 \quad\text{and row weights}\quad
 (s_1,s_2)=(-1,0).
\]
Then the orders of the four entries of \(\mathcal L_Z\) are at most
\[
 \begin{pmatrix}
 s_1+t_\eta&s_1+t_\psi\\
 s_2+t_\eta&s_2+t_\psi
 \end{pmatrix}
 =
 \begin{pmatrix}1&0\\2&1\end{pmatrix},
\]
and its weighted principal symbol at
\((Z,\omega;\tau,\theta)\) is
\begin{equation}\label{eq:DN-full-symbol}
 \sigma_{\rm DN}(\mathcal L_Z)(\tau,\theta)
 =
 \begin{pmatrix}
 i\tau&e\\
 A^{ab}\theta_a\theta_b&i(\tau-d^a\theta_a)
 \end{pmatrix}.
\end{equation}
The drift term \(-d^a\nabla_a\psi\) has weighted principal order one
and is therefore included in the lower-right entry.  Consequently
\begin{equation}\label{eq:DN-symbol-determinant}
\begin{split}
 \det\sigma_{\rm DN}(\mathcal L_Z)(\tau,\theta)
 &=-\tau^2+\tau d^a\theta_a-eA^{ab}\theta_a\theta_b\\
 &=-\left(\tau-\tfrac12d^a\theta_a\right)^2
   -\left(eA^{ab}-\tfrac14d^ad^b\right)\theta_a\theta_b.
\end{split}
\end{equation}
To verify uniform negativity, put \(v_s=M_s\mathbf h_s\).  The coefficient
formulas and weighted Cauchy--Schwarz give
\[
 \frac14(d\cdot\theta)^2
 =\left(\int_0^1\sec^2\phi_s\,v_s\cdot\theta\,ds\right)^2
 \le e\int_0^1\sec^2\phi_s\,(v_s\cdot\theta)^2\,ds.
\]
Hence
\[
 \left(eA-\tfrac14d\otimes d\right)[\theta,\theta]
 \ge e\int_0^1|M_s\theta|_g^2\,ds
 \ge \frac{4}{9k^2}|\theta|_g^2.
\]
Since \(d\) is uniformly bounded on the tail, completing the square
in \eqref{eq:DN-symbol-determinant} yields
\[
 \left|\det\sigma_{\rm DN}(\mathcal L_Z)(\tau,\theta)\right|
 \ge c_k\bigl(\tau^2+|\theta|_g^2\bigr)
 \qquad\text{for all }(\tau,\theta),
\]
after fixing the tail threshold.  This is precisely the
uniform Douglis--Nirenberg ellipticity condition
\cite{DouglisNirenberg1955}; in the round limit
\eqref{eq:DN-full-symbol} becomes
\[
 \begin{pmatrix}
 i\tau&1\\ k^{-2}|\theta|^2&i\tau
 \end{pmatrix}.
\]

Apply the compactly supported interior estimate for a uniformly
Douglis--Nirenberg elliptic system
\cite{DouglisNirenberg1955} with \(p=2\)
in a fixed finite atlas of \(\mathbb S^k\), subdivided into coordinate
balls of a uniformly small radius.  Patch with a partition of unity
whose derivatives are independent of \(Z_0\), and absorb the lower-order
commutator terms by interpolation.  This gives, for
every \(\mathbf v\) supported in the interior of the fixed translated
cylinder \(Q=\mathcal C_L(Z_0)\),
\begin{equation}\label{eq:round-DN-Sobolev}
 \mathcal X_m(\mathbf v;Q)
 \leq C_{m,L}\left\{
      \mathcal Y_m(\mathcal L_Z\mathbf v;Q)+\|\mathbf v\|_{L^2(Q)}
      \right\}.
\end{equation}
Here the first row is measured in \(H^{m+1}\), the second in \(H^m\),
and the two unknowns in \(H^{m+2}\) and \(H^{m+1}\),
respectively, exactly as dictated by
\((s_1,s_2;t_\eta,t_\psi)=(-1,0;2,1)\).
The constants are independent of the translated cylinder because
the ellipticity constant in \eqref{eq:DN-symbol-determinant} and the
required coefficient norms are uniform.  Since
\(\mathcal P_Z=\mathcal L_Z+\mathcal K_Z\), for high-mode pairs
\(\mathbf v\in\operatorname{Ran}P_{\ge2}\times\operatorname{Ran}P_{\ge2}\) estimates
\eqref{eq:round-K-small} and \eqref{eq:round-DN-Sobolev}, with
\(Z_*\) increased until
\(C_{m,L}\varepsilon_{m,L}(Z_*)\leq\tfrac12\), yield
\begin{equation}\label{eq:round-full-compact-estimate}
 \mathcal X_m(\mathbf v;Q)
 \leq C_{m,L}\left\{
      \mathcal Y_m(\mathcal P_Z\mathbf v;Q)+\|\mathbf v\|_{L^2(Q)}
      \right\}.
\end{equation}

\noindent\textbf{Step 4: localization and hole filling.}
We spell out the localization, because a direct use of
\eqref{eq:round-DN-Sobolev} on
$\mathcal L_Z\mathbf w=-\mathcal K_Z\mathbf w$ would put a highest-order norm on a larger
cylinder.  Let \(\mathcal C_\rho\Subset \mathcal C_\sigma\Subset \mathcal C_L\) and choose a scalar cutoff
\(\chi(Z)\) equal to one on \(\mathcal C_\rho\), supported in \(\mathcal C_\sigma\), with
\(|\partial_Z^p\chi|\leq C_p(\sigma-\rho)^{-p}\).  The finite-rank operator is angular
and order zero in $Z$, so it commutes with $\chi$.  Since
$\mathcal P_Z\mathbf w=0$,
\[
\mathcal P_Z(\chi\mathbf w)=[\mathcal L_Z,\chi]\mathbf w.
\]
Applying \eqref{eq:round-full-compact-estimate} to \(\chi\mathbf w\) gives
\[
\begin{aligned}
\mathcal X_m(\mathbf w;\mathcal C_\rho)
&\leq \mathcal X_m(\chi\mathbf w;\mathcal C_\sigma)\\
&\leq C\left(
\mathcal Y_m([\mathcal L_Z,\chi]\mathbf w;\mathcal C_\sigma)
 +\|\mathbf w\|_{L^2(\mathcal C_\sigma)}\right).
\end{aligned}
\]
The commutator has strictly lower Douglis--Nirenberg order.  Consequently,
for some integers $M_m,M'_m$ and every $\delta>0$, Sobolev interpolation
gives
\begin{align*}
\mathcal Y_m([\mathcal L_Z,\chi]\mathbf w;\mathcal C_\sigma)
&\leq C(\sigma-\rho)^{-M_m}\mathcal X_m^-(\mathbf w;\mathcal C_\sigma)\\
&\leq C(\sigma-\rho)^{-M_m}
\bigl(\delta \mathcal X_m(\mathbf w;\mathcal C_\sigma)
 +C\delta^{-M'_m}\|\mathbf w\|_{L^2(\mathcal C_\sigma)}\bigr),
\end{align*}
where
\[
\mathcal X_m^-(\mathbf w;Q):=\|\eta\|_{H^{m+1}(Q)}+\|\psi\|_{H^m(Q)}.
\]
Choose
$\delta$ to be a fixed small multiple of
\(\vartheta(\sigma-\rho)^{M_m}\).  We obtain, for any prescribed
$0<\vartheta<1$, after increasing the tail threshold,
\begin{equation}\label{eq:round-hole-filling}
\mathcal X_m(\mathbf w;\mathcal C_\rho)\leq\vartheta \mathcal X_m(\mathbf w;\mathcal C_\sigma)
+C_{m,L,\vartheta}(\sigma-\rho)^{-N_m}\|\mathbf w\|_{L^2(\mathcal C_L)}.
\end{equation}
Choose \(\vartheta\) so that \(2^{N_m}\vartheta<1\), set
\[
r_j=L-1+\tfrac12(1-2^{-j}),\qquad
A_j=\mathcal X_m(\mathbf w;\mathcal C_{r_j}),\qquad \mathcal N=\|\mathbf w\|_{L^2(\mathcal C_L)}.
\]
Since \(r_{j+1}-r_j=2^{-j-2}\), \eqref{eq:round-hole-filling} gives,
after absorbing a fixed power of \(2\) into \(C\),
\[
A_j\leq\vartheta A_{j+1}+C2^{N_mj}\mathcal N.
\]
Consequently
\[
A_0\leq\vartheta^JA_J
+C\mathcal N\sum_{j=0}^{J-1}(2^{N_m}\vartheta)^j.
\]
Smoothness of \(\mathbf w\) on \(\mathcal C_{L-1/2}\) gives
\(\vartheta^JA_J\to0\); letting \(J\to\infty\) therefore yields
\[
\mathcal X_m(\mathbf w;\mathcal C_{L-1})\leq C_{m,L}\|\mathbf w\|_{L^2(\mathcal C_L)}.
\]
The cylinder has dimension \(k+1\).  For the prescribed \(j\), choose
\(m\) so that \(m+2-(k+1)/2>j\).  Sobolev embedding, with the estimate
also applied one index higher for \(\psi\), proves
\eqref{eq:round-DN-Cj}; the threshold \(Z_{j,L}\) dominates the
coefficient, ellipticity, and perturbation-smallness thresholds, while
\(C_{j,L}\) is independent of \(Z_0\).
\end{proof}

The spectral gap upgrades translated-cylinder convergence to exponential
angular decay.

\begin{proposition}
\label{prop:round-angular-decay}
There is $\gamma>0$ such that, for every $j\geq0$,
\begin{equation}\label{eq:round-angular-decay}
\|q(Z,\cdot)-\bar q(Z)\|_{C^j(\mathbb S^k)}
+\|\phi(Z,\cdot)-\bar\phi(Z)\|_{C^j(\mathbb S^k)}
\leq C_j e^{-\gamma Z}.
\end{equation}
\end{proposition}

\begin{proof}
\noindent\textbf{Step 1: slow-maximum compactness.}
We first treat $\mathbf w=(\eta,\psi)$.  Set
\[
A(Z)=\|\mathbf w\|_{L^2([Z-1,Z+1]\times\mathbb S^k)}.
\]
By \eqref{eq:round-full-limit}, $A(Z)\to0$.  Suppose that no exponential
upper bound holds.  If $A$ vanishes identically on a terminal half-line,
then
\[
\eta=\psi=0,\qquad b=\beta(\bar q,\bar\phi,0,0)=0,\qquad
P_1(\tan\phi)=P_1(\tan\bar\phi)=0.
\]
Equation \eqref{eq:round-a-ode} reduces to
\((Z^{-1/2}a)'=0\); since \(a(Z)\to0\), also \(a=0\) there, and the
conclusion is immediate.  We exclude this case.  If no exponential upper bound holds, choose
$\epsilon_i\downarrow0$ and $Y_i\to\infty$ so that
$e^{\epsilon_iY_i}A(Y_i)$ escapes every fixed bound.  Since $A(Z)\to0$,
the function $A(Z)e^{-\epsilon_i|Z-Y_i|}$ attains a maximum on the tail;
choose a maximizer $Z_i$.  The choice of $Y_i$ forces $Z_i\to\infty$, and
the maximizing property gives
\begin{equation}\label{eq:round-slow-selection}
 A(Z)\le A(Z_i)e^{\epsilon_i|Z-Z_i|}
 \qquad (Z\ge Z_0).
\end{equation}
Indeed, one may maximize
$A(Z)e^{-\epsilon_i|Z-Y_i|}$ with $Y_i\to\infty$ chosen so that the
maximum cannot occur on any fixed initial interval.  This gives
$Z_i\to\infty$ and \eqref{eq:round-slow-selection}.

Normalize \(\mathbf w_i(\tau,\omega)=\mathbf w(Z_i+\tau,\omega)/A(Z_i)\).  Then
\[
\begin{aligned}
\|\mathbf w_i\|_{L^2([-1,1]\times\mathbb S^k)}&=1,\\
\|\mathbf w_i\|_{L^2([\tau-1,\tau+1]\times\mathbb S^k)}
&=\frac{A(Z_i+\tau)}{A(Z_i)}
\leq e^{\epsilon_i|\tau|}.
\end{aligned}
\]
On each fixed compact
cylinder, \eqref{eq:round-slow-selection} and
Lemma~\ref{lem:round-DN-estimate} give uniform bounds in every $C^j$.  A diagonal
subsequence converges smoothly on compact sets.  The translated coefficients
converge to the radial ones and \(\mathcal K_Z\to0\) by
\eqref{eq:round-K-small}, so the limit \(\mathbf w_\infty\) solves
\eqref{eq:round-linear-limit} on
\(\mathbb R\times\mathbb S^k\).  Smooth convergence and the normalized
identities give
\[
\|\mathbf w_\infty\|_{L^2([-1,1]\times\mathbb S^k)}=1,\qquad
\|\mathbf w_\infty\|_{L^2([\tau-1,\tau+1]\times\mathbb S^k)}\leq1.
\]
Covering the cylinders used in Lemma~\ref{lem:round-DN-estimate} by
finitely many unit slabs gives, for every integer \(r\geq0\),
\[
\sup_{\tau\in\mathbb R}\left(
\|\eta_\infty\|_{C^r(\mathcal C_1(\tau))}
+\|\psi_\infty\|_{C^r(\mathcal C_1(\tau))}\right)\leq C_r.
\]
Thus \(\mathbf w_\infty\) is bounded on the whole cylinder.
It contains only harmonics of degree at least two.

\noindent\textbf{Step 2: spectral exclusion of the high modes.}
Let \(f_\ell(\tau)\) and \(g_\ell(\tau)\) be the corresponding degree-\(\ell\)
coefficients of the limiting fields \(\eta_\infty\) and \(\psi_\infty\).  Since
\(-\Delta Y_\ell=\lambda_\ell Y_\ell\), where
\(\lambda_\ell=\ell(\ell+k-1)\), equation
\eqref{eq:round-linear-limit} gives
\(g_\ell=-f_\ell'\) and
\[
f_\ell''-\frac{\lambda_\ell-k}{k^2}f_\ell
=f_\ell''-\mu_\ell^2f_\ell=0,\qquad
\mu_\ell:=\frac{\sqrt{(\ell-1)(\ell+k)}}{k}>0.
\]
Thus
\[
f_\ell(\tau)=A_\ell e^{\mu_\ell\tau}+B_\ell e^{-\mu_\ell\tau}.
\]
Boundedness as \(\tau\to+\infty\) gives \(A_\ell=0\), while boundedness as
\(\tau\to-\infty\) gives \(B_\ell=0\).  Hence \(\mathbf w_\infty=0\), contradicting
its unit norm.  Therefore
$A(Z)\leq Ce^{-\gamma Z}$ for some $\gamma>0$, and
\[
\|\mathbf w\|_{L^2(\mathcal C_L(Z))}\leq C_Le^{-\gamma Z}.
\]
Applying Lemma~\ref{lem:round-DN-estimate} with index \(j+1\) yields
\[
\|\eta(Z,\cdot)\|_{C^j}
+\|\psi(Z,\cdot)\|_{C^j}\leq C_je^{-\gamma Z}.
\]
Equation \eqref{eq:round-beta-quadratic} gives exponential
decay of $b$.

\noindent\textbf{Step 3: recovery of the degree-one mode.}
It remains to control the degree-one part of $q$.  Projecting
\eqref{eq:round-q} onto degree one gives
\begin{equation}\label{eq:round-a-ode}
a'(Z)-\frac{a(Z)}{2Z}=-\mathbf f_1(Z),\qquad
\mathbf f_1(Z)\cdot\omega=P_1(\tan\phi).
\end{equation}
Put \(\zeta=b\cdot\omega+\psi\).  Taylor's formula gives
\[
\tan(\bar\phi+\zeta)
=\tan\bar\phi+\sec^2\bar\phi\,\zeta+O(\zeta^2).
\]
Since \(P_1(\tan\bar\phi)=0\) and \(P_1\psi=0\),
\[
|\mathbf f_1(Z)|
\leq C\bigl(|b(Z)|+\|\psi(Z,\cdot)\|_{C^0}^2
 +|b(Z)|\|\psi(Z,\cdot)\|_{C^0}\bigr)
\leq Ce^{-\gamma Z}.
\]
The integrating-factor calculation is
\[
\begin{aligned}
(Z^{-1/2}a)'&=-Z^{-1/2}\mathbf f_1,\\
Z_1^{-1/2}a(Z_1)-Z^{-1/2}a(Z)
&=-\int_Z^{Z_1}\sigma^{-1/2}\mathbf f_1(\sigma)\,d\sigma.
\end{aligned}
\]
By Lemma~\ref{lem:round-centering}, \(a(Z_1)\to0\).  Letting
\(Z_1\to\infty\) gives
\[
a(Z)=Z^{1/2}\int_Z^\infty \sigma^{-1/2}\mathbf f_1(\sigma)\,d\sigma
\]
and, since \(Z^{1/2}\sigma^{-1/2}\leq1\) for \(\sigma\geq Z\),
\[
|a(Z)|\leq C\int_Z^\infty e^{-\gamma\sigma}\,d\sigma
\leq\frac C\gamma e^{-\gamma Z}.
\]
This completes
\eqref{eq:round-angular-decay}.
\end{proof}

\subsection{Radial dynamics and a coercive barrier}

Let
\[
\bar q(Z)=\frac1{|\mathbb S^k|}\int_{\mathbb S^k}q\,d\omega,\qquad
\bar\phi(Z)=\frac1{|\mathbb S^k|}\int_{\mathbb S^k}\phi\,d\omega.
\]
Write
\[
\langle F\rangle_\omega
:=\frac1{|\mathbb S^k|}\int_{\mathbb S^k}F\,d\omega,\qquad
\zeta_{\rm ang}=\phi-\bar\phi,\qquad C_\eta=\nabla^2\eta+\eta g.
\]
Averaging \eqref{eq:round-q}--\eqref{eq:round-phi} gives the exact system
\begin{align}
\bar q'&=\frac{\bar q}{2Z}-\tan\bar\phi+\varepsilon_q(Z),\label{eq:round-mean-q}\\
\bar\phi'&=\sec\bar\phi-\frac{k}{\bar q}+\varepsilon_\phi(Z),\label{eq:round-mean-phi}\\
|\varepsilon_q(Z)|+|\varepsilon_\phi(Z)|&\leq Ce^{-\gamma Z},\qquad
\bar q\to k,\quad\bar\phi\to0.\label{eq:round-mean-errors}
\end{align}
Here
\[
\begin{aligned}
\varepsilon_q(Z)&=-\bigl\langle\tan(\bar\phi+\zeta_{\rm ang})-\tan\bar\phi\bigr\rangle_\omega,\\
\varepsilon_\phi(Z)&=\bigl\langle\sec(\bar\phi+\zeta_{\rm ang})-\sec\bar\phi\bigr\rangle_\omega\\
&\quad-\left\langle\operatorname{tr}(\bar q g+C_\eta)^{-1}-\frac{k}{\bar q}
\right\rangle_\omega\\
&\quad-\left\langle\sec^2(\bar\phi+\zeta_{\rm ang})
\langle(\bar q g+C_\eta)^{-1}\nabla\zeta_{\rm ang},\nabla\zeta_{\rm ang}\rangle
\right\rangle_\omega.
\end{aligned}
\]
The three cancellations are
\[
\langle\zeta_{\rm ang}\rangle_\omega=0,\qquad B_q=\bar q g+C_\eta,\qquad
\langle\operatorname{tr}_gE\rangle_\omega
=\langle(\Delta+k)\eta\rangle_\omega=0.
\]
The last equality uses \(\langle\eta\rangle_\omega=0\), and the degree-one
term is absent because \(\nabla^2(a\cdot\omega)+(a\cdot\omega)g=0\).
Taylor expansion gives
\[
\begin{aligned}
\tan(\bar\phi+\zeta_{\rm ang})
 &=\tan\bar\phi+\sec^2\bar\phi\,\zeta_{\rm ang}+O(\zeta_{\rm ang}^2),\\
\sec(\bar\phi+\zeta_{\rm ang})
 &=\sec\bar\phi+\sec\bar\phi\tan\bar\phi\,\zeta_{\rm ang}+O(\zeta_{\rm ang}^2),\\
\operatorname{tr}(RI+C_\eta)^{-1}
 &=\frac{k}{\bar q}-\frac1{\bar q^2}\operatorname{tr}C_\eta+O(|C_\eta|^2).
\end{aligned}
\]
All displayed linear terms vanish after averaging.  Uniform positivity of
\(B_q\) also gives
\[
\frac1{|\mathbb S^k|}\int_{\mathbb S^k}\sec^2\phi
\langle B_q^{-1}\nabla\phi,\nabla\phi\rangle\,d\omega
\leq C\|\nabla\zeta_{\rm ang}\|_{C^0}^2.
\]
Together with Proposition~\ref{prop:round-angular-decay}, this yields
\[
|\varepsilon_q(Z)|+|\varepsilon_\phi(Z)|
\leq C\left(\|\zeta_{\rm ang}\|_{C^0}^2+\|C_\eta\|_{C^0}^2
 +\|\nabla\zeta_{\rm ang}\|_{C^0}^2\right)
\leq Ce^{-2\gamma Z}.
\]
Relabeling the exponent
gives \eqref{eq:round-mean-errors}; moreover,
$\operatorname{tr}((\bar q g)^{-1})=k/\bar q$ gives the term $k/\bar q$ in
\eqref{eq:round-mean-phi}.

Define
\begin{equation}\label{eq:round-approx-pair}
q_*(Z)=k-\frac{k^2(k+4)}{8Z^2},\qquad
\phi_*(Z)=\frac{k}{2Z}
-\frac{k^2(17k+60)}{48Z^3}.
\end{equation}
A direct Taylor substitution in \eqref{eq:round-mean-q}--\eqref{eq:round-mean-phi}
gives
\begin{equation}\label{eq:round-approx-residual}
q_*'-\frac{q_*}{2Z}+\tan\phi_*=O(Z^{-5}),\qquad
\phi_*'-\sec\phi_*+\frac{k}{q_*}=O(Z^{-4}).
\end{equation}

The approximate radial pair controls the true mean modes as follows.

\begin{lemma}\label{lem:round-radial-error}
One has
\begin{equation}\label{eq:round-radial-error}
|\bar q(Z)-q_*(Z)|+|\bar\phi(Z)-\phi_*(Z)|=O(Z^{-3}).
\end{equation}
\end{lemma}

\begin{proof}
\noindent\textbf{Step 1: conservation and quadratic coercivity.}
Put $\rho=\bar q-q_*$, $\vartheta=\bar\phi-\phi_*$, and
$d=(\rho^2+\vartheta^2)^{1/2}$.  The autonomous system
\begin{equation}\label{eq:round-autonomous}
\rho'=-\tan\vartheta,\qquad
\vartheta'=\sec\vartheta-\frac{k}{k+\rho}
\end{equation}
preserves
\begin{equation}\label{eq:round-Hamiltonian}
H_k(\rho,\vartheta)
=(k+\rho)^k\cos\vartheta
-\frac{(k+\rho)^{k+1}}{k+1}.
\end{equation}
Indeed,
\[
\begin{aligned}
\partial_\rho H_k
&=(k+\rho)^{k-1}(k\cos\vartheta-k-\rho),\\
\partial_\vartheta H_k
&=-(k+\rho)^k\sin\vartheta,\\
\frac d{dZ}H_k
&=\partial_\rho H_k(-\tan\vartheta)
 +\partial_\vartheta H_k
\left(\sec\vartheta-\frac{k}{k+\rho}\right)=0.
\end{aligned}
\]
Taylor expansion at the origin gives
\begin{equation}\label{eq:round-H-quadratic}
H_k(0,0)-H_k(\rho,\vartheta)
=\frac{k^{k-1}}2\rho^2+\frac{k^k}2\vartheta^2+O(d^3).
\end{equation}

\noindent\textbf{Step 2: exact perturbation and dissipation.}
Subtracting \eqref{eq:round-approx-residual} from
\eqref{eq:round-mean-q}--\eqref{eq:round-mean-phi}, and using the addition formulas
for $\tan$ and $\sec$, writes the actual error equations as
\begin{align*}
\rho'&=-\tan\vartheta+\epsilon_1,\\
\vartheta'&=\sec\vartheta-\frac{k}{k+\rho}+\epsilon_2,
\end{align*}
where, on putting
\[
\begin{aligned}
r_1&=q_*'-\frac{q_*}{2Z}+\tan\phi_*=O(Z^{-5}),\\
r_2&=\phi_*'-\sec\phi_*+\frac{k}{q_*}=O(Z^{-4}),
\end{aligned}
\]
the errors are exactly
\[
\begin{aligned}
\epsilon_1
&=\frac{\rho}{2Z}
-[\tan(\phi_*+\vartheta)-\tan\phi_*-\tan\vartheta]
+\varepsilon_q-r_1,\\
\epsilon_2
&=[\sec(\phi_*+\vartheta)-\sec\phi_*-\sec\vartheta+1]\\
&\quad+\left[-\frac{k}{q_*+\rho}+\frac{k}{q_*}
+\frac{k}{k+\rho}-1\right]+\varepsilon_\phi-r_2.
\end{aligned}
\]
For large $Z$ and small $d$, Taylor expansion gives
\begin{align}
\epsilon_1&=\frac{\rho}{2Z}
+O\left(\frac d{Z^2}+\frac{d^2}{Z}+Z^{-5}+e^{-\gamma Z}\right),
\label{eq:round-epsilon1}\\
\epsilon_2&=\frac{k\vartheta}{2Z}
+O\left(\frac d{Z^2}+\frac{d^2}{Z}+Z^{-4}+e^{-\gamma Z}\right).
\label{eq:round-epsilon2}
\end{align}
For example,
$\tan(\phi_*+\vartheta)-\tan\phi_*-\tan\vartheta
=O(d/Z^2+d^2/Z)$, while
$\sec(\phi_*+\vartheta)-\sec\phi_*-\sec\vartheta+1
=k\vartheta/(2Z)+O(d/Z^2+d^2/Z)$; the reciprocal-radius
difference has the same stated remainder.

The autonomous contributions cancel, so
\[
H_k'=\partial_\rho H_k\,\epsilon_1
+\partial_\vartheta H_k\,\epsilon_2.
\]
Moreover,
\[
\partial_\rho H_k=-k^{k-1}\rho+O(d^2),\qquad
\partial_\vartheta H_k=-k^k\vartheta+O(d^2).
\]
and the leading terms are
\[
-\frac{k^{k-1}}{2Z}\rho^2
-\frac{k^{k+1}}{2Z}\vartheta^2.
\]
Thus, using \(e^{-\gamma Z}\leq CZ^{-4}\),
\[
H_k'\leq-\frac{c_1}{Z}d^2
+C\left(\frac{d^2}{Z^2}+\frac{d^3}{Z}+\frac d{Z^4}\right).
\]
Choose the tail so that \(C/Z+Cd\leq c_1/2\).  Then there are \(c,C>0\)
such that
\begin{equation}\label{eq:round-H-dissipation}
\frac{d}{dZ}H_k(\rho,\vartheta)
\leq-\frac cZ d^2+\frac C{Z^4}d
\end{equation}
for all sufficiently large $Z$.

\noindent\textbf{Step 3: the coercive barrier.}
Let $D(Z)=H_k(0,0)-H_k(\rho(Z),\vartheta(Z))$.  By
\eqref{eq:round-H-quadratic}, after increasing the initial value of $Z$,
there are $c_0,C_0>0$ such that
\[
c_0d^2\leq D\leq C_0d^2,\qquad D(Z)\longrightarrow0.
\]
Since $D'=-H_k'$, \eqref{eq:round-H-dissipation} gives
\[
D'\geq\frac cZ d^2-\frac C{Z^4}d
=\frac dZ(cd-CZ^{-3}).
\]
Choose \(K\) so large that \(c\sqrt{K/C_0}\geq2C\).  Then
\[
\begin{aligned}
D\geq KZ^{-6}
&\Longrightarrow d\geq\sqrt{K/C_0}\,Z^{-3}\\
&\Longrightarrow D'\geq\frac c{2Z}d^2>0.
\end{aligned}
\]
Set \(\mathcal B(Z)=D(Z)-KZ^{-6}\).  On \(\{\mathcal B\geq0\}\),
\[
\mathcal B'=D'+6KZ^{-7}>0.
\]
If \(\mathcal B(Z_1)\geq0\) at a tail point, then \(\mathcal B\) remains positive and
increasing thereafter; in particular \(D\) is positive and increasing,
contrary to \(D(Z)\to0\).  Hence
\[
D(Z)\leq KZ^{-6},\qquad
d(Z)^2\leq c_0^{-1}D(Z)\leq\frac{K}{c_0}Z^{-6}.
\]
Together with \eqref{eq:round-H-quadratic}, this proves
\eqref{eq:round-radial-error}.
\end{proof}

Combining Proposition~\ref{prop:round-angular-decay} and Lemma~\ref{lem:round-radial-error} gives
\begin{equation}\label{eq:round-q-expansion}
q(Z,\omega)=k-\frac{k^2(k+4)}{8Z^2}+O_{C^\infty(\mathbb S^k)}(Z^{-3}).
\end{equation}
The $Z^{-3}$ term of $\phi_*$ in \eqref{eq:round-approx-pair} was chosen to improve the first
residual; \eqref{eq:round-radial-error} does not identify that coefficient
as an asymptotic coefficient of the actual $\bar\phi$.

\subsection{Return from Gauss coordinates to Euclidean polar coordinates}

Put \(A=k^2(k+4)/8\).  Since \(Z=t^2/2\) and \(h=q/t\),
\eqref{eq:round-q-expansion} gives
\[
\begin{aligned}
h(t,\omega)=\frac{q(Z,\omega)}t
&=\frac kt-\frac{A}{tZ^2}
 +O_{C^\infty}(t^{-1}Z^{-3})\\
&=\frac kt-\frac{k^2(k+4)}{2t^5}
 +O_{C^\infty}(t^{-7}),
\end{aligned}
\]
which is
\begin{equation}\label{eq:round-h-expansion}
h(t,\omega)=\frac{k}{t}-\frac{k^2(k+4)}{2t^5}
+O_{C^\infty(\mathbb S^k)}(t^{-7}).
\end{equation}
The radial error has no angular derivative, while
Proposition~\ref{prop:round-angular-decay} gives
\[
\nabla_\omega h=t^{-1}\nabla_\omega(q-\bar q)
=O_{C^\infty}(t^{-1}e^{-\gamma t^2/2}),
\]
that is,
\begin{equation}\label{eq:round-h-angular}
\nabla h=O_{C^\infty}(t^{-1}e^{-\gamma t^2/2}).
\end{equation}
For \(x=X(t,\omega)\), \eqref{eq:round-support-map} gives
\[
|x|^2=h^2+|\nabla h|^2.
\]
Since \(h\sim k/t\),
\[
\begin{aligned}
r-h
&=\sqrt{h^2+|\nabla h|^2}-h\\
&=\frac{|\nabla h|^2}
 {\sqrt{h^2+|\nabla h|^2}+h}
=O_{C^\infty}(t^{-1}e^{-\gamma t^2}).
\end{aligned}
\]
This is smaller than every algebraic remainder, so
\begin{equation}\label{eq:round-radius-t}
r:=|x|=\frac{k}{t}-\frac{k^2(k+4)}{2t^5}+O(t^{-7}),
\end{equation}
uniformly with every angular derivative.

Define the Euclidean polar-direction map
\[
\mathcal P_t(\omega):=\frac{X(t,\omega)}{|X(t,\omega)|}.
\]
Then
\[
\mathcal P_t(\omega)
=\frac{h\omega+\nabla h}{(h^2+|\nabla h|^2)^{1/2}}
=\omega+O_{C^\infty}(e^{-\gamma t^2/2}).
\]
Equivalently, for every fixed \(m\),
\[
\|\mathcal P_t-\operatorname{id}\|_{C^m(\mathbb S^k)}
\leq C_me^{-\gamma t^2/2}.
\]
For large \(t\), \(D\mathcal P_t\) is invertible, so $\mathcal P_t$ is a
local diffeomorphism.  Because $\mathbb S^k$ is compact, the map is proper;
its image is therefore both open and closed in the connected sphere and is
nonempty, hence it is onto.  Thus $\mathcal P_t$ is a covering map.  Since
it is homotopic to the identity, its degree is one, so it is a global
diffeomorphism.  The quantitative inverse
function theorem gives, for every fixed \(m\),
\[
\|\mathcal P_t^{-1}-\operatorname{id}\|_{C^m(\mathbb S^k)}
\leq C_me^{-\gamma t^2/2}.
\]
Composing \eqref{eq:round-radius-t} with this inverse therefore preserves
all stated orders.

For the remaining implicit inversion we do not differentiate the
\(O(t^{-7})\) remainder in \eqref{eq:round-radius-t}.  Instead,
\eqref{eq:round-normal-velocity}, the radial estimate, and the angular
decay give directly
\[
\begin{aligned}
h_t=-\tan\phi
&=-\frac{k}{2Z}+O(Z^{-3})+O(e^{-\gamma Z})
 =-\frac{k}{t^2}+O(t^{-6}),\\
\nabla h_t&=O_{C^\infty}(e^{-\gamma t^2/2}).
\end{aligned}
\]
Since \(X_t=h_t\omega+\nabla h_t\) and
\(\mathcal P_t=X/|X|=\omega+O_{C^\infty}(e^{-\gamma t^2/2})\),
\[
\partial_t|X(t,\omega)|
=\mathcal P_t(\omega)\cdot X_t
=-\frac{k}{t^2}+O(t^{-6}).
\]
The mixed \(Z\)--angular estimates also give
\(\partial_t\mathcal P_t=O_{C^\infty}(t e^{-\gamma t^2/2})\);
therefore composing with \(\mathcal P_t^{-1}\) changes the last formula
only by an exponentially small term.  Thus, at fixed Euclidean polar
direction,
\[
\partial_t r=-\frac{k}{t^2}+O(t^{-6})<0,\qquad
D_\theta t=-\frac{D_\theta r}{\partial_t r}.
\]
Since \(D_\theta r\) is exponentially small while
\(\partial_t r=-k t^{-2}+O(t^{-6})\) stays uniformly away from zero after
rescaling by $t^2$, the first implicit derivative is exponentially small.
Differentiating the identity $r(t(r,\theta),\theta)=r$ repeatedly and
arguing inductively, using the mixed $t$--angular bounds above, shows that
every higher angular derivative of $t=t(r,\theta)$ is likewise negligible
relative to the algebraic remainder.

Putting $s=t^{-1}$, \eqref{eq:round-radius-t} reads
\[
 r=ks-\frac{k^2(k+4)}2s^5+O(s^7).
\]
Elementary series inversion gives
\[
 s=\frac rk+\frac{k+4}{2k^4}r^5+O(r^7),\qquad
 t=\frac kr-\frac{k+4}{2k^2}r^3+O(r^5),
\]
so
\begin{equation}\label{eq:round-final-expansion}
t=\frac{k}{r}-\frac{k+4}{2k^2}r^3+O(r^5).
\end{equation}
The high-level foliation assigns to every sufficiently small \(x\ne0\)
the unique leaf \(t=u(x)\), and both coordinate inverses above are uniform
on that foliation.  Hence \eqref{eq:round-final-expansion} holds pointwise
as \(x\to0\), not merely subsequentially or levelwise, with the claimed
uniform angular-derivative bounds.  With $k=n-1$, the second
coefficient is $-(n+3)/(2(n-1)^2)$.

\section{Bounded singularities and assembly of the classification}
\label{sec:bounded}
\begin{proposition}\label{prop:bounded-complete}
If $u$ is bounded in a punctured neighborhood of the origin, then it has a
unique $W^{1,1}_{\rm loc}(B_R)$ weak extension $\widetilde u$ across the
origin, unique up to equality almost everywhere, which solves
\eqref{eq:pde} distributionally.
\end{proposition}
\begin{proof}
Choose $B_{2r}\Subset B_R$ and $M$ with $|u|\le M$ on the punctured ball.
The finite-height estimate, applied to an interval containing the range,
gives
\[
 \int_{B_\rho}\sqrt{1+|Du|^2}\,dx\le C\rho^{n-1},
\]
so $Du\in L^1(B_{2r})$.  Integrating by parts on
$B_{2r}\setminus B_\varepsilon$, the inner boundary term for the
distributional derivatives is $O(M\varepsilon^{n-1})$, and the equation
flux is $O(\varepsilon^{n-1})$ because
$|Du|/\sqrt{1+|Du|^2}\le1$.  Hence the a.e.\ extension lies in
$W^{1,1}$ and satisfies the weak equation on the full ball.

Thus the almost-everywhere extension already belongs to
$W^{1,1}(B_{2r})$ and satisfies the capillary equation distributionally on
the full ball.  Repeating this argument on nested balls patches the local
extensions to an element of $W^{1,1}_{\rm loc}(B_R)$.  Any two such weak
extensions agree almost everywhere on $B_R\setminus\{0\}$ and hence almost
everywhere on $B_R$, which proves uniqueness in the weak class.

\end{proof}

\begin{proposition}[Continuous removability in dimensions $2\le n\le7$]
\label{prop:bounded-continuous-lowdim}
Assume $2\le n\le7$ and that $u$ is bounded in a punctured neighborhood of
the origin.  Then $u$ has a unique continuous extension $\widetilde u$ across
the origin.  This extension belongs to, and represents, the weak
$W^{1,1}_{\mathrm{loc}}(B_R)$ solution class constructed in
Proposition~\ref{prop:bounded-complete}.  Equivalently,
\[
 \lim_{x\to0}u(x)
\]
exists and is finite.
\end{proposition}
\begin{proof}
It is enough to work in a ball $B_{2r}\Subset B_R$ on which $|u|\le M$.
By Proposition~\ref{prop:bounded-complete}, after changing the value at the origin if
necessary, $u\in W^{1,1}(B_{2r})$ and
$\operatorname{div}\mathcal A(Du)=-u$ distributionally on $B_{2r}$.

\smallskip
\noindent\textit{Step 1: a calibrated perimeter quasiminimizer.}
Let $\Omega\Subset B_{2r}$ be a ball centered at the origin and consider
the restriction to $\Omega\times\mathbb R$ of the ambient subgraph $E$
defined in the introduction (equivalently, the subgraph of the weak
representative just constructed).  Since $u\in W^{1,1}_{\rm loc}$, $E$
has locally finite perimeter.  Put
\[
 W:=\sqrt{1+|Du|^2},
 \qquad
 \mathbf N_u(x,\tau):=\left(-\frac{Du(x)}{W(x)},\frac1{W(x)}\right).
\]
The subgraph formula gives
\[
 \mathbf N_u=\nu_E
 \qquad |D\chi_E|\text{-a.e.},
\]
and, in distributions,
\[
 |\mathbf N_u|=1,
 \qquad
 \operatorname{div}_{x,\tau}\mathbf N_u=-\operatorname{div}_x\mathcal A(Du)=u(x).
\]
To use this identity without assigning traces of $\mathbf N_u$ on arbitrary
competitors, mollify only in the base variable:
$\mathbf N_{u,\varepsilon}=\mathbf N_u*_{x}\varrho_\varepsilon$.  On every compactly contained cylinder,
\[
 |\mathbf N_{u,\varepsilon}|\le1,
 \qquad
 \operatorname{div}\mathbf N_{u,\varepsilon}=u_\varepsilon,
 \qquad
 |u_\varepsilon|\le M.
\]
Let $\mathcal O\Subset\Omega\times\mathbb R$ and let
$\widetilde E$ be a set of locally finite perimeter with
$E\triangle\widetilde E\Subset\mathcal O$.  Applying Gauss--Green with a
cutoff equal to one on a neighborhood of $E\triangle\widetilde E$ gives
\[
 \int_{\partial^*E\cap\mathcal O} \mathbf N_{u,\varepsilon}\cdot\nu_E\,d\mathcal H^n
 -\int_{\partial^*\widetilde E\cap\mathcal O} \mathbf N_{u,\varepsilon}\cdot\nu_{\widetilde E}\,d\mathcal H^n
 =\int_{\mathcal O}(\chi_E-\chi_{\widetilde E})u_\varepsilon\,dX.
\]
Away from the vertical axis $\{0\}\times\mathbb R$, the original graph is
smooth and $\mathbf N_{u,\varepsilon}\to\mathbf N_u$.  Moreover
$|D\chi_E|(\{0\}\times\mathbb R)=0$, since $n\ge2$ and the axis has
$\mathcal H^n$-measure zero.  Dominated convergence therefore yields
\[
 \int_{\partial^*E\cap\mathcal O}\mathbf N_{u,\varepsilon}\cdot\nu_E\,d\mathcal H^n
 \longrightarrow P(E,\mathcal O).
\]
Since $|\mathbf N_{u,\varepsilon}|\le1$,
\[
 \limsup_{\varepsilon\downarrow0}
 \int_{\partial^*\widetilde E\cap\mathcal O}\mathbf N_{u,\varepsilon}\cdot\nu_{\widetilde E}\,d\mathcal H^n
 \le P(\widetilde E,\mathcal O),
\]
whereas the right-hand side is bounded in absolute value by
$M|E\triangle\widetilde E|$.  Hence
\begin{equation}\label{eq:bounded-quasiminimizer}
 P(E,\mathcal O)\le P(\widetilde E,\mathcal O)+M|E\triangle\widetilde E|.
\end{equation}
Thus $E$ is a local perimeter quasiminimizer with bounded mean-curvature
scale.  More precisely, after shrinking the working radius if necessary, fix
$r_0>0$ with
\[
 Mr_0\le1.
\]
Then \eqref{eq:bounded-quasiminimizer}, restricted to competitors supported in
balls of radius less than $r_0$, is the standard $(M,r_0)$ perimeter
quasiminimality condition.  Replace $E$ by its canonical measure-theoretic representative.  The density estimates of
\cite[Theorem~\refnum{21.11}]{Maggi2012} identify its topological boundary locally with
$\operatorname{spt}|D\chi_E|$; we continue to denote this boundary by
$\Sigma_E$.  The reduced boundary is $C^{1,\beta}$ by
\cite[Theorem~\refnum{26.5}]{Maggi2012}.  The singular-set theorem
\cite[Theorem~\refnum{28.1}]{Maggi2012} gives no singular points when the ambient
dimension $N=n+1$ is at most seven and a locally discrete singular set
when $N=8$; tangent blow-ups at singular points are perimeter-minimizing
cones by \cite[Theorem~\refnum{28.6}]{Maggi2012}.  Consequently $\Sigma_E$ is
everywhere regular when $2\le n\le6$.

\smallskip
\noindent\textit{Step 2: the directed tangent-cone argument when $n=7$.}
Assume now $n=7$ and suppose, for contradiction, that
$p_0\in\Sigma_E$ is singular.  By \cite[Theorem~\refnum{28.6}]{Maggi2012}, after
passing to a blow-up sequence $r_j\downarrow0$ the sets
$E_j=(E-p_0)/r_j$ converge locally in $L^1$ to a perimeter-minimizing cone
$\mathcal C\subset\mathbb R^8$ with vertex at the origin.  Equivalently, the
rescaled quasiminimality errors vanish along the sequence.  Each $E_j$ is a vertical subgraph,
so $D_8\chi_{E_j}\le0$ distributionally.  Thus, for every nonnegative
$\varphi\in C_c^1(\mathbb R^8)$,
\[
 -\int\chi_{\mathcal C}\,\partial_8\varphi
 =\lim_{j\to\infty}
   \left(-\int\chi_{E_j}\,\partial_8\varphi\right)\le0.
\]
Therefore $D_8\chi_{\mathcal C}\le0$.  Since
$D\chi_{\mathcal C}=-\nu_{\mathcal C}|D\chi_{\mathcal C}|$, we obtain
\[
 J_8:=\nu_{\mathcal C}\cdot e_8\ge0
 \qquad |D\chi_{\mathcal C}|\text{-a.e.},
\]
and hence on the regular part of $\partial\mathcal C$.
For a seven-dimensional perimeter-minimizing boundary in $\mathbb R^8$,
the singular set is locally discrete.  Since a nonzero singular point of a
cone would generate a whole singular ray, $\partial\mathcal C\setminus\{0\}$ is
smooth.  Hence
\[
 \Lambda:=\partial\mathcal C\cap\mathbb S^7
\]
is a smooth compact embedded minimal hypersurface, possibly disconnected.
The vertical translation Jacobi field is homogeneous of degree zero, so on
each connected component of $\Lambda$ it satisfies
\[
 \Delta_{\Lambda}J_8+|\mathrm{II}_{\mathcal C}|_{r=1}^2J_8=0.
\]
Because $J_8\ge0$, the strong maximum principle shows that on each
component either $J_8\equiv0$ or $J_8>0$.  In the latter case,
integration over that component gives
\[
 \int_{\Lambda} |\mathrm{II}_{\mathcal C}|_{r=1}^2J_8\,d\mathcal H^6=0,
\]
so $\mathrm{II}_{\mathcal C}\equiv0$ there and that component is a great sphere
$\mathbb S^6\subset\mathbb S^7$.  No other component can be disjoint from
this great sphere: if $f$ denotes the corresponding normal coordinate on
$\mathbb S^7$, then $f$ has a strict sign on any disjoint connected
component $\Lambda'$, whereas minimality in the unit sphere gives
\[
 \Delta_{\Lambda'}f+6f=0,
\]
and integration forces $\int_{\Lambda'}f=0$, a contradiction.  Thus in this
case $\mathcal C$ is a half-space.

It remains to consider the case in which $J_8\equiv0$ on every
component.  The singular point at the cone vertex has zero perimeter
measure, hence the vertical component of the distributional derivative of
$\chi_{\mathcal C}$ vanishes:
\[
 D_8\chi_{\mathcal C}=0.
\]
By Fubini's theorem for distributions,
\[
 \mathcal C=\mathcal K\times\mathbb R
\]
up to a null set, for a cone $\mathcal K\subset\mathbb R^7$.  The minimizing
property of $\mathcal C$ implies that $\mathcal K$ is perimeter minimizing.  Indeed, insert a
compactly supported competitor for $\mathcal K$ into $\mathcal K\times(-T,T)$, close it with
end caps of $O(1)$ perimeter, divide the resulting inequality by $2T$, and
let $T\to\infty$.  The boundary $\partial\mathcal K$ has dimension six, so the
classical regularity theorem for perimeter minimizers has no singular set.
In particular the cone vertex is regular, so the unique tangent cone of
$\mathcal K$ at the origin is a half-space $H$.  But $\mathcal K$ itself is a cone, hence its
tangent cone at the origin is $\mathcal K$ itself.  Therefore $\mathcal K=H$, and so $\mathcal C$ is
a half-space.

The tangent cone furnished by Theorem~\refnum{28.6} at the hypothetical singular
point is therefore a half-space.  By the regularity/singular-set
characterization in \cite[Theorem~\refnum{26.5}]{Maggi2012}, a point admitting such
a flat tangent is regular, a contradiction.  Hence
\[
 \Sigma_E\in C^{1,\beta}_{\rm loc}
\]
also when $n=7$.

\smallskip
\noindent\textit{Step 3: analytic continuation of the complete hypersurface.}
Orient $\Sigma_E$ by the outer normal of the subgraph.  Away from the vertical
axis, $\Sigma_E$ is the classical graph of $u$, and
\[
 H_{\Sigma_E}=-\operatorname{div}\mathcal A(Du)=u=\tau.
\]
The quasiminimizer inequality \eqref{eq:bounded-quasiminimizer} gives
locally bounded generalized mean curvature.  For completeness, let
$\Phi_t$ be the flow of a compactly supported $C^1$ vector field $Y$ and
note that, since $\Sigma_E$ is $C^{1,\beta}$ and $Y$ is compactly supported, the
normal displacement under its flow satisfies
\[
 \frac{|E\triangle\Phi_t(E)|}{|t|}
 \longrightarrow
 \int_{\Sigma_E}|Y\cdot\nu_E|\,d\mathcal H^n.
\]
Applying \eqref{eq:bounded-quasiminimizer} with $F=\Phi_t(E)$ for both signs
of $t$, dividing
by $|t|$, and letting $t\to0$ gives
\[
 |\delta V_{\Sigma_E}(Y)|
 \le M\int_{\Sigma_E} |Y|\,d\mathcal H^n.
\]
Thus the first variation is absolutely continuous with respect to
$\mathcal H^n\lfloor\Sigma_E$ and has an $L^\infty$ generalized mean-curvature
vector.  Since the vertical axis has zero $\mathcal H^n$-measure and the
classical scalar identity $H_{\Sigma_E}=\tau$ holds off that axis, the
first-variation convention used here gives the vector identity
\begin{equation}\label{eq:bounded-pmc-complete}
 \mathbf H_{\Sigma_E}(X)=-X_{n+1}\nu_{\Sigma_E}(X)
 \qquad\mathcal H^n\text{-a.e.\ on }\Sigma_E.
\end{equation}
Now make a rigid motion and write a local side graph as
$\Xi(y)=p_c+\mathcal R(y,\gamma(y))$.  The ambient height is affine in these coordinates,
$\Xi_{n+1}=c_0+b_0\cdot y+a_0\gamma(y)$, so, according to the local orientation, $\gamma$
satisfies weakly
\begin{equation}\label{eq:bounded-sidegraph-pmc}
 \operatorname{div}\frac{D\gamma}{\sqrt{1+|D\gamma|^2}}
 =\varepsilon\bigl(c_0+b_0\cdot y+a_0\gamma\bigr),
 \qquad \varepsilon\in\{-1,1\}.
\end{equation}
Because $\gamma\in C^{1,\beta}$, its gradient is locally bounded, and the
principal matrix
\[
 \mathbb A_{\rm cap}^{ij}(p)=\frac{(1+|p|^2)\delta_{ij}-p_ip_j}
 {(1+|p|^2)^{3/2}}
\]
is uniformly elliptic on each smaller chart.  Difference quotients and
Schauder estimates first give $\gamma\in C^{2,\beta}$ and then, by bootstrap,
$\gamma\in C^\infty$.  Equation \eqref{eq:bounded-sidegraph-pmc} is real
analytic in $(y,\gamma,D\gamma,D^2\gamma)$ and elliptic in the $D^2\gamma$ variables, so analytic
elliptic regularity yields that $\Sigma_E$ is real analytic
\cite{Morrey1958,Morrey1966}.

\smallskip
\noindent\textit{Step 4: an analytic hypersurface cannot terminate in a
vertical segment.}
Set
\[
 u_-:=\liminf_{x\to0}u(x),
 \qquad
 u_+:=\limsup_{x\to0}u(x).
\]
Both numbers are finite.  Because $n\ge2$, every punctured ball is
connected, so the image of $u$ on it is an interval.  The nested
intersection of the closures of these images is therefore exactly the
cluster interval $[u_-,u_+]$.  Every graph point off the axis belongs to
$\Sigma_E$, and $\Sigma_E$ is closed, while a point of the vertical axis with
height strictly above $u_+$ or strictly below $u_-$ has a neighborhood in
which the subgraph is respectively empty or full up to a Lebesgue-null set.
Consequently
\begin{equation}\label{eq:bounded-axis-cluster}
 \Sigma_E\cap(\{0\}\times\mathbb R)
 =\{0\}\times[u_-,u_+].
\end{equation}
Suppose that $u_-<u_+$.  At $p_+=(0,u_+)$, the embedded analytic hypersurface
admits a local real-analytic defining function $\mathcal F_{\rm an}$ such that
\[
 \Sigma_E=\{\mathcal F_{\rm an}=0\},
 \qquad
 D\mathcal F_{\rm an}(p_+)\ne0.
\]
By \eqref{eq:bounded-axis-cluster}, the one-variable analytic function
$\tau\mapsto\mathcal F_{\rm an}(0,\tau)$ vanishes on an interval immediately to the left of
$u_+$.  The analytic identity theorem forces it to vanish also immediately
to the right of $u_+$.  This puts vertical-axis points of height larger than
$u_+$ in $\Sigma_E$, contradicting \eqref{eq:bounded-axis-cluster}.  Hence
$u_-=u_+$, so $u(x)$ has a finite limit as $x\to0$.

Assigning this value at the origin gives a continuous extension
$\widetilde u$ of the punctured solution.  Since changing one point does not
alter the weak $W^{1,1}$ class, $\widetilde u$ represents the extension from
Proposition~\ref{prop:bounded-complete}.  Its uniqueness as a continuous
extension follows because two continuous extensions which agree on the
punctured ball agree at the origin as well.
\end{proof}

\begin{proof}[Proof of Theorem~\ref{thm:main}]
If $u$ is bounded near the puncture, apply
Proposition~\ref{prop:bounded-complete}; when $2\le n\le7$, apply in addition
Proposition~\ref{prop:bounded-continuous-lowdim}.  Otherwise at least one tail is unbounded.
For $n\ge3$, Corollary~\ref{cor:highdim-centering} excludes simultaneous tails and
centers the remaining sign.  For $n=2$, the cluster set is an interval by
Lemma~\ref{lem:cluster-interval}; Theorem~\ref{thm:planar-exceptional-exclusion}
removes the finite half-lines and full cluster, leaving only
$\{+\infty\}$ or $\{-\infty\}$.  Thus in every dimension there is one
centered sign, and \eqref{eq:wholeblowdown-pressure} holds.

The pressure section upgrades this measure-theoretic convergence to a
smooth one-sheeted strictly convex foliation of a full punctured
neighborhood.  The Gauss-coordinate analysis then gives
\eqref{eq:positive-expansion-main}; the negative case follows by the exact
symmetry $u\mapsto-u$.  The compact-annulus boundedness of the fixed
solution gives the stated $r$-dependent localization of high and low levels.
The same argument localizes the pressure bands: for fixed $L>0$, choose
$0<\delta<\min\{r,r_0\}$ and set
\[
 M^{\rm abs}_{r,\delta}:=\max_{\overline B_r\setminus B_\delta}|u|.
\]
Once $t-L/t>M^{\rm abs}_{r,\delta}$, the condition
$|u-t|<L/t$ cannot occur in $\overline B_r\setminus B_\delta$.  Thus the
pressure piece formed inside $B_r$ agrees, on every fixed smaller vertical
slab and for all sufficiently large $t$, with the already controlled local
pressure piece near the puncture, and therefore has the same smooth
multiplicity-one cylindrical limit.  The bounded, positive-pole, and
negative-pole alternatives are mutually exclusive by their limiting
behavior, and the preceding argument exhausts the bounded and unbounded
cases; hence exactly one alternative in Theorem~\ref{thm:main} occurs.
\end{proof}
\section*{Acknowledgments}
Xi-Nan Ma thanks Professor Qi Ding of Fudan University for
many helpful discussions over the years.

This work was supported by the National Key R\&D Program of China (2025YFA1017603).

\section*{Declaration of AI use}
During the preparation of this manuscript, the authors used ChatGPT (OpenAI) as an auxiliary tool for language editing, organization of arguments, and preliminary checking of mathematical derivations. All mathematical statements, proofs, references, and conclusions appearing in the manuscript were independently verified by the authors, who take full responsibility for the content.

\end{document}